\documentclass[a4paper,11pt,onecolumn]{amsart}
\usepackage{amsmath}
  \usepackage{paralist}
  \usepackage{graphics} 
  \usepackage{epsfig} 
\usepackage{graphicx}  \usepackage{epstopdf, multicol}

\usepackage[all]{xy}

 \usepackage[colorlinks=true]{hyperref}
\hypersetup{urlcolor=blue, citecolor=red}

\newtheorem{theorem}{Theorem}[section]
\newtheorem{corollary}[theorem]{Corollary}

\newtheorem{lemma}[theorem]{Lemma}
\newtheorem{proposition}[theorem]{Proposition}
\newtheorem{conjecture}[theorem]{Conjecture}

\theoremstyle{definition}
\newtheorem{definition}[theorem]{Definition}
\newtheorem{remark}[theorem]{Remark}

\newtheorem*{conv}{Notation convention}

\title[Flat surface models of ergodic systems] 
      {Infinite type flat surface models of ergodic systems}

\author[Kathryn Lindsey and Rodrigo Trevi\~no]{}


 \email{klindsey@math.uchicago.edu}
 \email{rodrigo@math.nyu.edu}

\begin{document}
\begin{abstract}
We propose a general framework for constructing and describing infinite type flat surfaces of finite area.  Using this method, we characterize the range of dynamical behaviors possible for the vertical translation flows on such flat surfaces.  We prove a sufficient condition for ergodicity of this flow and apply the condition to
several examples.   We present specific examples of infinite type flat surfaces on which the translation flow exhibits dynamical phenomena not realizable by translation flows on finite type flat surfaces.  
\end{abstract}

\maketitle

\centerline{\scshape Kathryn Lindsey }
\medskip
{\footnotesize
 \centerline{Department of Mathematics}
   \centerline{University of Chicago}
   \centerline{Chicago, Illinois 60637, USA}
} 

\medskip

\centerline{\scshape Rodrigo Trevi\~no}
\medskip
{\footnotesize
 \centerline{Courant Institute of Mathematical Sciences}
   \centerline{New York University}
   \centerline{New York, New York 10012, USA}
}

\bigskip


\section{Introduction}
In this paper we develop a way of constructing flat surfaces out of combinatorial objects - namely, weighted, ordered, bi-infinite Bratteli diagrams (these terms are defined in \S \ref{sec:brat}). The motivation is three-fold: to introduce a new method for producing examples of flat surfaces of infinite genus and finite area; to construct specific examples of translation flows on flat surfaces that exhibit interesting dynamical properties heretofore not observed on flat surfaces; and to develop techniques to explore the dynamics of these flows.


\subsection{Statement of results}
\label{subsec:results}
We view a large class of flat surfaces as each consisting of a collection of rectangles whose edge identifications are determined by (possibly infinite) interval exchange transformations coming from cutting and stacking constructions (defined in Section  \S \ref{sec:CAS}).  Generalizations of Bratteli diagrams and adic maps on Bratteli diagrams, which we will call simply ``diagrams," serve as combinatorial descriptions of these flat surfaces. We develop a ``dictionary" (see Table \ref{tab:dictionary}) that translates between the languages of  diagrams and flat surfaces.

By importing tools and techniques from ergodic theory, we can use diagrams to construct flat surfaces whose translation flows exhibit a wide range of dynamical behaviors.  In particular, some infinite type finite area surfaces arising through this technique have dynamical properties that are not possible for translation flows on finite type flat surfaces:

\begin{theorem}
There exist flat surfaces of finite area and infinite genus whose translation flows
\begin{itemize}
\item are mixing (\S \ref{subsec:range}), or 
\item have positive topological entropy (\S \ref{subsubsec:entropy}), or
\item are minimal and have uncountably many ergodic invariant probability measures (\S \ref{subsubsec:pascal}).
\end{itemize}
\end{theorem}
\noindent We note that although we have shown the existence of a mixing translation flow, we do not know of a concrete example.  We conjecture that a specific surface has a mixing translation flow in Conjecture \ref{conj:mixing}.

Not only do diagrams allow us to construct individual flat surfaces with specific properties, the ability to import theorems from ergodic theory to the field of flat surfaces allows us to characterize the wide range of dynamical behaviors possible for flat surfaces. 
By translating a theorem of Krengel through the ``dictionary," we prove the following:

\begin{theorem}
\label{thm:model}
Any finite entropy, finite measure-preserving flow on a standard Lebesgue space is measure-theoretically isomorphic to the translation flow on a flat surface built from at most two rectangles by isometrically identifying intervals in the boundaries of the rectangles.
\end{theorem}

In Theorem \ref{thm:erg}, we prove a criterion for ergodicity of the vertical translation flow on flat surfaces built from diagrams according to our technique.  This result is an application of a criterion in \cite{rodrigo:erg} for the ergodicity for translation flows on flat surfaces of finite area.  The precise statement of Theorem \ref{thm:erg} is fairly technical, so we will state here only a ``non-technical" version.  However, we first introduce some concepts and notation. 

We will describe how, from any bi-infinite Bratteli diagram $\mathcal{B}$ and two weight functions $w^\pm$ (which roughly correspond to invariant measures, see Definition \ref{def:weightfunctiondef}) one can construct a flat surface $S(\mathcal{B},w^\pm)$ (this is detailed in \S \ref{sec:Dictionary}). There is a renormalization operation (described in detail in \S \ref{subsec:renorm}) which is manifested as the shift $\sigma$ of the diagram, which we denote here by $\sigma(\mathcal{B},w^\pm)$.  The map $\sigma$ shifts the indices of the diagram $\mathcal{B}$ by one and rescales the weight functions $w^+$ and $w^-$ by numbers which depend on $w^\pm$. There is an affine, hyperbolic diffeomorphism between the surfaces $S(\sigma(\mathcal{B},w^\pm))$ and $S(\mathcal{B},w^\pm)$ (Proposition \ref{prop:shift}) which relates their geometries. Therefore, given a surface $S(\mathcal{B},w^\pm)$, iterating the shifting operation yields a countable family of surfaces $\{S(\sigma^k(\mathcal{B},w^\pm))\}_k$.

\medskip
\noindent {\bf ``Non-technical" statement of Theorem \ref{thm:erg}.} \emph{Let $S(\mathcal{B},w^\pm)$ be a flat surface obtained from a diagram $\mathcal{B}$ with transverse measures given by $w^\pm$. If there exists a subsequence $k_i\rightarrow \infty$ such that geometry of the surfaces $S(\sigma^{k_i}(\mathcal{B},w^\pm))$ does not ``degenerate too quickly" (as measured in terms of a summability condition of geometric quantities),  then the vertical flow from any surface constructed from $(\mathcal{B},w^\pm)$ is ergodic with respect to the Lebesgue measure.}
\medskip

The full, technical statement is Theorem \ref{thm:erg} in Section \ref{sec:erg}. This criterion is useful both because it allows us to detect ergodicity in flat surfaces and because, by translating it from the field of flat surfaces into the realm of diagrams, it provides a new criterion for the ergodicity of adic transformations. In \S \ref{subsubsec:symmetric}-\ref{subsubsec:explosive} we use this result to prove the ergodicity of examples of infinite interval exchange transformations which are, to the best of our knowledge, outside the purview of other techniques for proving ergodicity. We believe that under additional hypotheses the criterion given by Theorem \ref{thm:erg} can be upgraded to yield to unique ergodicity, and we hope this to be the subject of a future paper.

The spirit of the proof of Theorem  \ref{thm:erg}  is that if the geometry of an evolving surface under the Teichm\"{u}ller deformation (see (\ref{eqn:teich})) can be controlled sufficiently well, then the translation flow on the surface is ergodic with respect to the Lebesgue measure. In this paper, we show that through the process of cutting and stacking (realized as a shift on the diagram) we can keep track of the evolving geometry of the flat surface.  This process can be seen as a generalization of Rauzy-Veech induction (\cite{rauzy:IET, rauzy:CFE, veech:gauss}), a classical tool for studying the dynamics of translation flows on compact flat surfaces.

Well-known examples of flat surfaces of infinite type, including the surface described by Chamanara in \cite{chamanara} and the Arnoux-Yoccoz-Bowman surface investigated in \cite{bowman}, arise via our technique as surfaces associated to relatively simple diagrams. In \S \ref{subsec:josh}, we obtain the ergodicity of the horizontal and vertical translation flows on the Arnoux-Yoccoz-Bowman surface as an application of Corollary \ref{cor:eventually}. In a detailed example illustrating the connections established in the dictionary, in \S \ref{subsubsec:Baker} we show the explicit connection between Chamanara's surfaces, $p$-adic odometers, and certain cutting and stacking transformations.

\subsection{Related work}
The connection between the dynamics of flows on surfaces of infinite genus and suspensions over cutting and stacking maps was previously observed in \cite{AOW}. The authors use suspensions over cutting and stacking maps to realize aperiodic measure-preserving flows, and then ``smooth out" the curvature at ``cone points" to obtain measurably isomorphic $C^{\infty}$ flows on open $2$-manifolds. Our approach is similar to the approach in \cite{AOW}, paying special attention to guarantee that the resulting surface has a flat surface structure. In order to achieve this, however, our construction has significantly more structure: we need the orders given to the Bratteli diagram (see \S \ref{ss:standardBratteliDefinitions}) in order to pin down the topology of the surface and we need another Bratteli diagram (or cutting and stacking transformation), to define the heights of the function under which a special flow is built. This extra structure is useful in at least two ways. First, the construction unambiguously yields a unique flat surface for which the translation flow \emph{in any direction} is defined for almost every point. Secondly, assuming that we have a bi-infinite Bratteli diagram as opposed to a standard Bratteli diagram allows our shift (renormalization) operation to seem more natural since it can be thought of as renormalization dynamics in the space of all bi-infinite diagrams (an idea already considered in \cite{bufetov:limit}, although in less generality). 


In \cite{bufetov:limit}, Bufetov introduced the notion of studying the dynamics of translation flows through their models of asymptotic foliations over Markov compacta (which is the analogue of space of infinite paths on a Bratteli diagram here). What we do in this paper is very similar in spirit, and is a continuation and generalization of the point of view introduced in \cite{bufetov:limit}, but our goals are different: whereas \cite{bufetov:limit} are mostly interested in studying limit distributions of translation flows on compact surfaces and \cite{bufetov:limitVershik, bufetov:measures} is concerned with limit theorems of flows over symbolic systems, we are interested in developing a theory connecting symbolic systems (Vershik automorphisms) obtained from bi-infinite Bratteli diagrams with flat surfaces of finite and infinite genus. It is known that any compact flat surface can be recovered through a Bratteli diagram of a very specific type which is dictated from the Rauzy-Veech induction \cite{GjerdeJohansen}. The flat surfaces that arise via the construction described here are, in most cases, of infinite topological type. Flat surfaces of infinite genus have been the focus of a great deal of recent research activity (e.g. \cite{infinite-step, chamanara, HLT:Ehrenfest, hooper:measures, DHL:wind-tree, bowman, RalstTroub, infinite-staircase, corinna, rodrigo:erg}).  One of the challenges currently faced by researchers is the lack of a general, concrete theory for non-compact flat surfaces (i.e. how to define a moduli space of infinite type flat surfaces, how to describe and parametrize such surfaces, etc.). It is through the dictionary developed here that we can construct new examples and discover new phenomena which is unique to flat surfaces of infinite genus. Moreover, it may be the first steps towards building a theory of flat surfaces of finite area which includes both surfaces of finite and infinite genus. A comparison between the dictionary established here in Table \ref{tab:dictionary} and that of Bufetov's \cite[\S 1.9.2]{bufetov:limit} reflects the similarities of the constructions and differences in aims.

Other criteria for the ergodicity of adic transformations exist (see, for example, \cite{fisher, FFT, BKMS1, BKMS2}), all relying on Perron-Frobenius tools. Consequently, most impose restrictions on the Bratteli diagrams used to define the transformations, requiring, for example, that the Bratteli diagram be stationary or have a uniformly bounded number of vertices in each level.  The criterion in \cite{fisher}, which is the most general of the Perron-Frobenius based criteria, requires certain connectivity restrictions which render it inapplicable to systems such as the Chacon middle third transformation, which we treat in \S \ref{subsubsec:chaconMiddleThird} (see also Remark \ref{rem:primitivity}). One of the strengths of Theorem \ref{thm:erg} is that we do not impose any restrictions on the Bratteli diagram beyond the connectivity which is necessary for ergodicity. Although Theorem \ref{thm:erg} yields only ergodicity (as opposed to unique ergodicity of other methods) it may detect ergodicity of one measure even when a system is not uniquely ergodic. 

\subsection{Questions}
This paper gives rise to several possible directions for future research and we briefly summarize three of them.
\begin{itemize}
 
\item Finite truncations of diagrams corresponded surfaces of finite type, so diagrams hint at a ``moduli space" encompassing translation surfaces of both finite and infinite type, and in which infinite type surfaces are limits of sequences of finite type surfaces.   How could one define a ``moduli space of flat surfaces" that includes surfaces of both finite and infinite type?

\item This paper only considers translation flows in the horizontal and vertical directions.  What is the relationship between a diagram and the dynamics of translation flows in directions other than horizontal and vertical on the associated flat surface?

\item  What role dimension groups play in this geometric interpretation of adic transformations via flat surfaces?  It is known that there is a deep connection between states of dimension groups and certain invariant measures of adic transformations \cite{HermanPutnamSkau}.
\end{itemize}

\subsection{Organization}
The paper is organized as follows.  Section \S \ref{subsubsec:Baker} presents an example which illustrates some central concepts and spirit of the paper.   Sections \S \ref{sec:flat}-\ref{sec:CAS} summarize the necessary background material for flat surfaces, Bratteli diagrams and cutting and stacking transformations, respectively.  We give extensive background in these sections to make the paper self-contained, since the reader which is very familiar with flat surfaces may not be as familiar with Bratteli diagrams, and vice versa. These concepts come together in \S \ref{sec:Dictionary}, where we describe how to construct a unique flat surface from a bi-infinite Bratteli diagram along with two orders and a weight function, and develop a dictionary, which is summarized in Table \ref{tab:dictionary} on page \pageref{tab:dictionary}. Section \S \ref{subsubsec:chaconMiddleThird} covers an examples which illustrates our construction. In \ref{subsec:renorm} we describe the renormalization tool obtained by the interaction between deforming the surface and shifting the Bratteli diagram. Section \S \ref{sec:properties} presents a translation flow realization theorem (Theorem \ref{t:weCanGetAnyFlow}) as well as exhibits examples demonstrating the range of new phenomena appearing for translation flows on surfaces of infinite type. In \S \ref{sec:erg} we state and prove Theorem \ref{thm:erg}, and then apply it to many different systems, some of which have appeared in the literature before, but many of which have not.

\subsection{Acknowledgements}
The authors thank Jon Aaronson, Giovanni Forni, Thierry Giordano, Kostya Medynets,  Cesar Silva, and John Smillie for helpful conversations during the development of this paper. Kathryn Lindsey was partially supported by the NSF under a Graduate Research Fellowship and, later, a Mathematical Sciences Postdoctoral Research Fellowship.  Rodrigo Trevi\~no was partially supported by the NSF under Award No. DMS-1204008, BSF Grant 2010428, and ERC Starting Grant DLGAPS 279893.

\section{Example: Odometers and Chamanara's surfaces} 
\label{subsubsec:Baker}

In this example, we identify Chamanara's surface \cite{chamanara}, one of the first and best-known examples of a flat surface of infinite genus and finite area, with suspensions of the dyadic odometer.   We then use techniques from the field of flat surfaces to deduce the classical result that the dyadic odometer is ergodic.    (A crucial ingredient in our proof of Theorem \ref{thm:erg}, a condition guaranteeing ergodicity, will be the use of renormalization maps, which in this example are manifested as the uniformly hyperbolic Baker's transformation.)  

We now list five descriptions of the dyadic odometer, the last of which is Chamanara's surface.

\begin{figure}[ht]
  \centering
  \includegraphics[width=2in]{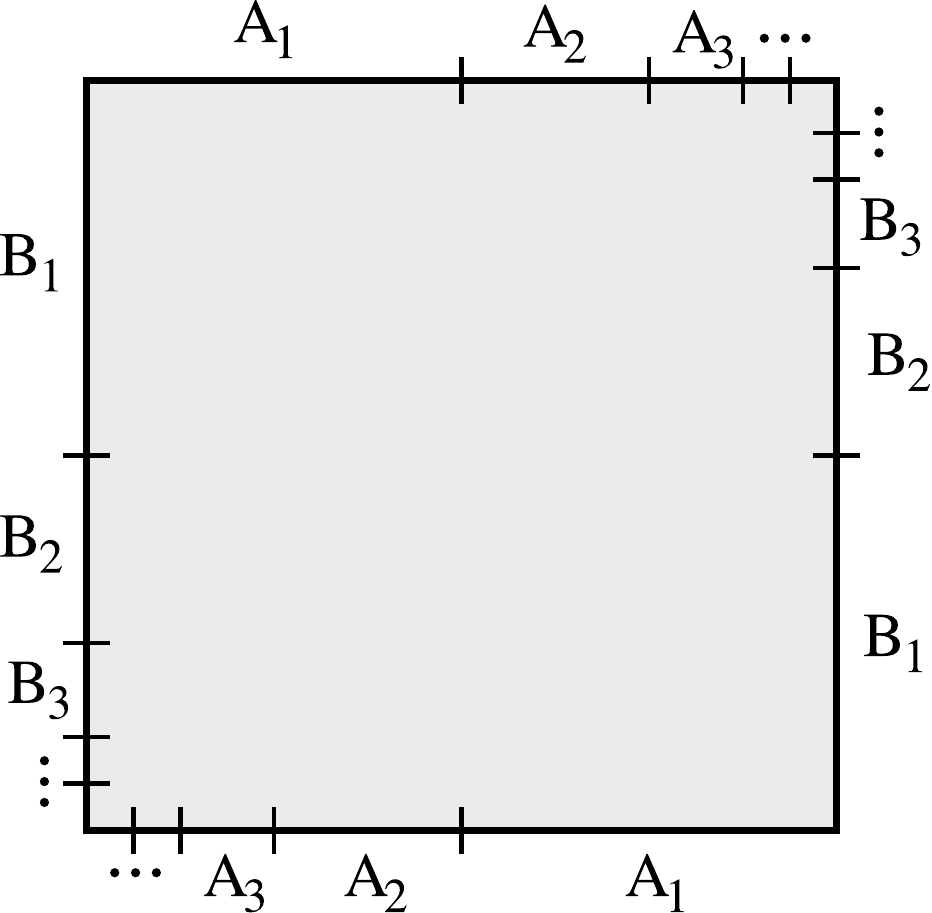}
  \caption{Chamanara's surface. The length of an edge labeled $A_i$ or $B_i$ is $2^{-i}$.}
  \label{fig:Chamanara}
\end{figure}

\noindent \textbf{Description 1.}  The \emph{dyadic odometer} is the map $\Phi:X\rightarrow X$, where $X = \{0,1\}^\mathbb{N}$, defined as addition by 1 in base two of $.1000\dots$ with infinite carry to the right. The dyadic odometer is a minimal, non weak-mixing, and uniquely ergodic transformation of the Cantor set $X$. 

\smallskip 

\noindent \textbf{Description 2.}
The directed graph in Figure \ref{fig:dyadicOdom} is a  \emph{Bratteli diagram} for the dyadic odometer (Bratteli diagrams will be defined in \S \ref{ss:standardBratteliDefinitions}). The space $X$ can be identified with the space of all infinite paths starting at the vertex $V_0$ and moving uniformly downwards along the diagram. The dyadic odometer can be defined as a homeomorphisms of the space of all infinite paths on this diagram. This is an example of an \emph{adic transformation}.

\smallskip

\noindent \textbf{Description 3.}
Define a map $B: X \rightarrow [0,1]$ by $B(a) = \sum_{i=1}^\infty a_i2^{-i}$. Outside a countable set of $X$, this mapping is a bijection onto $I =  [0,1]\backslash P$, where $P$ is some countable subset of $[0,1]$.  Consider the map $R: I\rightarrow I$ defined by $R(x) = B\circ\Phi\circ B^{-1}(x)$. It is a restriction to $I$ of the map $\bar{R}:[0,1]\rightarrow [0,1]$ defined by 
\begin{equation}
\label{eqn:VDC}
\bar{R}(1-2^{-n}+x) = 2^{-(n+1)}+x \hspace{.35 in } \mbox{ for } \hspace{.35 in } 0\leq x < 2^{-(n+1)},\,\,\, n\in\mathbb{N}
\end{equation}
and $\bar{R}(1) = 0$. The map $\bar{R}$ is also known as the \emph{Van der Corput map}. It is a piecewise isometry of the unit interval which can also be described as an interval exchange transformation on infinitely many intervals. A graph of the Van der Corput map is found in Figure \ref{fig:dyadicOdom}.

\smallskip

\noindent \textbf{Description 4.} The map $\bar{R}$ can also be constructed via the process of \emph{cutting and stacking} (which we will define in \S \ref{sec:CAS}), as follows (see Figure \ref{fig:dyadicOdom}). Consider the interval $[0,1]$ and cut it into the two disjoint intervals $[0,\frac{1}{2})$ and $[\frac{1}{2},1)$. Consider the map $T_1: [0,\frac{1}{2})\rightarrow [0,1]$ defined as the unique, orientation-preserving isometry sending $0$ to $\frac{1}{2}$. This map can also be seen as the map defined by ``stacking'' the interval $[\frac{1}{2},1)$ over the interval $[0,\frac{1}{2})$, thereby creating a ``tower'' made up of two intervals, and mapping a point $x\in [0,\frac{1}{2})$ to the point directly above it in the upper level of the stack. 

We now define a map $T_2:[0,\frac{3}{4}) \rightarrow [0,1]$ with the property that $T_2|_{[0,\frac{1}{2})} = T_1$. Considering the tower consisting of the interval $[\frac{1}{2},1)$ over the interval $[0,\frac{1}{2})$, we cut this tower into 4 intervals of equal length: $[0,\frac{1}{4})$ and $[\frac{1}{4},\frac{1}{2})$ on the bottom and $[\frac{1}{2},\frac{3}{4})$ and $[\frac{3}{4},1)$ on top. We now stack the two rightmost intervals ($[\frac{1}{4},\frac{1}{2})$ and $[\frac{3}{4},1]$) on top of the tower created by the leftmost intervals, thereby creating a tower consisting of 4 intervals of length $\frac{1}{4}$. The map $T_2$ is defined, for a point $x$ on the three bottom intervals, as its image by moving up one level on the tower. As such it is a piecewise isometry and it satisfies $T_2|_{[0,\frac{1}{2})} = T_1$.

We can continue this process indefinitely and create a sequence of maps $T_k:[0,\frac{2^k-1}{2^k}) \rightarrow [0,1]$ with the property that $T_{k+1}|_{[0,\frac{2^k-1}{2^k}]} = T_k$. Let $T:[0,1]\rightarrow [0,1]$ be the pointwise limit of this sequence of maps which maps $1$ to $0$. The limiting map $T$ coincides with the Van der Corput map (\ref{eqn:VDC}).

\begin{figure}[h!]
  \centering
  \includegraphics[width=.8 \linewidth]{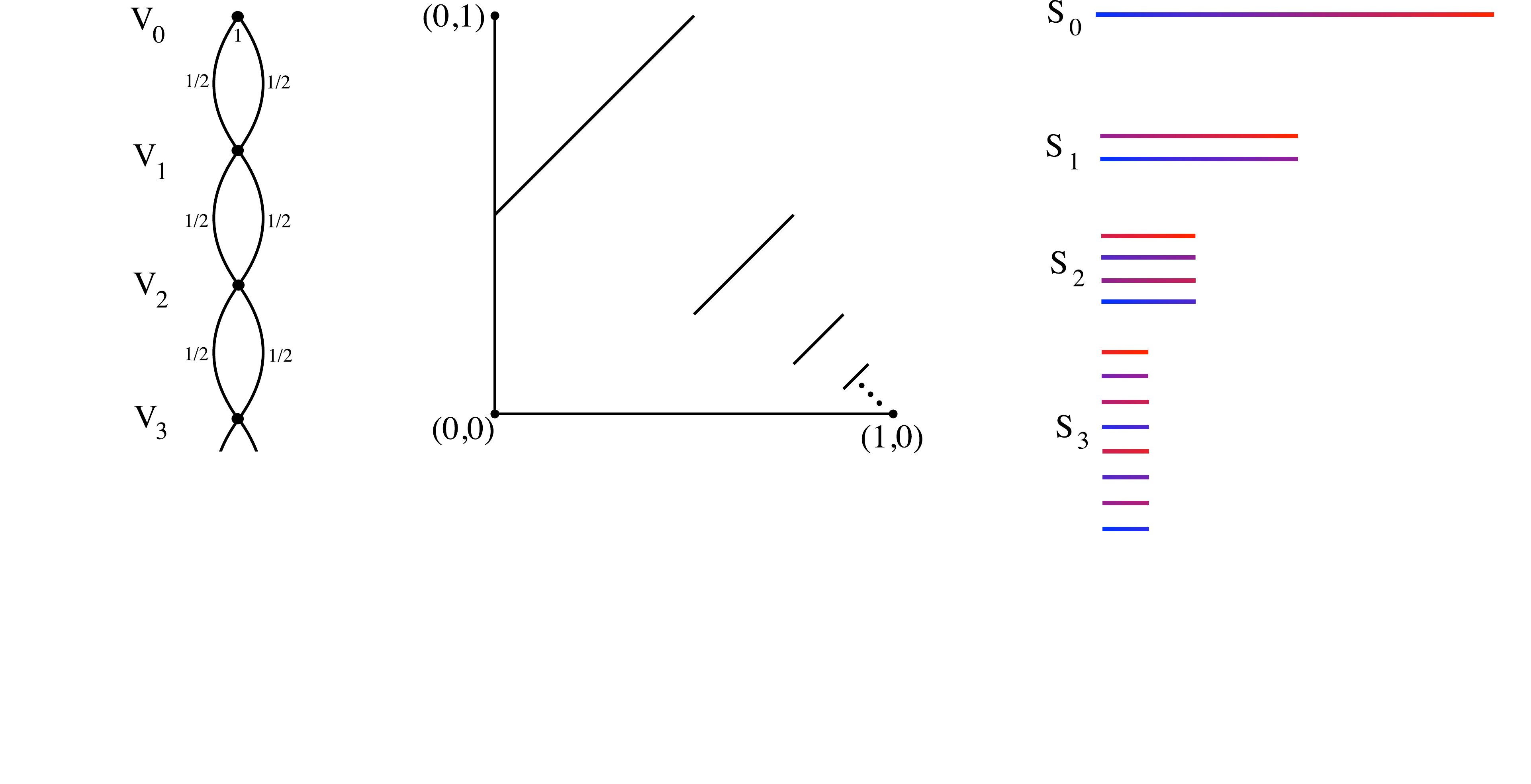}
  \vspace{-2cm}
  \caption{On the left, a Bratteli diagram for the dyadic odometer. In the middle, the graph of the Van der Corput map. On the right, the corresponding cutting-and-stacking representation for the dyadic odometer.}
  \label{fig:dyadicOdom}
\end{figure}

\begin{figure}[h!]
  \centering
  \includegraphics[width = 3 in]{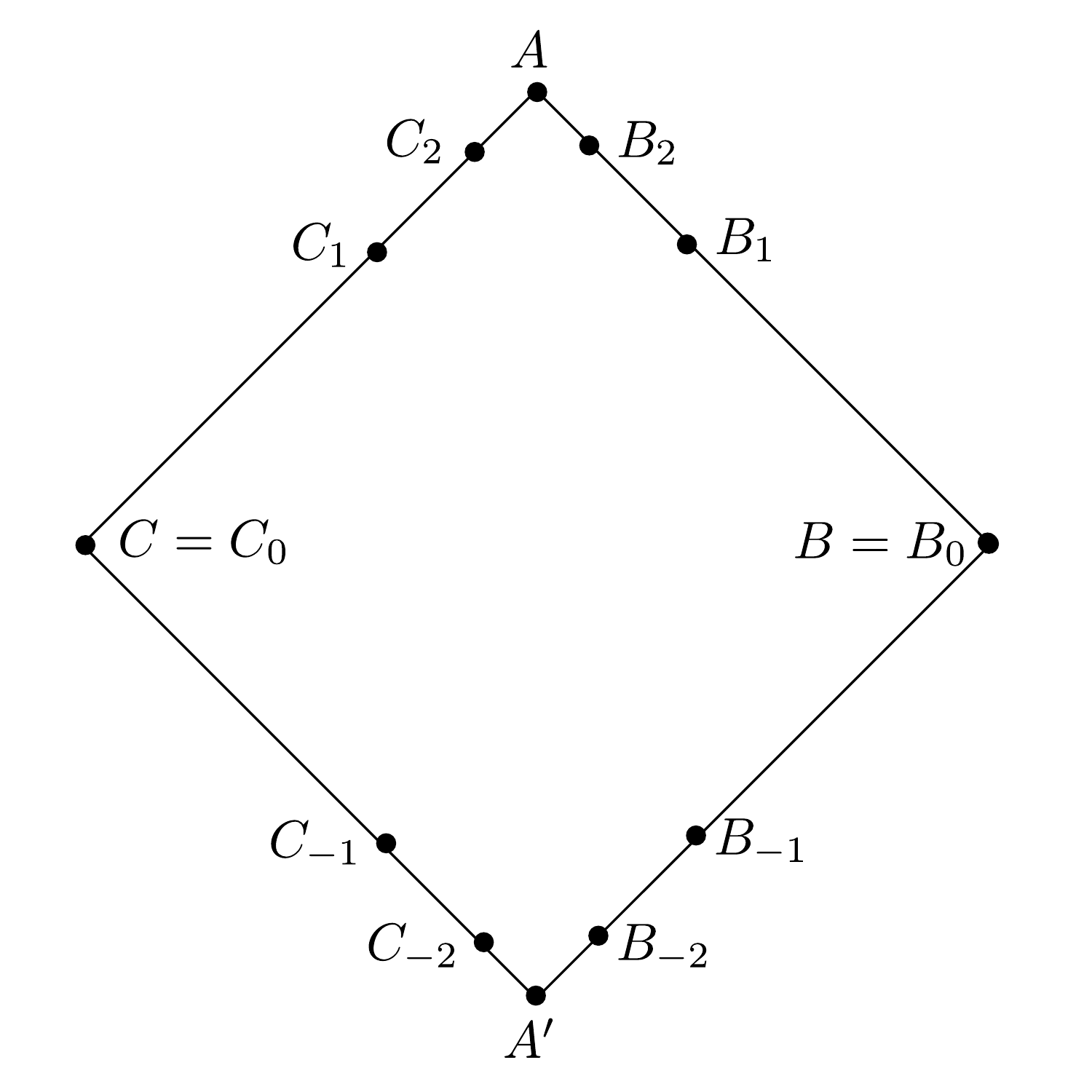}
  \caption{Construction of the surface $S_p$.}
  \label{fig:chaSurf}
\end{figure}

\noindent \textbf{Description 5a.} Let us consider the suspension flow $\phi_t$ for the map $\bar{R}$: it is the vertical flow generated by the vector field $\partial_y$ on the surface $S'$ obtained by gluing edges of unit square $[0,1]^2$ through the identifications $(x,1)\sim (R(x),0)$. Since $\bar{R}$ is conjugate to the odometer, the flow $\phi_t$ is non weak-mixing and uniquely ergodic. Identifying the vertical edges $\{0,1\}\times[0,1]$ of $S'_{2^{-1}}$ through $(1,y)\sim (0,R(y))$ gives us a surface $S_{2^{-1}}$ (the 2 denotes the fact that we used the dyadic odometer to construct it). Since we used the dyadic odometer twice to obtain the surface $S_{2^{-1}}$ (once for the identifications along de horizontal edges, another for identifications along vertical edges), we will extend the diagram to consider its bi-infinite version, that is, using two copies of the Bratteli diagram in Figure \ref{fig:dyadicOdom} which are glued at $V_0$: one diagram will help in giving the identifications $A_i$ along the top/bottom edges while the other will correspond to identifications $B_i$ along the left/right edges. See Figure \ref{fig:Chamanara}.

The horizontal flow on $S_{2^{-1}}$ is generated by the vector field $\partial_x$ is conjugated to the vertical flow $\phi_t$ through the involution $i:(x,y)\mapsto (y,x)$. The surface $S_{2^{-1}}$ is a non-compact surface of infinite genus, has finite area, and has a flat metric everywhere. Note that the ``point'' at the corner of the square is not part of the surface since, by the identifications, it admits no regular neighborhood.

\smallskip
\noindent \textbf{Description 5b.}
Let us now consider the construction due to Chamanara of an infinite family of flat surfaces of infinite genus parametrized by $p\in(0,1)$ \cite{chamanara} (see Figure \ref{fig:chaSurf}). Let $\mathcal{S} = ABA'C$ be a square centered at the origin in $\mathbb{C}$ such that its sides have length one and the diagonal $BC$ is on the real line. Set $B_0 = B$ and $C_0 = C$. For $i\geq 1$ define $B_i$ (respectively $B_{-i}$, $C_i$ and $C_{-i}$) to be the point on the interval $BA$ (respectively $BA'$, $CA$ and $CA'$) such that the length of $AB_i$ (respectively $A'B_{-i}$, $AC_i$, and $A'C_{-i}$) is $p^i$ for some $0<p<1$. The sides $B_iB_{i+1}$ and $C_{-(i+1)}C_{-i}$ are identified by a translation. This identifies all the points of the form $B_{2k+1}$ and $C_{2k}$ and the points of the form $B_{2k}$ and $C_{2k+1}$. We denote the identification map by $Q_p$. The resulting surface obtained from the above is denoted by $S_p = Q_p(\mathcal{S})$ and it is clear that it is a flat surface of finite area. It is shown in \cite[Proposition 9]{chamanara} that it is an infinite genus surface with one end. It is also easy to see that it is the geometric limit of finite genus surfaces: let $\mathcal{S}^n$ be the subset of $\mathcal{S}$ bounded from above by $C_nB_n$ and below by $C_{-n}B_{-n}$. Then for each $n$, $S^n_p = Q_p(\mathcal{S})$ is a translation surface of genus $n$ with two singularities of order $n-1$. Then limiting surface $S^n_p\longrightarrow S_p$ is our infinite genus surface with singularities of infinite order. 

The surface $S_{\frac{1}{2}}$ constructed through identifications given by the dyadic odometer (via the conjugacy with the Van der Corput map) is the same surface as Chamanara's surface $S_p$ for $p = \frac{1}{2}$. In fact, for any prime $p$, any $S_p$ can be constructed in a similar way by suspending the $p$-adic odometer (defined as addition by 1 in base $p$) as we did for the dyadic odometer. Through this identification we are therefore able to go from statements of the dynamics of translation flows on flat surfaces of infinite topological type to statements about the dynamics of the $p$-adic odometers. For example, since we know that the dynamics of the odometers are strictly ergodic and non weak-mixing, then translation flows parallel to the edges of the unit square used to construct $S_p$ also have the same ergodic properties. It is also known that the flow is uniquely ergodic for flows in other directions of $S_p$ for $p\in\mathbb{Q}\cap(0,1)$, and this corresponds to unique ergodicity of a system which couples two odometers in a very non-trivial way which depends on the direction of the flow.

\bigskip

We now give a short proof that the dyadic odometer is uniquely ergodic by applying a result developed for surfaces to the representation of the dyadic odometer as a flat surface (Description 5a). We may deform the flat surface $S_{2^{-1}}$ by uniformly stretching the horizontal direction by a factor of 2 (by the same deformation taking $[0,1]^2$ to $[0,2]\times [0,2^{-1}]$), compressing in the vertical direction by a factor of $2^{-1}$, cutting it through the vertical line dividing the rectangle into two pieces of equal area, and stacking the piece on the right on top of the piece on the left (since these edges are identified, we are not changing the surface). This is the so-called \emph{Baker's transformation}. The end result of this procedure gives us back the same surface, $S_{2^{-1}}$. Therefore, this procedure of stretching and cutting and stacking induces a uniformly hyperbolic automorphism of $S_{2^{-1}}$. The stable and unstable foliations of this automorphism coincide with those generated by the vector fields $\partial_y$ and $\partial_x$, respectively. Therefore, by \cite{BowenMarcus}, these foliations, considered as orbits of flows, are uniquely ergodic. Since these vector fields generated a system which is measurably isomorphic to all descriptions of the dyadic odometer, we conclude that the dyadic odometer is uniquely ergodic.

\section{Flat Surfaces}
\label{sec:flat}
In this section we briefly review necessary concepts from the field of flat surfaces and Teichm\"{u}ller dynamics. Two references for this are \cite{zorich:intro, FM:intro}, although they concentrate exclusively on compact flat surfaces. There is no standard reference for non-compact flat surfaces (which are the types of surfaces with which we mostly deal here) although \cite{strebel:book} has relevant results which apply to general flat surfaces. 

 A \emph{translation surface} is a two-dimensional, real manifold $S$ with no boundary such that all transition maps between manifold charts are translations, i.e., maps of the form $z\mapsto z+c$ for some $c\in \mathbb{C}$. Because the transition maps on $S$ are translations, both the Euclidean metric on charts and directions (e.g. ``vertical," ``horizontal," etc.) are preserved by all transition maps and hence are well-defined on $S$. In this paper, we will call such surfaces also \emph{flat surfaces}. 

More specifically, a \emph{flat surface} is a Riemann surface $S$ together with a non-constant closed $1$-form $\alpha$ on $S$ that is holomorphic with respect to the complex structure and is singular at some set of points $\Sigma \subset \bar{S}$, where $\bar{S}$  is some metric completion of $S$ to be described shortly. The set of all points in $S$ which are zeros of $\alpha$ is the set $\Sigma\cap S$. We will always assume that $S$ has empty boundary. The adjective \emph{flat} describes the fact that we can choose adapted local coordinates on $S-\Sigma$ by integrating $\alpha$ which give an atlas for which the change of charts are given by functions which are translations. In other words, for $p\in S-\Sigma$ and $q$ in a neighborhood $\mathcal{U}_p$ of $p$, the charts $q\mapsto z(q) = \int_p^q \alpha$ maps $\mathcal{U}_p$ to a neighborhood of the origin in $\mathbb{C}$. The choice of path from $p$ to $q$ is irrelevant since $\alpha$ is closed and therefore locally exact in $\mathcal{U}_p$ since it is simply connected. By pulling back the Euclidean metric on $\mathbb{C}$ we obtain a flat metric on $S-\Sigma$. For this metric $\bar{S}$ is the metric completion of $S$. We will denote flat surfaces by pairs of the form $(S,\alpha)$ whenever we need to emphasize the role of the 1-form $\alpha$.

Two flat surfaces $(S,\alpha)$ and $(S',\alpha')$ are equivalent if there is a biholomorphic map $f:S\rightarrow S'$ such that $f(S) = S'$ and $f^*\alpha' = \alpha$; i.e. the flat metric of one can obtained as the pull-back of the flat metric of the other. All the notions defined hereafter are well-defined for an equivalence class in that they are invariant under conformal homeomorphism.  

There are two distinguished foliations that come with a flat surface $(S,\alpha)$, called the \emph{vertical and horizontal foliations}, denoted by $\mathcal{F}_\alpha^v$ and $\mathcal{F}_\alpha^h$, obtained by integrating the distributions given by the real and imaginary parts of $\alpha$ in the set $S-\Sigma$:
$$\mathcal{F}_\alpha^v \equiv \langle \mbox{ker}\, \Re(\alpha)\rangle \hspace{.5in}\mbox{ and }\hspace{.5in} \mathcal{F}_\alpha^h \equiv \langle \mbox{ker}\, \Im(\alpha)\rangle.$$

The unit speed parametrizations of the of the vertical and horizontal foliations $\mathcal{F}_\alpha^v$ and $\mathcal{F}_\alpha^h$ give us the \emph{vertical and horizontal flows} on $(S,\alpha)$. These are \emph{translation flows}. From now on, we will refer to the vertical flow whenever we talk about a translation flow.

For any flat surface $(S,\alpha)$, there is a set of deformations of the flat metric which is parametrized by group $SL(2,\mathbb{R})$. Namely, for $A\in SL(2,\mathbb{R})$, the flat surface given by $A\cdot(S,\alpha)$ is given by post-composing the charts of $(S,\alpha)$ with $A$. The image in $PSL(2,\mathbb{R})$ of the stabilizer of this deformation is called the \emph{Veech group} of $(S,\alpha)$ and it is denoted $SL(S,\alpha)$. It can be also defined as the image in $PSL(2,\mathbb{R})$ of the group of derivatives of affine diffeomorphisms of $(S,\alpha)$.  The action of the one-parameter diagonal subgroup
\begin{equation}
\label{eqn:teich}
g_t\equiv \left\langle \left( \begin{array}{cc} e^t & 0 \\ 0 & e^{-t} \end{array} \right) : t\in \mathbb{R} \right\rangle
\end{equation}
gives what is called the \emph{Teichm\"uller deformation}. We will denote by $g_t(S,\alpha)$ the one-parameter family of flat surfaces obtained from $(S,\alpha)$ through Teichm\"uller deformations.

Let $\mbox{dist}_t(x,y)$ denote the distance in $g_t(S,\alpha)$ between $x$ and $y$. Since $\Sigma$ is a subset of the metric completion, the metric extends to points in $\Sigma$. For any set $A\subset S-\Sigma$, the quantity $\mbox{dist}_t(A,\Sigma)$ is defined as the value of the infimum $\inf\{dist_t(x,y) \mid x \in A, y \in \Sigma \}$, which roughly measures the distance of the closest point in $A$ to the ideal boundary of $S-\Sigma$, i.e., the distance of the closets point of $A$ to the singularities $\Sigma$. 
We will use the following result proved in \cite{rodrigo:erg}.
\begin{theorem}
\label{thm:flat}
Let $(S,\alpha)$ be a flat surface of finite area. Suppose that for any $\eta>0$ there exist a function $t\mapsto \varepsilon(t)>0$, a one-parameter family of subsets
$$S_{\varepsilon(t),t} = \bigsqcup_{i=1}^{C_t}S_t^i$$ 
of $S$ made up of $C_t < \infty$ path-connected components, each homeomorphic to a closed orientable surface with boundary, and functions $t\mapsto \mathcal{D}_t^i>0$, for $1\leq i \leq C_t$, such that for 
$$\Gamma_t^{i,j} = \{\mbox{paths connecting }\partial S_t^i \mbox{ to }\partial S_t^j\}$$
and
\begin{equation}
\label{eqn:systole1}
\delta_t = \min_{i\neq j} \sup_{\gamma\in\Gamma_t^{i,j} }\mbox{dist}_t(\gamma,\Sigma)
\end{equation}
the following hold:
\begin{enumerate}
\item  $\mathrm{Area}(S\backslash S_{\varepsilon(t), t}) < \eta$ for all $t>0$,
\item $\mbox{dist}_t(\partial S_{\varepsilon(t),t},\Sigma) > \varepsilon(t)$ for all $t>0$,
\item the diameter of each $S_t^i$, measured with respect to the flat metric on $(S,\alpha_t)$, is bounded by $\mathcal{D}_t^i$ and
\begin{equation}
\label{eqn:integrability}
\int_0^\infty \left( \varepsilon(t)^{-2}\sum_{i=1}^{C_t}\mathcal{D}_t^i + \frac{C_t-1}{\delta_t}\right)^{-2}\, dt = +\infty.
\end{equation}
\end{enumerate}
Moreover, suppose the set of points whose translation trajectories leave every compact subset of $S$ has zero measure. Then the translation flow is ergodic.
\end{theorem}

\section{Bratteli diagrams}
\label{sec:brat}
Bratteli diagrams were introduced in \cite{bratteli} to study $C^*$-algebras; Vershik associated dynamical systems to these diagrams in \cite{vershik}. These maps, which are called Bratteli-Vershik or adic transformations, are defined on the space of infinite paths starting at a root vertex in a Bratteli diagram; the transformation maps a path to its successor (when possible) under a given ordering.  Vershik showed that every measure-preserving transformation on a Lebesgue space is measure-theoretically isomorphic to an adic transformation (\cite{vershik}). In \S \ref{ss:standardBratteliDefinitions}, we review some of the theory of Bratteli diagrams.    In \S \ref{ss:biinfiniteBratteliDefinitions}, we introduce bi-infinite generalizations of Bratteli diagrams.   Definition \ref{def:diagram} defines a \emph{diagram}, a bi-infinite Bratteli diagram together with some additional data; diagrams are the basic combinatorial objects we will associate to infinite type translations surfaces.  

\subsection{Bratteli diagrams}
\label{ss:standardBratteliDefinitions}
\bigskip

In this section we present some background and definitions in the study of Bratteli diagrams.  For more information on the theory of Bratteli diagrams and associated dynamical systems, see, for example, \cite{HermanPutnamSkau, BKM09,DHS}.

\begin{definition} \label{def:BratteliDiagram}
A \emph{Bratteli diagram} $B=(V,E)$ is a connected infinite directed graph together with partitions of the vertex set $V$ and edge set $E$ of the graph into countable unions of pairwise disjoint nonempty finite sets 
$$V = \bigsqcup_{i \geq 0} V_i \textrm{ and } E = \bigsqcup_{i > 0}E_i$$
such that $s(E_i)=V_{i-1}$ and $r(E_i)=V_i$ for all $i>0$, where $s$ and $r$ are the associated \emph{source} and \emph{range} maps ($s,r:E\rightarrow V$), respectively.
\end{definition}
		
The set $V_i$ of vertices is called the $i^{th}$ level of the Bratteli diagram.  We will denote $|V_i|$ by $c_i$.   Note that the conditions $s(E_i)=V_{i-1}$ and $r(E_i)=V_i$ for all $i>0$ imply that every vertex in $V_0$ is the source of an edge in $E_1$ and every vertex $v \in V_i$ for $i>0$ is both the source vertex of an edge in $E_{i+1}$ and the range vertex of an edge in $E_i$.  

Given a Bratteli diagram $B$, for $i > 0$, the incidence matrix $F_i = [f^i_{v,w}]$ is a $c_{i} \times c_{i-1}$ matrix whose entries $f^i_{v,w}$ are the number of edges between the vertices $v \in V_{i}$ and $w \in V_{i-1}$:
$$f^i_{v,w} = | \{e \in E_i \mid r(e)=v \textrm{ and }s(e)=w\}|.$$
It follows from the conditions $s(E_i)=V_{i-1}$ and $r(E_i)=V_i$ for all $i>0$ that none of the matrices $F_k$ have a row or column which consists of all zero entries.  
Given an initial vector $h^0 = (h_1^0,\dots, h_{|V_0|}^0)\in \mathbb{R}^{|V_0|}_+$ with all positive entries, for each $i \geq 0$, we define (recursively) a \emph{height} vector $h^i=(h^i_1,\dots,h^i_{c_i})  \in \mathbb{R}^{c_i}.$  The height vectors are then given by the recursive formula
\begin{equation}
\label{eqn:heights}
h^{i+1} = F_i h^i. 
\end{equation}

For nonnegative integers $k < l$, a \emph{finite path} from a vertex in $V_k$ to a vertex in $V_l$ is a set of edges $e_{k+1}, \dots, e_l$, such that  $e_i \in E_i$ and $r(e_i) = s(e_{i+1})$ for all $i$. We will denote such a path by $(e_k, \dots,  e_l)$. For a path $p = (e_i,\dots, e_j),$ we define $s(p) = s(e_i)$ and $r(p) = r(e_j)$. For $0\leq k<k'$ We will denote by $E_{k,k'}$ the set of all (oriented) paths from $V_k$ to $V_{k'}$, and by $E_{k',k}$ the set of paths from $V_{k'}$ to $V_k$, which are just paths in $E_{k,k'}$ in reverse orientation.

For a Bratteli diagram $B$, we denote by $X_B$ the set of all infinite paths in $B$ which start at a vertex in $V_0$.  For a point $x \in X_B$, denote by $x_i$ the $i^{th}$ edge of the path $x$.  We topologize $X_B$ by specifying a clopen basis of all cylinder sets
$$U(e_1,\dots,e_n) := \{x \in X_B \mid x_i = e_i \textrm{ for all } i \in \{1,\dots,n\} \},$$ 
where $(e_1,\dots, e_n)$ is a finite path starting at a vertex in $V_0$.  As such, $X_B$ is a compact, Hausdorff, zero-dimensional space with a countable basis of clopen sets.

\begin{definition} 
\label{def:order1}
An \emph{ordered} Bratteli diagram $(B,\leq_r)$ is a Bratteli diagram $B=(V,E)$ together with a partial order $\leq_r$  on $E$ so that edges $e$ and $e^{\prime}$ are comparable under $\leq_r$ if and only if $r(e)=r(e^{\prime})$.  \end{definition}

To pass from Bratteli diagrams to cutting and stacking maps (and flat surfaces) in a canonical way, we will want an additional partial order that compares edges with the same source vertex.  Thus, we define fully ordered Bratteli diagrams:

\begin{definition}
\label{def:fullyOrderedB}
A \emph{fully ordered} Bratteli diagram $(B,\leq_{r,s})$ is an ordered Bratteli diagram $(B,\leq_r)$ together with a partial order $\leq_s$ on $E \cup V_0$ so that any two edges $e,e^{\prime}$ are comparable under $\leq_s$ if and only if $s(e)=s(e^{\prime})$, $\leq_s$ is a total order on $V_0$, and edges are not comparable with vertices.\end{definition}

The partial order $\leq_r$ in an ordered Bratteli diagram $(B,\leq_r)$ induces a lexicographic partial order on the set of all finite paths from $V_i$ to $V_j$ for any $j>i$. Namely, we write
 $$(e_{i+1},e_{i+2},\dots,e_j) <_r  (f_{i+1},f_{i+2},\dots, f_j)$$
if and only if the exists $k\in\{i+1,\dots, j\}$ such that $e_l = f_l$  for $k<l\leq j$ and $e_k <_r f_k$.  Two infinite paths $x$ and $y$ in $X_B$ are comparable under $\leq_r$ if they agree after some level $n$ ($x_k=y_k$ for all $k > n$) and $x_n \not = y_n$; then we define $x <_r y$ if and only if $x_n <_r y_n$.

An infinite path $x \in X_B$ is \emph{maximal} under $\leq_r$ if $x_i$ is a maximal edge according to $\leq_r$ for all $i \in \mathbb{N}$.  Denote by $X_{max}$ the set of maximal paths in $X_B$; $X_{min}$ is defined similarly.   Given any path $x \in X_{B} \setminus X_{max}$, there exists a smallest integer $i$ such that $x_i$ is not maximal.  Since there exist only finitely many (finite) paths from a vertex in $V_0$ to the vertex $r(x_i)$, the infimum $\inf \{y \in X_B \mid y >_r x\}$ is achieved by a path in $X_B$.

\begin{definition}
Let $(B,\leq_r)$ be an ordered Bratteli diagram.  For a point $x \in X_B \setminus X_{max}$, define the \emph{successor} of $x$ to be
\[ \alpha = \inf \{y \in X_B \mid y >_r x\} \] 
\end{definition}

\begin{definition}Let $(B,\leq_r)$ be an ordered Bratteli diagram. The \emph{Bratteli-Vershik or adic} transformation $T:X_B\backslash X_{max} \rightarrow X_B\backslash X_{min}$ is the map which sends a point  $x\in X_B\backslash X_{max}$ to its successor.
\end{definition}

\begin{definition}
Let $B = (V,E)$ be a Bratteli diagram. The \emph{tail equivalence relation} is a relation $\sim$ on $X_B$ defined by
$$x\sim y \hspace{.2in} \mbox{ if and only if }\hspace{.2in} \exists N\geq 0 \mbox{ such that }x_k = y_k \hspace{.2in} \mbox{ for all } \hspace{.2in}k>N.$$
\end{definition}
Note that the tail equivalence relation is independent of the many possible choices of orders $\leq_{r,s}$ on a Bratteli diagram.

\begin{definition} \label{def:aperiodic}
A Bratteli diagram is \emph{aperiodic} if every tail equivalence class of $X_B$ is infinite.  In this case any adic transformation defined on $X_B$ is also called \emph{aperiodic}.  
\end{definition}

\begin{definition} \label{def:periodic}
A Bratteli diagram is \emph{completely periodic} if every tail equivalence class of $X_B$ is finite.  In this case any adic transformation defined on $X_B$ is also called \emph{completely periodic}.
\end{definition}
The notion of complete periodicity will be made more clear in the decomposition (\ref{eqn:decomposition}) below.

\begin{remark}
\label{rem:extendingAdicMap}
Whenever $|X_{min}| = |X_{max}| < \infty$, the adic transformation can be extended to all of $X_B$ and defines a homeomorphism.  In particular, any finite tail equivalence class has a unique maximal path and a unique minimal path in $X_B$; in this case, it is natural to extend the adic transformation so that it maps this maximal path to this minimal path.  
Thus, the (natural extension of the) adic map on a completely periodic Bratteli diagram (defined in Definition \ref{def:periodic}) is periodic.  
\end{remark}

\begin{definition}  \label{def:minimalBratteli}
A minimal subset $X'$ of $X_B$, for a Bratteli diagram $B=(V,E)$, is a set that is closed under the tail equivalence relation $\sim$ and is minimal among such sets with respect to inclusion.  A Bratteli diagram $B$ is \emph{minimal} if $X_B$ has no proper minimal subsets.  
\end{definition}

\begin{remark}
Definition \ref{def:minimalBratteli} is equivalent to the following condition: $X_B$ is minimal if for any $x = (x_1,x_2,\dots )\in X_B$, $k>0$, and $v\in V_k$, there exists an integer $j >k$ and a path $(e_{k+1},\dots, e_j)$, with $e_i \in E_i$ for all $i$, such that $s(e_{k+1}) = v$ and $r(e_j) = s(x_{j+1})$.
\end{remark} 

\begin{definition}
A Borel probability measure $\mu$ on $X_B $ is an \emph{invariant measure for the tail equivalence relation} if  for any two infinite paths $p_1 = (e_1,e_2,\dots)$ and $p_2 = (f_1,f_2,\dots) $ in $X_B$ with $p_1 \sim p_2$ and for any $l \in \mathbb{N}$ such that $e_k=f_k$ for all $k>l$, we have $\mu(U(e_1,\dots,e_l)) = \mu(U(f_1,\dots,f_l))$. 
\end{definition}

\begin{remark}
\label{rem:measures}
For an ordered Bratteli diagram $B$,  a Borel probability measure on $X_B$ that is invariant with respect to the adic transformation is also an invariant measure for the tail equivalence relation.  The converse is not true: the support of  a Borel probability measure which is invariant for the tail equivalence relation could be contained in $X_{max}$ for some order $\leq_r$; this set has empty intersection with the domain of the adic transformation. In fact, it is possible that every invariant Borel probability measure for the tail equivalence relation have a support contained in $X_{max}$.  See \S \ref{subsubsec:HajianKakutani} for an example of an adic transformation which admits no invariant Borel probability measure but does admit an invariant infinite Borel measure.   
\end{remark}

Using the Compact Representation Lemma \cite{AS:Chi2} with the Krylov-Bogolyubov theorem we obtain the following basic result (see also \cite{PetersenSchmidt}).

\begin{proposition}
\label{thm:existence}
Let $ (B,\leq_r)$ be an ordered Bratteli diagram. Then there is at least one Borel probability measure on $X_B$ which is invariant for the adic transformation defined by the partial order $\leq_r$.
\end{proposition}

For any Bratteli diagram $B$, there is a decomposition of $X_B$ as
\begin{equation}
\label{eqn:decomposition}
X_B = X_P \bigsqcup X_M,
\end{equation}
where 
$$X_P = \bigsqcup_{i} \bigcup_{x\in X_P^i} x $$
where each $X_P^i$ is a finite tail-equivalence class, called a \emph{periodic component}. The set $X_M$ consists of the \emph{minimal components}
$$X_M = \bigsqcup_i X_M^i,$$
where each $X_M^i$ is a minimal subset. This decomposition will be analogous to the decomposition of a flat surface into minimal and periodic components.  If a Bratteli diagram is not made up only of a single minimal component, there is a clear obstruction to ergodicity of any adic transformation defined from it.

We now introduce a special type of measure on the space of all paths. It will be invariant in the sense made precise in the paragraph before Lemma \ref{l:weighttomeasure}.
\begin{definition} 
\label{def:weightfunctiondef}
A \emph{weight function} for a Bratteli diagram $B=(V,E)$ is a map $w:V_0 \cup E \rightarrow (0,\infty)$ such that 
\begin{enumerate}
\item \label{it:weight1} for any vertex $v \in V$ and any two positively oriented finite paths $(e_1,\dots,e_j)$ and $(f_1,\dots,f_j)$ from vertices in $V_0$ to $v$, 
$$w(s(e_1)) \cdot \prod_{i=1}^j w(e_i) = w(s(f_1)) \cdot \prod_{i=1}^j w(f_i).$$
\item \label{it:weight2} for any $v \in V$, $$\sum_{e \in s^{-1}(v)} w(e) = 1,$$
\item \label{it:weight3} for any infinite path $x= (x_1,x_2,\dots) \in X_{B}$ that does not belong to a finite tail equivalence class (i.e. is an element of a minimal component), $$\lim_{n \rightarrow \infty} w(s(x_1)) \cdot \prod_{i=1}^n w(x_i) =0.$$
\end{enumerate}
\end{definition}

For a weight function $w$ and $v\in V_k$ with $k>0$, we can define the quantity $w(v)$ by
\begin{equation}
\label{eqn:extWeight}
w(v) = w(s(e_1)) \cdot \prod_{i=1}^j w(e_i)
\end{equation}
for any path $(e_1,\dots, e_k)$ with $r(e_k) = v$ from $V_0$ to $v$. By (i) in Definition \ref{def:weightfunctiondef}, it is independent of the path $(e_1,\dots,e_k)$ taken. For an element $x = (e_1, e_2, \dots) \in X_B$ we also define the quantity
$$w(x) = w(s(e_1))\prod_{k=1}^\infty w(e_k).$$

\begin{definition} A weight function $w$ on a Bratteli diagram $B=(V,E)$ is said to be a \emph{probability weight function} if $$\sum_{v \in V_0} w(v) = 1$$ and is said to be a \emph{finite weight function} if $$\sum_{v \in V_0} w(v) < \infty.$$
\end{definition}

The following lemma, whose proof is straightforward and is left to the reader, records the fact that weight functions on Bratteli diagrams correspond to invariant measures for the tail equivalence relation.  This correspondence between a measure $\mu$ and weight $w$ to which we refer is obtained by setting $w(v) = \mu(v)$ for $v \in V_0$ and $w(e) = \frac{\mu(r(e))}{ \mu(s(e))}$ for $e \in E$.  

\begin{lemma}
\label{l:weighttomeasure}
A probability weight function $w$ on a Bratteli diagram $B=(V,E)$ determines a unique invariant Borel probability measure for the tail equivalence relation. Conversely, an invariant Borel probability measure for the tail equivalence relation determines a unique probability weight function on $B$.    
\end{lemma}

\begin{remark}
In section \S \ref{sec:Dictionary}, we will develop a correspondence between weighted, fully ordered Bratteli diagrams and cutting and stacking maps (\S \ref{sec:CAS}).   Each vertex $v \in V_i$ in Bratteli diagram $B=(V,E)$ will correspond to a tower in the stack $S_i$, and the value assigned to a vertex by the Borel measure associated to a weight function (as in Lemma \ref{l:weighttomeasure}) will be the width of the levels of that tower.  We will see that condition \ref{it:weight1} means that two subtowers of $S_i$ which are stacked on top of each other to form a tower of stack $S_{i+1}$ have the same width.  Condition \ref{it:weight2} reflects the fact that the sum of the widths of the subtowers into which a given tower is cut must equal the width of that tower.  Condition \ref{it:weight3} says that the widths of the stacks which limit to a minimal set for the limit map must go to zero. 
\end{remark}

\begin{definition}
Let $B=(V,E)$ be a Bratteli diagram. Let $m,n$ be distinct non-negative integers with $m<n$, and for each $i$, $m \leq i \leq n$, let $e_i$ be an edge in $E_i$ such that $r(e_j) = s(e_{j+1})$ for all $m \leq j < n$.  The ordered sequence $e_m,e_{m+1},\dots,e_n$ is a \emph{positively oriented path} in $B$, and the sequence $e_{n},e_{n-1},\dots,e_m$ is a \emph{negatively oriented path} in $B$. \end{definition}

Denote by $E_{m,n}$ the set of positively oriented finite paths connecting vertices in $V_m$ with vertices in $V_n$, and denote by $E_{n,m}$ the set of negatively oriented finite paths connecting vertices in $V_n$ with vertices in $V_m$.

\begin{definition}
\label{def:telescoping}
Let $B=(V,E)$ be a Bratteli diagram and let $$0 = m_0 < m_1 < m_2 <\cdots$$ be an increasing sequence in $\mathbb{N}$ . For $l \in \mathbb{N}$ and $k \in \{0,\dots, l-1\}$, we define another Bratteli diagram $B^{\prime} = (V^{\prime},E^{\prime})$ by setting $V^{\prime}_0 = V_0$, $V^{\prime}_n = V_{m_n}$ for all $n \in \mathbb{N}$, and $E_n^{\prime}$ is identified with $E_{m_{n-1},m_n}$.  Then $B^{\prime}=(V^{\prime},E^{\prime})$ is called the \emph{telescoping} of $B$ to $\{m_n\}_{n\geq 0}$.
\end{definition}
Whenever we write \emph{up to telescoping}, we will mean up to collapsing some levels of the diagram (and thus just shifting indices), since this is what happens when we telescope using the above definition.

With the notation used in the definition of telescoping, the incidence matrices $F_n^{\prime}$ for $B^{\prime}=(V^{\prime},E^{\prime})$ are given by $$F_n^{\prime} = F_{m_n} F_{m_n -1} \dots F_{m_{n-1}+1}.$$

\subsection{Bi-infinite Bratteli diagrams}
\label{ss:biinfiniteBratteliDefinitions}
We now introduce bi-infinite Bratteli diagrams. Similar objects have been considered before in \cite[\S 1.2.1]{bufetov:limitVershik}, and perhaps even before that, but we are unaware of any use prior to \cite{bufetov:limitVershik}.
\begin{definition} 
\label{def:biinfiniteBratteliDiagram}
A \emph{bi-infinite Bratteli diagram} $\mathcal{B}=(\mathcal{V},\mathcal{E})$ is an infinite directed graph together with partitions of the vertex set $\mathcal{V}$ and edge set $\mathcal{E}$ of the graph into countable unions of pairwise disjoint nonempty countable sets 
$$\mathcal{V} = \bigsqcup_{i \in \mathbb{Z}} \mathcal{V}_i \hspace{1in}\textrm{ and } \hspace{1in}\mathcal{E} = \bigsqcup_{i\in \mathbb{Z}\backslash \{0\}}\mathcal{E}_i$$
with associated range and source maps $r,s:\mathcal{E}\rightarrow\mathcal{V}$ such that $s(\mathcal{E}_i) = \mathcal{V}_{i-1}$ and $r(\mathcal{E}_i) = \mathcal{V}_i$ for all $i \in \mathbb{N}$ and $s(\mathcal{E}_i) = \mathcal{V}_{i}$ and $r(\mathcal{E}_i) = \mathcal{V}_{i+1}$ for all $i < 0$.
\end{definition}

\begin{conv}
We will henceforth use uppercase letters in calligraphy font $\mathcal{B},\mathcal{V}, \mathcal{E}, \mathcal{F}$ to refer to bi-infinite Bratteli diagrams, while we will use regular uppercase letters $B,V,E,F$ to refer to ``singly-infinite" Bratteli diagrams.  If an adjective ``bi-infinite" or ``singly-infinite" is not explicitly stated, we will rely on font to make it clear which type of diagram we are referring to.
\end{conv}

\begin{definition} For a bi-infinite Bratteli diagram $\mathcal{B}=(\mathcal{V},\mathcal{E})$, the \emph{positive half} of $\mathcal{B}$, denoted $\mathcal{B}^+=(\mathcal{V}^+,\mathcal{E}^+)$, is the subgraph of $\mathcal{B}$ corresponding to the vertices in $\mathcal{V}_i$ for $i \geq 0$ and the edges in $\mathcal{E}_i$ for  $i > 0$.  The \emph{negative half} of $\mathcal{B}$, denoted $\mathcal{B}^-=(\mathcal{V}^-,\mathcal{E}^-)$, is the subgraph of $\mathcal{B}$ corresponding to the vertices in $\mathcal{V}_i$ for $i \leq 0$ and the edges in $\mathcal{E}_i$ for $i < 0$. 
\end{definition}

For $m<n$, denote by $\mathcal{E}_{m,n}$ the set of positively oriented finite paths connecting vertices in $\mathcal{V}_m$ with vertices in $\mathcal{V}_n$, and denote by $\mathcal{E}_{n,m}$ the set of negatively oriented finite paths connecting vertices in $\mathcal{V}_n$ with vertices in $\mathcal{V}_m$.   An (unoriented) \emph{infinite  path} $x$ in $\mathcal{B}$ consists of a map $x:\mathbb{Z}\backslash \{0\} \rightarrow \mathcal{E}$ such that $x(i) \in \mathcal{E}_i$  and $r(x(i))=s(x(i'))$ for all $i \in \mathbb{Z}$, where $i'$ is the successor of $i$ in $\mathbb{Z}\backslash\{0\}$.  Denote the set of (unoriented) infinite paths in $\mathcal{B}$ by $X_{\mathcal{B}}$.  For $x \in X_{\mathcal{B}}$, we will  use $x_i$ to denote the edge $x(i)$. 

The set $X_{\mathcal{B}}$ has a natural product structure: let $B^+ = (V^+,E^+)$ be the Bratteli diagram defined by the positive part $\mathcal{B}^+$ of $\mathcal{B}$ and $B^-=(V^-,E^-)$ be the Bratteli diagram defined by the negative part $\mathcal{B}^-$ (interchanging the role of the source and range maps when we switch between  $\mathcal{B}^-$ to $B^-$ since we must switch between the indices taking values in $-\mathbb{N}$ and $\mathbb{N}$).   Since $|V^+_0| = |V^-_0|$, 
 we can identify each vertex in $V^+_0$ with one in $V^-_0$ and make the identification
\begin{equation}
\label{eqn:prodStr}
X_\mathcal{B} = \{(x,y) \in  X_{B^+} \times X_{B^-} : s(x(0)) = s(y(0))\}
\end{equation}
since $V^+_0 = V^-_0$.
  
A bi-infinite Bratteli diagram $\mathcal{B}$ formed from two Bratteli diagrams $B^+$ and $B^-$ in this way, for some choice of a bijection between $V_0^+$ and $V_0^-$, is called \emph{a welding} of $B^+$ and $B^-$.   If $(B^+,\leq^+_{r,s})$ and $(B^-,\leq^-_{r,s})$ are both fully ordered Bratteli diagrams with $|V^+_0|=|V^-_0|$, there is a canonically chosen welding of $B^+$ and $B^-$ determined by the partial orders $\leq_s^{\pm}$: since the vertices of $V^{\pm}_0$ are totally ordered by $\leq_s^{\pm}$, we identify each vertex in $V_0^+$ with the vertex in $V_0^-$ that has the same relative place in the orders (i.e. the vertex in $V_0^+$ that is the greatest with respect to $\leq_s^+$ is identified with the vertex in $V_0^-$ that is the greatest with respect to $\leq_s^-$, etc.).  In this case, we call the welding of $B^+$ and $B^-$ determined by $\leq_s^{\pm}$ \emph{the} welding: 

\begin{definition}
The bi-infinite Bratteli diagram $\mathcal{B}$ that is the welding of two fully ordered Bratteli diagrams $(B^+,\leq^+_{r,s})$ and $(B^-,\leq^-_{r,s})$ according to the identifications determined by $\leq_s^{\pm}$ is \emph{the welding} of $B^+$ and $B^-$ and we will denote it by $\mathcal{B}=\mathcal{B}(B^+,B^-)$.
\end{definition}

When welding two diagrams $B^+$ and $B^-$, the $0^{th}$ level vertices of both $B^+$ and $B^-$  ``fuse" to into the vertex set $\mathcal{V}_0$ of $\mathcal{B}=\mathcal{B}(B^+,B^-)$, the $i^{th}$ level vertices of $B^+$, for $i \in \mathbb{N}$, are identified with the $i^{th}$ level vertices of $\mathcal{B}$, and the $i^{th}$ level vertices of $B^-$, for $i \in \mathbb{N}$, are identified with the $(-i)^{th}$ level vertices of $\mathcal{B}$. The edges in $\mathcal{E}^\pm_i$ are identified with those in $E^\pm_{\pm i}$ for $i\in\mathbb{N}$ while the range and source maps of $\mathcal{B}^-$ are reversed for the negative part: for $e\in \mathcal{E}^-_i, v \in \mathcal{V}_i^-, v'\in\mathcal{V}_{i+1}^-$, we have $r(e) = v'$ and $s(e) = v$ if and only if $r(e)\in V_{-i}$ and $s(e)\in V_{i-1}$.

\begin{definition} \label{def:fullyOrdered}
A \emph{fully ordered} bi-infinite Bratteli diagram $(\mathcal{B},\leq_{r,s})$ is a Bratteli diagram $\mathcal{B}=(\mathcal{V},\mathcal{E})$ together with partial orders $\leq_r$ and $\leq_s$ on $\mathcal{E}$ so that edges $e,e^{\prime}$ are comparable under $\leq_r$ if and only if $r(e)=r(e^{\prime})$ and are comparable under $\leq_s$ if and only if $s(e)=s(e^{\prime})$.
\end{definition}

\begin{remark}
\label{rem:orders}
The welding $\mathcal{B}=\mathcal{B}(B^+,B^-)$ of two fully ordered Bratteli diagrams $(B^+,\leq^+_{r,s})$ and $(B^-,\leq^-_{r,s})$ is itself a fully ordered bi-infinite Bratteli diagram. Since the range and source maps are reversed for the negative part $\mathcal{B}^-$ when welding two diagrams $B^+,B^-$, the orders on the negative part of $\mathcal{B}$ are also reversed: the orders $\leq^-_r$ and $\leq^-_s$ at $v\in V^-_i$ become the orders $\leq_s$ and $\leq_r$, respectively, at $v\in \mathcal{V}_{-i}$.
\end{remark}

The definition of the incidence matrices $\mathcal{F}_i$ for Bratteli diagrams generalizes to the case of bi-infinite Bratteli diagram. In particular, when welding two Bratteli diagrams $B^+, B^-$ with matrices $F^+_i, F^-_i$ to obtain $\mathcal{B}(B^+,B^-)$, the matrices $\mathcal{F}_i$ are $\mathcal{F}_i = F_i^+$ for $i>0$ and $\mathcal{F}_i = (F^-_{-i})^T$ for $i<0$. The notion of telescoping also extends to bi-infinite Bratteli diagrams: for any sequence $\{m_n\}_{n\in\mathbb{Z}}$ with $m_0 = 0$ and $m_i < m_j$ if and only if $i<j$, the telescoping of $\mathcal{B}$ to $\{m_n\}$ is obtained by telescoping the positive and negative parts of $\mathcal{B}$, respectively, to the positive and negative parts of $\{m_n\}$.

\begin{definition} 
\label{def:probweightedbiinfinite}
A \emph{probability weighted} bi-infinite Bratteli diagram is a bi-infinite Bratteli diagram $\mathcal{B}=(\mathcal{V},\mathcal{E})$ together with a pair of weight functions $w^+:\mathcal{V}_0 \cup \mathcal{E}^+ \rightarrow (0,\infty)$, $w^-:\mathcal{V}_0 \cup \mathcal{E}^- \rightarrow (0,\infty)$ such that 
\begin{enumerate}
\item \label{it:it1weightedbiinfinite} $w^+$ is a probability weight function for $\mathcal{B}^+=(\mathcal{V}^+,\mathcal{E}^+)$, 
\item $w^-$ is a finite weight function for $\mathcal{B}^-=(\mathcal{V}^-,\mathcal{E}^-)$,
\item \label{item3probweightbiinfinite} $$\sum_{v \in \mathcal{V}_0} w^+(v) \cdot w^-(v) = 1$$
\end{enumerate}
\end{definition}

\begin{remark}
 Definition \ref{def:probweightedbiinfinite} involves a choice of normalization; we chose to make the $w^+$ weight a probability weight, while only requiring that the $w^-$ weight be finite and satisfy condition \ref{item3probweightbiinfinite}.  We will see in Section \S \ref{sec:Dictionary} that condition \ref{item3probweightbiinfinite} means that the associated flat surface has area $1$.  
 \end{remark}

\begin{definition} \label{def:diagram}
A \emph{diagram}  is a bi-infinite, fully-ordered, probability weighted Bratteli diagram $(\mathcal{B},w^\pm, \leq_{r,s})$ .
\end{definition}

\section{Cutting and Stacking}
\label{sec:CAS}

Cutting and stacking is a basic tool in ergodic theory used to construct infinite I.E.T.s.  This technique is described in, for example, \cite{AOW} and \cite{SilvaBook}.  We will review the cutting and stacking technique here.  

A cutting and stacking transformation $T$ will be defined by constructing a sequence of maps $T_0, T_1, T_2, \dots $ on subsets of the real line such that for all $i$ $$\textrm{domain}(T_i) \subseteq \textrm{domain}(T_{i+i})$$ and $$T_{i+1} |_{\textrm{domain}(T_i)} = T_i.$$  We set $$\textrm{domain}(T) = \bigcup_i \textrm{domain}(T_i)$$ and then define $T$  to be the pointwise limit of the maps $T_i$.  

We will always require that the range and domain of a cutting and stacking map be equal except for countably many points, i.e. there exist countable sets $P$ and $P^{\prime}$ such that $$\textrm{domain}(T) \setminus P = \textrm{range}(T) \setminus P^{\prime}.$$  

Associated with each map $T_i$ is a {\bf stack} $S_i$ consisting of a finite number $c_i \in \mathbb{N}$ of columns $C_{i,1},\dots,C_{i,c_i}$.  A {\bf column} $C_{i,j}$, consists of a finite number $h_{i,j}$ of open subintervals $I_{i,j,1},\dots,I_{i,j,h_{i,j}}$ of the real line, all of equal, finite measure.  The intervals $I_{i,j,k}$, for $k \in \{1,\dots,h_{i,j}\}$, are called the {\bf levels} of column $C_{i,j}$.  We require that for every $n$, all levels of all columns of the stack $S_n$ are pairwise disjoint.   We think of the levels of a column $C_{i,j}$ as being ``stacked"  with level $I_{i,j,1}$ at the bottom of the column and level $I_{i,j,(h_{i,j})}$ at the top of the column.  

The domain of the map $T_i$ is the union of all levels of all columns of the stack $S_i$ except for the top levels of the columns of $S_i$.  That is, 
$$\textrm{domain}(T_i) = \bigsqcup_{j=1}^{c_i} \bigsqcup_{k=1}^{(h_{i,j})-1} I_{i,j,k}.$$ 
For any point $x \in \textrm{domain}(T_i)$, we define $T_i(x)$ to be the point directly ``above" $x$ in the stack.  In other words, if $x$ is a point in the level $I_{i,j,k}$, then $T_i(x)$ is the point $y$ in the level $I_{i,j,k+1}$ such that $\lambda([a_{i,j,k},x])=\lambda([a_{i,j,k+1},y])$, where $\lambda$ denotes the Lebesgue (or other) measure and $a_{i,j,k}$ is the left endpoint of the interval $I_{i,j,k}$.  (The reason $\textrm{domain}(T_i)$ does not include the top levels of $S_i$ is because there are no levels in the stack $S_i$ above the top levels for $T_i$ to map points into.) 

The stacks $S_i$ (and thus the transformations $T_i$) are defined inductively.  The initial data that defines a cutting and stacking map is the stack $S_0$ along with ``rules" for the inductive steps, specifying how to obtain each stack $S_{i+1}$ from stack $S_i$ for each $i$.  The ``rules" consist of three types of moves. 

The first type of move is ``cutting" columns into finitely many subcolumns of specified positive widths.   We take these intervals to be open, and specify that the endpoints of these open intervals are not in the domain of subsequent maps $T_j$ for $j > i$.  For example, a column may be cut into two subcolumns of equal width (measure).  To do this, divide each level of the column into two open subintervals of equal width -- a ``left half" and a ``right half."  Now, all the ``left halves" form a subcolumn (keeping the stacking order of the levels) and all the ``right halves" for a subcolumn (with the same order).  

The second type of move is adding {\bf spacers}.  The $i^{th}$ step spacers are open intervals in $\mathbb{R}$ which are disjoint from the union of all levels in $S_i$ (and disjoint from each other), and which are added above a subcolumn of $S_i$ with the same width as that of the spacer.  Finitely many  spacers may be added (in a specified order) to the top of any subcolumn.  The ``rules" would specify which subinterval(s) in $\mathbb{R}$ is (are) ``stacked" above which subcolumn. 

The third type of move is stacking (sub)columns (possibly containing $i^{th}$-step spacers) of $S_i$ to form the columns of $S_{i+1}$.  (For example, if a column is cut into two subcolumns of equal width, we could stack the left subcolumn under the right subcolumn.  The resulting column, which is a column of $S_{i+1}$, is half as wide and twice as tall as the original column.  The left half of the top level of the original column is no longer a top level of the new column, and so is in $\textrm{domain}(T_{i+1})$).  We require that (sub)columns which are stacked on top of each other have equal width. 

\subsubsection{Truncating cutting and stacking processes}
\label{ss:truncating}

The cutting and stacking process may be viewed as taking the limit of a sequence of periodic maps.  The domain of each map $T_i$ is the union of all levels except the top level of the corresponding stack $S_i$.  For a fixed $i\in\mathbb{N}$, it is possible to extend $T_i$ to a homeomorphism $\widetilde{T_i}$ of the union of \emph{all} the levels of the stack $S_i$ by specifying that $\widetilde{T_i}$ maps the top level of $S_i$ to the bottom level of $S_i$ in an isometric, orientation-preserving way.  The map $\widetilde{T_i}$ is periodic; the period under $\widetilde{T_i}$ of a point in $S_i$ is the height of $S_i$.    

For a typical cutting and stacking construction, the width of a level in the stack $S_i$ converges to $0$ as $i$ goes to $\infty$, so the limit map $T$ is defined up to a set of measure $0$.  However, we will want to consider cutting and stacking processes in which only finitely many stacks $S_1,\dots,S_N$ are defined, and the width of a level in $S_N$ is nonzero.  We consider this case to be the same as the case in which infinitely many stacks $S_i$ are defined but $S_m=S_n$ for all $m,n > N$ for some $N \in \mathbb{N}$.  Thus, throughout the paper, we will adopt the convention that in this case the ``limit" map determined by the cutting and stacking process is the periodic map $\widetilde{T_N}$.  

\begin{remark}
In Section \S \ref{sec:Dictionary}, we will make use of correspondence between the adic map on Bratteli diagrams and cutting and stacking maps. Beyond some finite level of the Bratteli diagram $B$, all infinite paths in a periodic component of $X_B$ merge since the tail equivalence class is finite. The restriction of the adic map to this periodic component of $X_B$ is not a priori defined on the maximal path in this component, but admits a natural extension that sends the maximal path in the periodic component to the minimal path in the periodic component.  (Compare with Remark \ref{rem:extendingAdicMap}.) Therefore, we want the analogous cutting and stacking map - a map associated to a finite tower - to send the top level of the tower to the bottom level of the tower. Thus, a periodic component of $X_B$ will be associated to a periodic cutting and stacking map of the form $\widetilde{T_N}$ for some $N$.  
 \end{remark}

\section{The Dictionary}
\label{sec:Dictionary}
In this section we present a construction that associates a flat surface to a diagram, and develop a dictionary between combinatorial objects related to diagrams and geometric properties of flat surfaces constructed from them. The dictionary is summarized in Table \ref{tab:dictionary} on page \pageref{tab:dictionary}.

\subsection{Interpreting a diagram as a flat surface}
\label{subsec:interpretingadiagram}

 The core idea of the technique is to interpret each ``half" of a diagram $(\mathcal{B},w^\pm,\leq_{r,s})$ as determining an interval exchange map (likely with infinitely many intervals), i.e., a piecewise isometry of a finite interval.  One of these interval exchange maps will determine the dynamics of a ``first return map" to a transversal of the vertical flow, and the other interval exchange map will determine the dynamics of the first return map to a transversal of a horizontal flow.  
We will divide our description of how to define this map into two steps: first we will describe how to construct a flat surface from a pair of interval exchange maps and a collection of rectangles, and second we will describe how to interpret a diagram as a pair of interval exchange maps together with a collection of rectangles.  

\subsubsection{Obtaining a flat surface from a pair of interval exchange maps and a collection of rectangles}
\label{subsubsec:projecting}
In this section we describe a way of creating a flat surface from a pair of interval exchange maps (possibly of infinitely many intervals). We note that this construction is very general and that one should expect, in general, to obtain a flat surface of infinite genus. 

By \emph{rectangle} $R$, we mean a subset of $\mathbb{R}^2$ of the form $I_1 \times I_2$, where $I_1$ and $I_2$ are closed intervals.  Using this notation, we will refer to the Euclidean length of $I_1$ as the \emph{width} of $R$ and to the Euclidean length of $I_2$ as the \emph{height} of $R$.   

Fix $n \in \mathbb{N}$, and real numbers $b_1,b_2 >0$. Write the interval $[0,b_1]$ as a union of $n$ intervals $X_1,\dots,X_n$ which overlap only at endpoints, i.e., $X_i\cap X_{i+1}$ consists of one point, and write the interval $[0,b_2]$ as a union of $n$ intervals $Y_1,\dots,Y_n$ which also overlap only at endpoints. Now for each $i=1,\dots,n$, define the rectangle $R_i = X_i \times Y_i$.  Thus, $R_1,\dots,R_n$ is a collection of $n$ rectangles, arranged diagonally in $\mathbb{R}^2$, whose widths sum to $b_1$ and whose heights sum to $b_2$.  (See Figure \ref{f:generalConstruction}.)

Let $T$ be an interval exchange transformation defined on the interval $[0,b_1]$ and let $S$ be an interval exchange transformation defined on the interval $[0,b_2]$.  Each point $\tilde{x} \in [0,b_1]$ is the $x$-coordinate of a unique point on the ``top" edge of one of the rectangles $R_1,\dots,R_n$ (unless $\tilde{x}$ belongs to an edge of an interval $X_i$), and we denote this point $top(x)$.  Denote by $bottom(\tilde{y})$ the unique point on the ``bottom" edge of one of the rectangles $R_1,\dots,R_n$ whose $y$-coordinate is $\tilde{y}$ (unless $\tilde{y}$ is on an edge of some interval $Y_i$). The functions $bottom$ and $left$ are defined analogously. More precisely, if we define the functions $\tau, \rho$ defined in the interior of the intervals $X_1,\dots, X_n$ and $Y_1,\dots, Y_n$, respectively, to be the functions such that $\tau(x)=i$ if and only if $x\in X_i$ and $\rho(y) =i$ if and only if $y\in Y_i$, then 
\begin{equation*}
\begin{split}
top(x) = \left( x, \sum_{i=1}^{\tau(x)} |Y_i|\right)\hspace{.5in}&\mbox{ and } \hspace{.5in} bottom(x) = \left( x, \sum_{i=1}^{\tau(x)-1} |Y_i|\right), \\
right(y) = \left( \sum_{i=1}^{\rho(y)} |X_i|, y\right) \hspace{.5in} &\mbox{ and }\hspace{.5in} left(y) = \left( \sum_{i=1}^{\rho(y)-1} |X_i|, y\right).
\end{split}
\end{equation*}

Define
\begin{equation*}
\begin{split}
 \Sigma_{t} &= \{ x \in [0,b_1] \mid T \textrm{ is not continuous at } x\}, \\
\Sigma_{r} &= \{ y \in [0,b_2] \mid S \textrm{ is not continuous at } y\}.
\end{split}
\end{equation*}
For each $i\in\{1,\dots,n\}$, let
\begin{equation*}
\begin{split}
G^+_i &= \bigcup_{x\in X_i - \Sigma_t}\,top(x)\cup bottom(T(x))  \\ 
G^-_i &= \bigcup_{y\in Y_i - \Sigma_r}\,right(y)\cup left(S(y)) 
\end{split}
\end{equation*}
and define
$$\Sigma = C_i \cup \partial\left(\bigcup_{i=1}^n R_i \right) \backslash \bigcup_{i=1}^n \left( G^+_i\cup G^-_i  \right)$$
where $C_i$ are all the corners of the rectangle $X_i\times Y_i$.

\begin{figure}[h!]
\begin{center}
  \includegraphics[width=.6\linewidth]{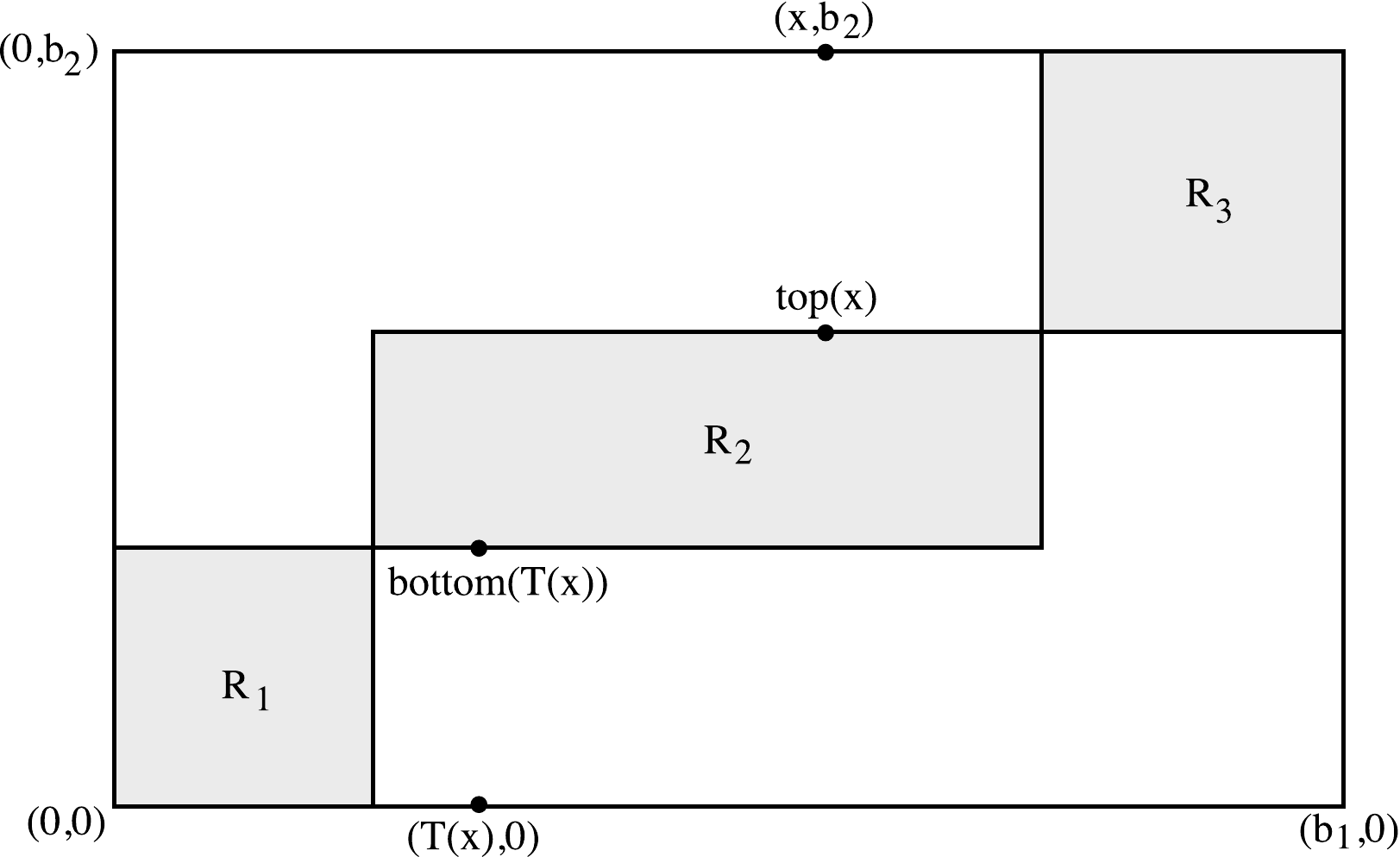}
\caption{Constructing a flat surface from  an interval exchange transformation $T$ and a set of three rectangles $R_1,R_2,R_3$. }
\label{f:generalConstruction}
\end{center}
\end{figure}

The flat surface associated to the pair $(T,S)$ and the collection of rectangles $R_1,\dots,R_n$ is the surface obtained by 
\begin{equation}
\label{eqn:surface}
\left(\bigcup_i R_i \setminus \Sigma\right) / \sim,
\end{equation}
where $\sim$ is the equivalence relation defined, for $x \in (X_i-(\Sigma^t\cup \partial X_i))$,
$$top(x) \sim bottom(T(x)) \hspace{.5in} \mbox{ and } \hspace{.5in} right(y) \sim left(S(y))$$
for $y \in (Y_i-(\Sigma^r\cup \partial Y_i))$. See Figure \ref{f:generalConstruction}. 

The associated holomorphic 1-form $\alpha$ on this surface is defined such that its vertical foliation coincides with lines locally of the form $x = const.$ and horizontal foliation locally of the form $y = const.$ when representing the surface as in (\ref{eqn:surface}).

\subsubsection{Interpreting a diagram as a pair of interval exchange maps and a collection of rectangles}
\label{subsec:constr}

Let $(\mathcal{B},w^\pm, \leq_{r,s})$ be a diagram.  We will use it to define a collection of rectangles and two interval exchange maps; we will then use the construction introduced in \S \ref{subsubsec:projecting} to construct a flat surface $S(\mathcal{B},w^\pm, \leq_{r,s})$ from the obtained interval exchange transformations.

We will define a collection of $c_0 = |\mathcal{V}_0|$ rectangles.  Let $v_1,\dots,v_{c_0}$ be the vertices in $\mathcal{V}_0$.   For each integer $i,$ $1 \leq i \leq c_0$, we will define a rectangle $R_i$ of width $w^+(v_i)$ and height $w^-(v_i)$, and we arrange these rectangles ``diagonally."  That is,
\begin{equation}
\label{eqn:rectangles}
R_i = \left[ \sum_{j<i} w^+(v_j), \sum_{j \leq i} w^+(v_j)] \right] \times \left[\sum_{j < i} w^-(v_j),\sum_{j \leq i} w^-(v_j)\right].
\end{equation}
  
We will describe how to obtain a cutting and stacking map from an weighted, fully ordered Bratteli diagram $(B,w,\leq_{r,s})$. We will then apply this construction to the positive half of $(\mathcal{B},w^\pm,\leq_{r,s})$ to determine the map which we will associate to the union of the horizontal sides of the rectangles, as well as apply it to the negative half of $(\mathcal{B},w^\pm,\leq_{r,s})$ to determine the map which we will associate to the union of the vertical sides of the rectangles. Recall that when considering the negative part $(\mathcal{B},w^\pm,\leq_{r,s})$ of a diagram as a Bratteli diagram $(B^-,w,\leq^-_{r,s})$, the orders $\leq_{r,s}$ are switched to obtain $\leq^-_{r,s}$ (see Remark \ref{rem:orders}).
 
We aim to interpret a weighted ordered Bratteli diagram $(B,w,\leq_{r,s})$ as combinatorial description of ``cutting and stacking" instructions. We wish to construct a measure-preserving map $\phi$ from $X_B$ to a real interval minus a countable set of points so that $\phi$ conjugates the adic map on $X_B$ to a cutting-and-stacking map on the interval.  A point in $X_B$ and its successor (which is determined by the partial order $\leq_r$) must be mapped by $\phi$ to a point and the point directly ``above" it in the cutting and stacking process.  However, this requirement does not determine a unique cutting and stacking process, since it does not, for example, specify whether a given subcolumn is to the left or right of the other subcolumns which comprise a column.   Thus, the order $\leq_r$ on the Bratteli diagram determines a family of measurably isomorphic cutting-and-stacking maps.  The partial order $\leq_s$ will be used to pick out a unique such map.  Namely, we will use $\leq_s$ to give the relative orders of subcolumns of a column, as well as the order of the $0^{th}$ level intervals.  

By the decomposition (\ref{eqn:decomposition}), it will suffice to describe the construction for minimal components and for periodic components of $X_B$.

Assume $B$ is minimal and let $I = [0,\sum_{v\in V_0}w(v)]$. We will define a family of injective maps $f_i:D_i \rightarrow I$, $D_i \subset I$, indexed by $i \in \mathbb{N}$, such that 
\begin{enumerate}
\item $D_i \subset D_{i+1}$ for all $i$, 
\item $\bigcup_i D_i = I \setminus \Sigma$ for some set $\Sigma \subset I$ of Lebesgue measure $0$, 
\item the restrictions $f_i |_{D_j} = f_j $ for all $j < i$.  
\end{enumerate}
We will then define the map $f:I \setminus \Sigma \rightarrow I$ to be the pointwise limit $f = \lim_{i \rightarrow n} f_i |_{I \setminus \Sigma}$. This will be an interval exchange transformation.

By our order $\leq_s$ we have an order on the level $V_0$. Therefore, for $v_i \in V_0$, $1\leq i \leq |V_0|$, we define the ``(stage 0) column over $v_i$" to be the interval 
\begin{equation}
\label{eqn:stage0}
J^0_i = \left(\sum_{j = 1}^{i-1} w(v_j), \sum_{j = 1}^i w(v_j)\right).
\end{equation}
Up to finitely many points, these intervals cover $I$. For each fixed $v_i \in V_0$, denote by $e_1^i,\dots,e_n^i$ the edges coming out of $v_i$ in increasing order with respect to $\leq_s$.  Partition the level 0 tower over $v_i$ into open subintervals $$J_1(v_i)=(j_0,j_1), \ \  J_2(v_i)=(j_1,j_2), \ \ \dots , \ \  J_n(v_i)=(j_{n-1},j_n)$$ with $\sum_{j < i}w(v_j) = j_0 < j_1 < ... < j_n = \sum_{j\leq i}w(v_j)$ and such that the Lebesgue measures of the intervals $J_1(v_i),...,J_n(v_i)$ are, respectively, $$w(e_1^i),w(e_2^i), \dots ,w(e_n^i).$$ Thus, each edge $e_j^i$ is associated to one subinterval of the column over $v_i$.  

For each vertex $v \in V_1$, we will form the ``(stage 1) column over $v$" as follows.  Let $e^{\prime}_1,\dots,e^{\prime}_m$ denote the edges that terminate at $v$ in increasing order with respect to $\leq_r$.  Stack the subintervals from the level $0$ columns associated with the edges $e^{\prime}_1,..,e^{\prime}_m$ in order, so that the subinterval associated with $e^{\prime}_1$ is the bottom of the stack, and the subinterval associated with $e^{\prime}_m$ is at the top of the stack.  

The domain $D_1$ of the map $f_1$ will be the union over all $v \in V_1$ of all but the top level of the stage 1 column over $v$. Because the weight function $w$ is compatible with $X_B$, all subintervals in a stack will have the same width.  The map $f_1$ is defined by mapping a point $x$ in a subinterval to the corresponding point in the subinterval directly above it.  
  
We define $D_k$ and $f_k$ by induction on $k$.  Assume the (stage $k$) columns over the vertices in $V_k$ have been defined.  For each vertex $v_i \in V_k$, denote by $e^i_1,\dots,e^i_n$ the edges coming out of $v_i$ in increasing order with respect to $\leq_s$.  Cut each level of the (stage $k$) column over $v_i$ into open subintervals such that the relative lengths of the subintervals (in increasing order from left to right) are, respectively, $$w(e^i_1),w(e^i_2), \dots ,w(e^v_n).$$  In this way, for each $j$ the edge $e^i_j$ is associated to the subset (a ``subcolumn") of the column over $v_i$ consisting of the $j^{th}$ subinterval of each level of the column over $v_i$.

For each vertex $v'$ in $V_{k+1}$ we will form the ``(stage $k+1$) column over $v'$" as follows.  Let $e^{\prime}_1,\dots,e^{\prime}_m$ denote the edges coming into $v'$ in increasing order with respect to $\leq_r$. Each edge $e^{\prime}_i$ is associated to a subset of a stage $n$ column.  Stack the subcolumns associated with the edges $e^{\prime}_1,\dots,e^{\prime}_m$ in order, so that the subcolumn associated with $e^{\prime}_1$ is on the bottom of the column and the subcolumn associated to $e^{\prime}_m$ is on the top of the column.  

The domain $D_{k+1}$ of the map $f_{k+1}$ is the union over all vertices $v' \in V_{k+1}$ of all but the top level of the (stage $k+1$) column over $v^{\prime}$. Because the Bratteli diagram is compatibly weighted, all levels of each stage $k+1$ column will have the same width.  The map $f_k$ is defined by mapping any point $x$ in a non-maximal level of any stage $k+1$ column to the corresponding point in the level immediately above it.  

Define 
\begin{equation}
\label{eqn:discontinuities}
\Sigma = \bigcap_{k=1}^{\infty} \textrm{TopLevels}(k)
\end{equation}
where $\textrm{TopLevels}(k)$ denotes the union over $v \in V_k$ of the top level of the (stage $k$) tower over $v$. Since $B$ is minimal, by (iii) in Definition \ref{def:weightfunctiondef}, and $\Sigma$ is countable, we have that $\Sigma$ has Lebesgue measure $0$. Note that the set $\Sigma$ is in bijection with $X_{max}$. The sets $\textrm{TopLevels}(k)$ are nested; for any $x \in I \setminus \Sigma$, there exists $n \in \mathbb{N}$ such that $N>n$ implies $x$ is in some non-top level of a stage $N$ tower.  Thus, $\lim_{n \rightarrow \infty}f_n(x)$ is well-defined for all $x \in I \setminus \Sigma$.  Thus, the pointwise limit function $f= \lim_{n \rightarrow \infty} f_n$ is well-defined on $I \setminus \Sigma$. Furthermore, $f$ is injective and Lebesgue measure-preserving.

Let us now assume that $B$ consists of a single periodic component according to (\ref{eqn:decomposition}), i.e., $|X_B| < \infty$. The finite set of paths of $X_B$ is ordered by the ordering $\leq_r$ in $r^{-1}(v^*)$, where $v^*\in V_k$ is the first vertex after which all paths in $X_B$ coincide. As such, there are $|X_B|$ open intervals of length $|X_B|^{-1}$, bijectively identified to paths starting at $V_0$ and ending at $v^*$ which are permuted by the map according to the order $\leq_r$. The interval corresponding to the maximal path in $X_B$ is mapped to the one corresponding to the minimal path. Therefore we have defined a periodic interval exchange transformation $f:I\backslash \Sigma \rightarrow I\backslash \Sigma$ of period $|X_B|$, where $\Sigma =  \{\frac{i}{|X_B|} : i\in \{0,\dots,|X_B| \} \}$.  Note that mapping the maximal path in a periodic component to the minimal path in that periodic component agrees with the convention established in subsection \S\ref{ss:truncating} of interpreting finitely many steps (or an infinite process with only finitely many nontrivial steps) of a cutting and stacking process as determining a periodic map.  

Let $(B,w,\leq_{r,s})$ be any fully ordered, weighted Bratteli diagram and assume that $w$ is a probability weight function. We can define an injective map $f:I\backslash \Sigma \rightarrow I$ by defining a map on each minimal component $X_M^i$ and periodic component $X_P^i$ as above. Since (\ref{eqn:decomposition}) is a decomposition into invariant subsets of the tail equivalence relation, the union of the maps for each component gives the map on $I\backslash \Sigma$ which corresponds to the cutting and stacking transformation defined on $I$ by weighted, fully ordered Bratteli diagram $(B,w, \leq_{r,s})$.

Let $(\mathcal{B},w^\pm,\leq_{r,s})$ be a probability weighted, fully ordered Bratteli diagram. By the construction above we have two interval exchange maps $T^\pm$ defined on full measure subsets of $I^+ = [0,1]$ and $I^-$ constructed as cutting and stacking transformations. Using the construction from \S \ref{subsubsec:projecting}, using these maps along with the rectangles (\ref{eqn:rectangles}), we can build a unique flat surface $S(\mathcal{B},w^\pm,\leq_{r,s})$ associated to $(\mathcal{B},w^\pm,\leq_{r,s})$.
 
 \begin{remark}
It follows from condition \ref{item3probweightbiinfinite} of Definition \ref{def:probweightedbiinfinite} that a surface constructed from a diagram has surface area $1$.   The adjective ``probability" in condition \ref{it:it1weightedbiinfinite} of the same definition implies that the sum of the areas of the rectangles is $1$. 
 \end{remark} 

\subsubsection{Conventions for drawing diagrams}
\label{subsec:conventions}
 
We now establish conventions for representing a diagram $(\mathcal{B},w^\pm,\leq_{r,s})$ as a 	picture, and we will adhere to these conventions throughout the paper.   Dots representing vertices in the same level of $\mathcal{B}$ will be drawn in a horizontal row.   Vertices in negative levels will be towards the ``top" of the picture, and vertices in positive levels will be towards the ``bottom" of the picture.  Vertices in $\mathcal{V}_0$ will be arranged from left to right in ascending order according to $\leq_s$.   Edges whose source is a vertex $v$ will be drawn coming out of $v$ in order from left to right according to $\leq_s$ (i.e., in a neighborhood immediately ``below" $v$ in the picture, edge $e_1$ will be to the left of $e_2$ if $e_1 \leq_s e_2$).  Edges whose range is a vertex $v$ will be drawn entering $v$ in order from left to right according to $\leq_r$ (i.e. in a neighborhood immediately ``above" $v$ in the picture, edge $e_1$ will be to the left of $e_2$ if $e_1 \leq_r e_2$).  Weights assigned to each vertex $v$ in $\mathcal{V}_0$ by $w^+$ (resp. $w^-$) will be written just below (resp. above) the dot in the picture corresponding to $v$.  Weights assigned to edges in $\mathcal{E}^+$ by $w^+$ and to edges in $\mathcal{E}^-$ by $w^-$ will be written next to those edges.  

\subsection{Example 2: The Chacon middle third transformation}
\label{subsubsec:chaconMiddleThird}
\begin{figure}[h!]
\begin{center}
  \includegraphics[width=5in]{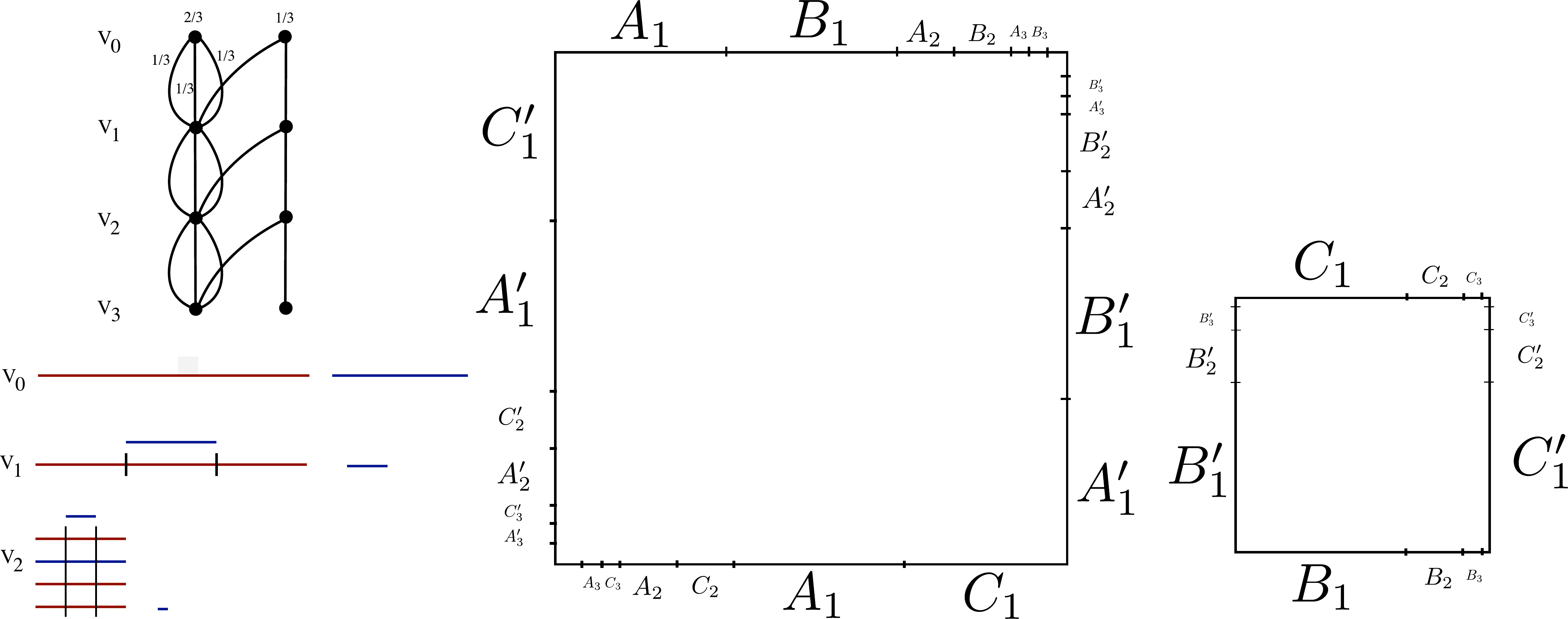}
\caption{\textbf{The Chacon middle third transformation.} Top left: The Bratteli diagram associated to it. Bottom left: The transformation as cutting and stacking. Right: The surface obtained through the construction in \S \ref{subsec:interpretingadiagram} obtained after welding two Bratteli diagrams to obtain a bi-infinite diagram. The resulting surface is of infinite genus.}
\label{f:Chaconpic}
\end{center}
\end{figure}
We work out the surface corresponding to the Chacon middle third example (\ref{subsubsec:chaconMiddleThird}) 
because it illustrates how to represent ``spacers" in a cutting and stacking process using a diagram. This will illustrate our unified point of view through the dictionary developed.

A well-known example of a transformation which is mild mixing but not light mixing is the Chacon middle third transformation \cite{chacon}. Figure \ref{f:Chaconpic} gives the first few steps of the Bratteli diagram and cutting-and-stacking construction for this map.  The union of all spacers used in the construction has measure $1/3$; at each stage, the unused spacers correspond to the rightmost vertex of each level in the Bratteli diagram.  Start with a single interval of with $2/3$; at each stage in the cutting-and-stacking process, cut the column in three equal subcolumns, put a spacer over the middle subcolumn, and stack the subcolumns from left to right (with left on the bottom).  

It should be noted that the tail equivalence relation is \emph{not} minimal. Indeed, the path $x_\infty$ in Figure \ref{f:Chaconpic} consisting of the right-most edge at every level is in its own tail equivalence class. This means that the decomposition (\ref{eqn:decomposition}) for this diagram has one minimal component and one periodic component consisting of $x_\infty^+$. As such the tail equivalence is \emph{not} uniquely ergodic: there is an ergodic probability measure supported at $x_\infty^+$ and another one supported on the minimal component of $X_B \backslash \{ x_\infty \}$ (this is proved in \S\ref{subsec:ChaconReprise}). The extra point $x_\infty^+$ in the Bratteli diagram in Figure \ref{f:Chaconpic} is somewhat artificial: by wanting to represent the Chacon middle third transformation through a Bratteli diagram, we have added an artificial point which is not actually part of the cutting and stacking construction used to define the Chacon middle third transformation. This is byproduct of encoding cutting and stacking transformations through a Bratteli diagram. From the point of view of the surface, since we have welded two Bratteli diagrams to construct it, the bi-infinite path $x_\infty$ (whose positive part is identified with $x^+_\infty$) corresponding to the right-most vertex at every level is seen to correspond to the infinite-angle singularity which is not part of the surface.
See \S \ref{subsubsec:HajianKakutani} for another phenomenon which occurs when encoding the use of spacers in a Bratteli diagram.

\subsection{Renormalization}
\label{subsec:renorm}
In this section we develop a mechanism which will serve as a renormalization tool for the vertical flow on flat surfaces $S(\mathcal{B},w^\pm,\leq_{r,s})$ constructed from diagrams $(\mathcal{B},w^\pm,\leq_{r,s})$. The spirit of the procedure is that as we deform the surface $S(\mathcal{B},w^\pm,\leq_{r,s})$ by the Teichm\"{u}ller deformation (\ref{eqn:teich}) we can perform a step of cutting and stacking on our surface and arrive at another surface which corresponds to the surface constructed from the shift of the diagram $(\mathcal{B},w^\pm,\leq_{r,s})$. This is summarized in Proposition \ref{prop:shift}. The 1-parameter deformation along with the renormalization maps can be seen as a generalization of the ``zippered rectangle flow'' of Veech and of Rauzy-Veech induction (see \cite{viana:iet}).

Let $(\mathcal{B},w^\pm,\leq_{r,s})$ be a diagram whose positive part is not completely periodic. Throughout this section, we will assume that, if the positive part of $\mathcal{B}$ is not aperiodic, then the invariant measure determined by $w^+$ assigns zero value to periodic components. Recall that we can define $h_v^k = w^-(v)$ for $v\in \mathcal{V}_{k}$ with $k>0$ using (\ref{eqn:heights}) and $w^+(v)$ using (\ref{eqn:extWeight}). Define
\begin{equation}
\label{eqn:heightsWidths}
\ell_v^k = w^+(v)\hspace{.3in} \mbox{ and } \hspace{.3in}  h_v^k=w^-(v)
\end{equation}
for $v \in \mathcal{V}_k$, $i \in \mathbb{N}$. Notice also that $\sum_{v \in V_0} \ell_v^0  = 1$ by assumption. 

We will define a sequence of maps 
$$\mathcal{R}_k:S(\mathcal{B},w^{\pm},\leq_{r,s}) \rightarrow S(\mathcal{B}^{\prime}_k,w_k^{\pm},\leq_{r,s}^k)$$
taking surfaces constructed from diagrams to other such surfaces. The data defining the surfaces will be related as follows. For $\mathcal{B} = (\mathcal{V},\mathcal{E})$, $\mathcal{B}'_k = (\mathcal{V}',\mathcal{E}')$ is obtained by shifting $\mathcal{B}$:  $\mathcal{V}'_i = \mathcal{V}_{i+k}$ and $\mathcal{E}'_i = \mathcal{E}_{i+k}$ along with their orders $\leq_{r,s}$ and $w^\pm_k = e^{\pm t_k}w^\pm$, where the $t_k$ belong to the sequence of renormalization times
\begin{equation}
\label{eqn:rTimes}
t_k \equiv -\log\left( \sum_{v\in \mathcal{V}_k}\ell^k_v \right) = -\log\left( \sum_{v\in \mathcal{V}_k}w^+(v) \right)
\end{equation}
for $k>0$. By (iii) in Definition \ref{def:weightfunctiondef}, we have that $t_k\rightarrow\infty$ if $\mathcal{B}^+$ is aperiodic. Moreover, up to telescoping, we can assume that $\inf_k(t_k-t_{k-1}) > 0$. The renormalized heights and widths are obtained from (\ref{eqn:heightsWidths}) and (\ref{eqn:rTimes}) by
\begin{equation}
\label{eqn:rHV}
\bar{h}^k_v = e^{-t_k}h^k_v\hspace{.2in}\mbox{ and } \hspace{.2in}\bar{\ell}^k_v = e^{t_k} \ell^k_v
\end{equation}
for any $v\in \mathcal{V}_k$. 

Let $(\mathcal{B},w^\pm, \leq_{r,s})$ be a diagram and $S(\mathcal{B},w^\pm, \leq_{r,s})$ be the flat surface constructed from it through the construction in \S \ref{subsec:constr}. Let $$S_t(\mathcal{B},w^\pm,\leq_{r,s}) = g_tS(\mathcal{B},w^\pm,\leq_{r,s})$$ be the surface obtained by deforming $S(\mathcal{B},w^\pm,\leq_{r,s})$ using the Teichm\"{u}ller deformation (\ref{eqn:teich}). Consider the surface $S_{t_1}(\mathcal{B},w^\pm, \leq_{r,s})$, for $t_1$ defined in (\ref{eqn:rTimes}).

Choose some vertex $v_i \in \mathcal{V}_0$. By our deformation of the surface, the interior of every deformed rectangle $g_{t_1}R_i$ in (\ref{eqn:rectangles}) corresponding to the vertex $v_i$ in $\mathcal{V}_0$ is isometric to $(0,\bar{\ell}_i^0)\times (0,\bar{h}^0_i)$. We cut the rectangle $g_{t_1}R_i$ (associated to the tower over $v_i$) into sub-rectangles of width $\bar{\ell}^0_{v_i}w(e)$ and height $\bar{h}^0_{v_i}$ using the order $\leq_s$ on $v_i$ for every $e$ is an edge with $s(e)=v_i$. Doing this for every vertex $v_i\in \mathcal{V}_0$ we have $|E_1|$ subrectangles corresponding to edges in $E_1$ which were obtained as subrectangles of the $R_j$.

Now we stack our sub-rectangles into $|\mathcal{V}_1|$ new towers using the orders given by the order in $r^{-1}(v)$ for each $v\in \mathcal{V}_1$. For some $v\in \mathcal{V}_1$, let $(e_1,\dots,e_n)$ be the ordered set of edges in $r^{-1}(v)$. For all $i\in \{1,\dots, |r^{-1}(v)|-1\}$, we identify the interior of the top of the sub-rectangle corresponding to the edge $e_i$ to the interior of the bottom edge of the sub-rectangle corresponding to the edge $e_{i+1}$.  Denote the surface obtained by the process of deforming and cutting and stacking described above as $DS(\mathcal{B},w^\pm,\leq_{r,s})$ and define
\begin{equation}
\label{eqn:renMap}
\mathcal{R}:S(\mathcal{B},w^\pm,\leq_{r,s})\longrightarrow DS(\mathcal{B},w^\pm,\leq_{r,s})
\end{equation}
to be the map taking one surface to the other by this process. We point out that the map between $S_{t_1}(\mathcal{B},w^\pm,\leq_{r,s})$ and $S(\mathcal{B}', w^\pm_1, \leq^1_{r,s})$ is an isometry: the cutting and stacking does not change the flat metric in any way.

\begin{definition}[Shifting]
The diagram $(\mathcal{B}',w^\pm_1,\leq'_{r,s})$ with $\mathcal{B}' = (\mathcal{V}',\mathcal{E}')$ is the \emph{shift} of $(\mathcal{B},w^\pm,\leq_{r,s})$ with $\mathcal{B} =(\mathcal{V},\mathcal{E})$ if it can be constructed as follows. $\mathcal{V}'_{i} = \mathcal{V}_{i+1}$ for all $i\in\mathbb{Z}$ and $\mathcal{E}'_i=\mathcal{E}_{i+1}$ for all $i\neq -1$. For $i=-1$, $\mathcal{E}'_{-1} = \mathcal{E}_1$. As such, there is a bijection $\sigma: \mathcal{B}'\rightarrow \mathcal{B}$ corresponding to this shift. 

Let $w^\pm_1:\mathcal{V}_0'\cup \mathcal{E}'\rightarrow (0,1)$ be the weight function obtained from $(\mathcal{B},w^\pm)$ as follows: for $v\in \mathcal{V}_0'$, $w^\pm_1(v) = e^{\pm t_1}w^\pm(\sigma(v))$, $w^+_1(e) = w^+(\sigma(e))$ for any $e\in \sigma(\mathcal{E}^+)\backslash \mathcal{E}'_{-1}$, and $w^-_1(e) = w^-(\sigma(e))$ for any $e\in \sigma(\mathcal{E}^-)$. Let $\leq_{r,s}'$ be defined on $\mathcal{B}'$ by $e\leq_{r,s}' f$ if and only if $\sigma(e) \leq_{r,s} \sigma(f)$ on $\mathcal{B}$.
\end{definition}

We will denote by $\sigma(\mathcal{B},w^\pm,\leq_{r,s})$ the shift of $(\mathcal{B},w^\pm,\leq_{r,s})$, and by $\sigma^k$ the process of shifting $k$ times. It is straightforward to check from the definition that if $(\mathcal{B},w^\pm,\leq_{r,s})$ is a diagram, then so is $\sigma(\mathcal{B},w^\pm,\leq_{r,s})$. 

\begin{proposition}[Functoriality]
\label{prop:shift}
Let $(\mathcal{B},w^\pm, \leq_{r,s})$ be a diagram. Then
$$S(\sigma(\mathcal{B},w^\pm, \leq_{r,s})) = \mathcal{R}(S(\mathcal{B},w^\pm, \leq_{r,s})),$$
where $\mathcal{R}$ is the map defined in (\ref{eqn:renMap}).
\end{proposition}
\begin{proof}
The construction of the map $\mathcal{R}$ was done through uniform deformation in addition to cutting and stacking. It is straightforward then by the definition of the shift $\sigma$ that $S(\sigma(\mathcal{B},w^\pm, \leq_{r,s}))$, through the construction described in \S \ref{subsec:constr}, has the same number of rectangles of the same widths as those in $\mathcal{R}(S(\mathcal{B},w^\pm, \leq_{r,s})).$ Moreover, the identifications on the top and bottom edges of $S(\sigma(\mathcal{B},w^\pm, \leq_{r,s}))$ coincide with those of $\mathcal{R}(S(\mathcal{B},w^\pm, \leq_{r,s}))$. It remains to show that the heights and left/right identifications of $S(\sigma(\mathcal{B},w^\pm, \leq_{r,s}))$ coincide with those of $\mathcal{R}(S(\mathcal{B},w^\pm, \leq_{r,s}))$.

Let $(B^-,\leq_{r,s})$ be the negative part of $(\mathcal{B},\leq_{r,s})$, indexed now by $\mathbb{N}\cup\{0\}$ so that the orders $\leq_{r,s}$ are reversed. Let $B'$ be the negative part of $\sigma(\mathcal{B},\leq'_{r,s})$, also indexed now by $\mathbb{N}\cup\{0\}$. The shifting operation $\sigma$ has the following effect: The $\mathcal{V}_1$ vertices, along with their orders, go to $\mathcal{V}_0$ vertices while the ones in  $\mathcal{V}_0$ go to $\mathcal{V}_{-1}$ vertices. This means that when we consider the negative part $B'$ of the shifted diagram, the $\leq_r$ orders at $\mathcal{V}_1$ become $\leq'_s$ orders at $V_0(B')$ and the $\leq_s$ orders at $\mathcal{V}_0$ become $\leq'_r$ orders at $V_1(B')$.

Let $Y$ be the (ordered) path space consisting of all oriented infinite paths starting from $V_1(B')$, the first level of vertices in $B'$. There is an order-preserving bijection between $Y$ and $X_{B^-}$. Consider the adic transformation $T:X_{B'}\rightarrow X_{B'}$. For $x = (x_1,x_2,\dots) \in X_{B'}$, suppose that $x_1$ is maximal (with respect to the order $\leq'_r$ used to define the adic transformation). Then the map at the point $x$ depends only on the tail starting at $x_2$, that is, on the path $(x_2,x_3,\dots)$. Therefore, the map here coincides with the adic transformation on $X_{B^-}$ through the order-preserving conjugacy. In other words, the cutting and stacking operations dictated by $B'$ (after the first stage) coincide with those of $B^-$.

Suppose for $x = (x_1,x_2,\dots)$, $x_1$ is not maximal. Then the adic transformation sends $x\mapsto (x_1+1,x_2,\dots)$. But the order $\leq'_r$ at $V_0(B')$ came from the $\leq_s$ order at $\mathcal{V}_0$, meaning that the order in which the columns are stacked in the first step of cutting and stacking for $B'$ comes from the order in which we cut the rectangles (\ref{eqn:rectangles}) for $(\mathcal{B},w^\pm, \leq_{r,s})$. Therefore, the geometry is compatible with the combinatorics of the first step of cutting and stacking. Since the cutting and stacking steps of $B'$  (after the first stage) agree with those of $B^-$, the left/right edge identifications for $S(\sigma(\mathcal{B},w^\pm, \leq_{r,s}))$ given by the limit map $f'$ (obtained from the cutting-and-stacking operations) used to define them agree with those obtained by deforming the surface and cutting and stacking, i.e., the ones for the surface $\mathcal{R}(S(\mathcal{B},w^\pm, \leq_{r,s}))$.
\end{proof}
\begin{table}[t]
\renewcommand\thetable{6.1}
\begin{tabular}{|p{.5\linewidth}| p{.5\linewidth}|}
\hline
Diagram $(\mathcal{B},w^{\pm},\leq_{r,s})$ & Flat Surface $S = S(\mathcal{B},w^{\pm},\leq_{r,s})$\\
\hline \hline
$|\mathcal{V}_0|$  & Number of rectangles used to draw $S$\\
\hline
Positive part of $(\mathcal{B},w^{\pm},\leq_{r,s})$ & geometry \& dynamics of vertical translation flow on $S$ \\
\hline
Negative part of $(\mathcal{B},w^{\pm},\leq_{r,s})$ & geometry \& dynamics of horizontal translation flow on $S$ \\ 
\hline
A vertex $v \in \mathcal{V}_k$, $k\in\mathbb{N}\cup\{0\}$ & A rectangular subset of $S$ of width $w^+(v)$ and height $w^-(v)$ obtained from $k$ steps of cutting and stacking.\\
\hline 
Weight functions $w^{\pm}$ & Transverse measures to vertical/horizontal foliations \\
\hline
Minimal/periodic components in the decomposition (\ref{eqn:decomposition}) of the positive (resp. negative) part of $(\mathcal{B},w^{\pm},\leq_{r,s})$ & Minimal/periodic components of the vertical (resp. horizontal) translation flow on $S$. \\ 
\hline 
Shift operator $\sigma$ & Teichm\"{u}ller deformation\\ 
\hline 
\end{tabular}
\hspace{.5cm}
\caption{A dictionary. }
\label{tab:dictionary}
\end{table}

The shift $\sigma$ on a Bratteli diagram yields a sequence of surfaces 
$$S_k(\mathcal{B},w^\pm, \leq_{r,s}) := S(\sigma^k(\mathcal{B},w^\pm, \leq_{r,s})) = (\mathcal{B}_k,w_k^\pm, \leq^k_{r,s})$$
 which are obtained as a shift on the starting Bratteli diagram $B$ and by rescaling the weights in the shifted diagram by the appropriate quantities. We will denote by $\mathcal{R}_k$ the map satisfying
\begin{equation}
\label{eqn:Rmaps} 
\mathcal{R}_k(S(\mathcal{B},w^\pm, \leq_{r,s})) = S(\sigma^k(\mathcal{B},w^\pm, \leq_{r,s})),
\end{equation}
which is obtained through composition of maps of the type defined in (\ref{eqn:renMap}). By Proposition \ref{prop:shift}, for each $k$ the surface $S_k(\mathcal{B},w^\pm, \leq_{r,s})$ is obtained by deforming $S(\mathcal{B},w^\pm, \leq_{r,s})$ for time $t_k$ and then cutting and stacking.

\section{Dynamical properties of the translation flow}
\label{sec:properties}

In this section we will exhibit flat surfaces whose translation flows exhibit a variety of phenomena which cannot occur for translation flows on flat surfaces of finite type. It is known that translation flows for compact flat surfaces are not mixing \cite{katok:mixing}, have zero topological entropy, and admit finitely many ergodic invariant measures \cite{veech:IETs}. We show that these limitations do not apply to flat surfaces of infinite type and finite area.  We show the existence of flat surfaces of infinite type and finite area whose translation flow is mixing (Corollary \ref{cor:mixing}  of \S \ref{subsec:range}), flat surfaces whose translation flows have positive topological entropy (\S \ref{subsubsec:entropy}), and translation flows which are minimal and admit uncountably many ergodic invariant measures (\S \ref{subsubsec:pascal}).  In fact, Theorem \ref{t:weCanGetAnyFlow} shows that any finite entropy, ergodic aperiodic flow on a finite measure Lebesgue space can be realized by the translation flow of a flat surface.  

\subsection{The range of dynamical behaviors of translation flows}
\label{subsec:range}
The main goal of this subsection is to establish the following theorem, which implies Theorem \ref{thm:model}.
\begin{theorem}
\label{t:weCanGetAnyFlow}
Let $\varphi_t$ be a measurable ergodic aperiodic flow on a finite measure Lebesgue space $(X,\mu)$ with finite entropy $h(\varphi_1)$.  Then there exists an ordered, weighted Bratteli diagram $(\mathcal{B},w^\pm,\leq_{r,s})$ with $|\mathcal{V}_0| = 2$ such that the vertical flow on $S(\mathcal{B},w^\pm,\leq_{r,s})$ is isomorphic to $\varphi_t$.
\end{theorem}
Theorem \ref{t:weCanGetAnyFlow} shows that translation flows on flat surfaces of infinite type exhibit a wide range of measure-theoretic dynamical properties (in marked contrast to the much more restricted range of behaviors possible for finite type flat surfaces). In particular, for example, if we take $\varphi_t$ to be the horocycle flow on the unit tangent bundle of a compact Riemann surface of constant negative curvature, we obtain the following result.

\begin{corollary}
\label{cor:mixing}
There exist translation flows on flat surfaces of infinite genus that are mixing.
\end{corollary}
Corollary \ref{cor:mixing} shows that there exist mixing translation flows on surfaces on infinite type, but gives no concrete example. We believe that the cutting and stacking transformation in \S \ref{subsubsec:staircase} is a good candidate to yield a mixing translation flow if suspended appropriately (see Conjecture \ref{conj:mixing}).

A \emph{flow built under a function} is given by a quadruple $(B, T, m, f)$, where $B$, the base, is a non-atomic Lebesgue space with measure $m$ (either finite or $\sigma$-finite), $T$ is a measure-preserving automorphism of $B$, and $f:B \rightarrow \mathbb{R}^+$ is an $m$-measurable map from $B$ to $\mathbb{R}^+$ with $\sum_{i=0}^{\infty} f(T^i(b)) = \infty$ for all $b \in B$ and $\int_B f dm = 1$. 
On the set $$\Omega = \{(b,x):b \in B, 0 \leq x < f(b)\}$$ a measure is given by the restriction of the completed product measure $m \cdot \lambda$ to $\Omega$, where $\lambda$ denotes Lebesgue measure.  The measure-preserving flow $\varphi_t$ is defined by 
$$\varphi_t(b,x) = \left(T^i(b),x+t-\sum_{j=0}^{i-1} f(T^j(b))  \right),  $$
where $i$ is the unique integer such that 
$$\sum_{j=0}^{i-1} f(T^j(b)) \leq x + t < \sum_{j=0}^i f(T^j(b)).$$

\begin{theorem}
 \label{t:CSmagic}
(\cite{AOW})  Any aperiodic measure-preserving automorphism of a unit measure Lebesgue space is measurably isomorphic to a cutting and stacking map on the unit interval with Lebesgue measure.
\end{theorem}

Ambrose characterized the flows isomorphic to those built under a constant function: those whose corresponding unitary group has eigenfunction with nonzero eigenvalue (\cite{Ambrose}). Combining Ambrose's theorem with Theorem \ref{t:CSmagic} immediately yields a characterization of flows on flat surfaces built from a single rectangle according to our construction.  

\begin{corollary}
Let $\varphi_t$ be a measure-preserving flow on a unit measure Lebesgue space.  Then the flow $\varphi_t$ is measurably isomorphic to the vertical flow on some surface $S(\mathcal{B},w^\pm, \leq_{r,s})$ corresponding to a fully-ordered finite-weighted diagram $(\mathcal{B},w^\pm, \leq_{r,s})$ with $|\mathcal{V}_0|=1$ if and only if the unitary group corresponding to $\varphi_t$ has an eigenfunction with nonzero eigenvalue.
\end{corollary}

Ambrose (\cite{Ambrose}) proved that if $\varphi_t$ is a measurable, measure-preserving ergodic flow on a Lebesgue space of finite measure, then there is a flow built under a function $(B, T, m, f)$, with $m(B) < \infty$ and $f$ bounded strictly away from $0$ and $\infty$, that is isomorphic to $\varphi_t$.  It is not clear how to apply this result to obtain \emph{flat} surfaces, since the construction in \S \ref{subsubsec:projecting} uses \emph{rectangles} -- i.e. intervals on which the height function $f$ is constant.  However, a
 stronger version of this theorem, due to Rudolph (\cite{Rudolph}), shows that the function $f$ can be chosen so that it takes on only two values, and the associated partition is generating:

\begin{theorem} (\cite{Rudolph}) 
\label{t:Rudolph}
Let $\varphi_t$ be a measurable ergodic aperiodic flow on a finite measure Lebesgue space $(X,\mu)$ with finite entropy $h(\varphi_1)$.  Fix $p,q>0$ such that $p/q$ is irrational and $h(\varphi_1)<\frac{2}{p+q}$.  Then there is a finite measure-preserving flow built under a function $(B, T,m,p \chi_{P}+q \chi_{P^c})$ with $m(B) = m(P\cup P^c) =1$,  that is isomorphic to $\varphi_t$, and $(P,P^c)$ is a generating partition for $T$ on $B$.
\end{theorem}

\noindent Theorem \ref{t:Rudolph} is almost enough to apply the technique of \S \ref{subsubsec:projecting} to build flat surfaces from rectangles, but we need a refinement to the theorem: we must be able to assume that the transformation of the base is an infinite interval exchange map and that the height function $p \chi_{P}+q \chi_{P^c}$ is constant on each interval.  

The proof in \cite{AOW} of Theorem \ref{t:CSmagic} shows how to construct a cutting a stacking model of the base transformation $T$.  We can slightly modify this process so that each level of the stacks constructed at each stage is contained entirely in $P$ or $P^c$, thus ensuring that the height function $p \chi_{P}+q \chi_{P^c}$ is constant over each interval.  


\begin{proposition}
\label{p:almostThere}
Let $\varphi_t$ be a measurable ergodic aperiodic flow on a finite measure Lebesgue space with finite entropy $h(\varphi_1)$.  Fix $p,q>0$ such that $p/q$ is irrational and $h(\varphi_1)<\frac{2}{p+q}$. Then there is a finite measure-preserving flow built under a function $(B,T,m,p \chi_{P}+q \chi_{P^c})$ such that $m(B) = m(P\cup P^c ) = 1$, $T$ is a cutting and stacking transformation on $[0,1]$, and $p \chi_P+q \chi_{P^c}$ is constant on each interval of the infinite interval exchange map $T$ on $B$ associated to the cutting and stacking transformation. 
\end{proposition}

To prove Proposition \ref{p:almostThere}, we must thus ensure that at each step of the cutting and stacking process described in the proof \cite{AOW} of Theorem \ref{t:CSmagic}, the stacks can be further subdivided into narrower stacks so that each level is contained in either $P$ or $P^c$.  We carry out the details of this process in the proof below.  

We first recall some definitions. For a measurable dynamical system $(X,\mathcal{A},T,\mu)$ and a partition $Q=\{Q_1,\dots,Q_k\}$ of $X$, the partition $Q$ is said to be \emph{generating} for $T$ if $\bigvee_{j=0}^{\infty} T^{-j}(Q) = \mathcal{A}$. 
Given a Rokhlin tower with height $n$ and base $B$, and a finite partition $\mathcal{P}$, \emph{purifying} the tower with respect to $\mathcal{P}$ means partitioning the base $B$ into sets $B_m$, $1 \leq m \leq M$ such that for each $0 \leq j < n$, the set $T^j(B_m)$ is contained in a single set of the partition $\mathcal{P}$.  The configuration $B_m,T(B_m),\dots,T^{n-1}(B_m)$ is called a \emph{pure} column with respect to $\mathcal{P}$.

\begin{proof}
By Theorem \ref{t:Rudolph}, we may assume that $\varphi_t$ is a flow built under a function $$(B,T,m,p \chi_{P}+q \chi_{P^c})$$ such that $m(B) = m(P\cup P^c) = 1$ and $(P,P^c)$ is generating for $T$.  Since $\varphi_t$ is aperiodic, $(B,\mathcal{A},m,T)$ is also an aperiodic measure-preserving dynamical system (here $\mathcal{A}$ denotes the $\sigma$-algebra $m$-measurable subsets of $B$).    

Following \cite{AOW}, we denote by $\mathcal{T}(L,n,\epsilon)$ a Rokhlin tower of height $n$ and base $L$ with residual set of measure $\epsilon$, so that the sets $$L,T(L),\dots,T^{n-1}(L)$$ are disjoint and $\mu(|\mathcal{T}|)=1-\epsilon$, where we define $|\mathcal{T}(L,n,\epsilon)| := \bigcup_{j=0}^{n-1}T^j(L)$.  
We fix a sequence of Rokhlin towers $\{ \mathcal{T}(L_j,n_j,\epsilon_j) \}_{n \in \mathbb{N}}$ 
such that $n_j \nearrow \infty$, $\epsilon \searrow 0$, and for each $j \in \mathbb{N}$, 
$$\left( |\mathcal{T}_{j+1}| \setminus (L_{j+1} \cup T^{n_{j+1}-1}(L_{j+1})) \right)\supset |\mathcal{T}_j|,$$ 
where $L_j$ denotes the base level of $\mathcal{T}_j$.  That such a sequence of Rokhlin towers exists is proven in \cite{AOW}.  

Define a sequence of partitions $\{\mathcal{P}_i\}_{i \in \mathbb{N}}$ of $B$ by setting $\mathcal{P}_1 = (P,P^c)$ and for $i >1$ setting
$$\mathcal{P}_i = \left( \bigvee_{j=1}^{i-1} \left(|\mathcal{T}_j|, |\mathcal{T}_j|^c \right) \right)
 \vee 
 \left( \bigvee_{j=0}^{i-1} \left(T^{-j}(P),T^{-j}(P^c)\right) \right) .$$ Each $\mathcal{P}_i$ is a finite partition of $B$, $\mathcal{P}_{i+1}$ is a refinement of $\mathcal{P}_i$, and $\bigvee_{i =1}^{\infty} \mathcal{P}_i = \mathcal{A}$ because $(P,P^c)$ is generating for $T$.  

We will now use the Rokhlin towers $\mathcal{T}_i$ to define a cutting and stacking process on a real interval $I = [0,1]$.  We will use $S_i$ to denote the $i^{\textrm{th}}$ stack in the process.  

To determine the stack $S_1$, we first purify the Rokhlin tower $\mathcal{T}_1$ with respect to $\mathcal{P}_1$.  In other words, partition $L_1$ into maximal sets $L_{1,1},\dots,L_{1,M_1}$ so that for all $0 \leq n \leq n_1$ and all $1 \leq k \leq M_1$, the set $T^n(L_{1,k})$ lies entirely within a single set of $\mathcal{P}_1$.  For each $1 \leq k \leq M_1$, form a column $C_{1,k}$ in $S_1$ of height $n_1$ and width $m(L_{1,k})$.  Thus $S_1$ consists of $M_1$ columns, each with $n_1$ levels, and there is a measure-preserving bijection (call the bijection $\phi_1$) between the set of levels of $S_1$ and the set of levels of the $\mathcal{P}_1$-pure columns comprising $\mathcal{T}_1$.  Furthermore, $\phi_1$ preserves the property of one level being immediately above another level, i.e. a level $A_1$ is immediately above a level $A_2$ in a $\mathcal{P}_1$-pure column of $\mathcal{T}_1$ if and only if $\phi_1(A_1)$ is immediately above $\phi(A_2)$ in a column of $S_1$.

To determine the stack $S_2$, we first purify the Rokhlin tower $\mathcal{T}_2$ with respect to $\mathcal{P}_2$.  This determines a partition of $L_2$ into maximal sets $L_{2,1},\dots,L_{1,M_2}$ so that for all $0 \leq n \leq n_2$ and all $1 \leq k \leq M_2$, the set $T^n(L_{2,k})$ lies entirely within a single set of $\mathcal{P}_2$.  Each $\mathcal{P}_2$-pure column in $\mathcal{T}_2$  contains ``blocks" from $|\mathcal{T}_1|$ (i.e. $n_1$ successive levels consisting entirely of points from $\mathcal{T}_1$), as well as levels contained in $|\mathcal{T}_1|^c$.  Since $\mathcal{P}_2$ is a refinement of $\mathcal{P}_1$, each of these``blocks" is a subset of a single $\mathcal{P}_1$-column in $\mathcal{T}_1$.  The condition that $$|\mathcal{T}_{2}| \setminus (L_2 \cup T^{n_2-1}(L_2)) \supset |\mathcal{T}_1|$$ ensures that each block consists of precisely $n_1$ levels -- a block does not run off the top or bottom of the $\mathcal{P}_2$-pure column.  
For each $1 \leq k \leq M_2$, we will form a column $C_{2,k}$ in $S_2$ of height $n_2$ and width $m(L_{1,k})$ as follows.  For each level of the $\mathcal{P}_2$-pure column over $L_{2,k}$ that is contained in $\mathcal{T}_1^c$, add a ``spacer" level of width $m(L_{2,k})$, and for each ``block" of levels which is a subset of a $\mathcal{P}_1$-pure column of $\mathcal{T}_1$, cut and use a width-$m(L_{2,k})$ subcolumn of the column in $S_1$ corresponding via $\phi_1$ to that $\mathcal{P}_1$-pure column of $\mathcal{T}_1$.  Thus $S_2$ consists of $M_2$ columns, each of height $n_2$, and there is a measure-preserving bijection between the set of levels of $S_2$ and the set of levels of the $\mathcal{P}_2$-pure columns comprising $\mathcal{T}_2$.  Furthermore, the bijection $\phi_2$ preserves the property of one level being immediately over another level.  We note also that $S_2$ is obtained from $S_1$ via cutting and stacking.  

This procedure is repeated inductively, yielding a sequence of stacks $S_i$ obtained via cutting and stacking such that $S_i$ consists of $M_i$ columns, each of height $n_i$, and there is a measure-preserving bijection $\phi_i$ between the set of levels of $S_i$ and the set of levels of the $\mathcal{P}_i$-pure columns comprising $\mathcal{T}_i$, and $\phi_i$ preserves the property of one level being immediately above another level.  Define $R$ to be the limiting map on $I$ defined by the cutting and stacking process.  

As a consequence of the fact that $\bigvee_i \mathcal{P}_i = \mathcal{A}$, the maps $\phi_i$ determine a measure-preserving embedding of $\mathcal{A}$ into the $\sigma$-algebra of Borel subsets of the interval $I$ with Lebesgue measure. The measure space $(B,\mathcal{A},m)$ is a standard Lebesgue space.  As such, each point $x \in B$ is the intersection of a decreasing sequence of sets $\{C_i\}_{i \in \mathbb{N}}$ where $C_i$ is a level of a pure column comprising $\mathcal{T}_i$.  The sequence $\{\phi_i(C_i)\}_{i \in \mathbb{N}}$ is a decreasing sequence of nested intervals, say $I_1,I_2,\dots$, whose intersection $\cap_{i=1}^{\infty} I_i$ has measure $0$, and is thus a point $a \in I$. We can thus define $\phi(x) = a$.  The map $\phi:B \rightarrow I$ is thus a measurable, measure-preserving isomorphism between the systems $(B,\mathcal{A},m,T)$ and $(I,\mathcal{B},\lambda,R)$, where $\mathcal{B}$ denotes the Borel $\sigma$-algebra on $I$.

The stacks $S_i$ have been constructed so that each level of each stack is in bijection with some level of a $\mathcal{P}_j$-pure column, for some $j$.  Define a function $f$ whose domain is the collection of levels of the columns $S_i$ for all $i$, by $f(A) = (p \chi_{P}+q \chi_{P^c})(\phi_j^{-1}(A))$ for a level $A$ in stack $S_j$.  (The function $f$ is well-defined because $p \chi_{P}+q \chi_{P^c}$ is constant on the set $\phi_j^{-1}(A)$, since it is a level of a $\mathcal{P}_j$-pure column, and hence contained in either $P$ or $P^c$.)   It follows immediately that the flows built under functions $(T,B,m,p \chi_{P}+q \chi_{P^c})$ and $(R,I,\lambda,f)$ are isomorphic, and $(R,I,\lambda,f)$ is the desired system.  
\end{proof}

\begin{lemma}
\label{l:relabel}
The sets $P$ and $P^c$ in the statement of Proposition \ref{p:almostThere} may be assumed to each consist of a single subinterval of $B$.  
\end{lemma}

\begin{proof}
It did not matter which specific subintervals of $I$ were used to form the stacks $S_i$ in the proof of Proposition \ref{p:almostThere}.  The inductive nature of the construction of the stacks $S_i$ provides an order on the intervals -- we simply order them according to the order in which we first use an interval when performing the cutting and stacking process.  Thus, we might, for example, take the intervals on which $f$ takes the value $p$ starting from the left side of $I$, and take the intervals on which $f$ takes the value $q$ starting from the right side of $I$, and work toward the middle.  
\end{proof}

The proof of Proposition \ref{p:almostThere} yields the construction of a cutting and stacking map from an measure preserving transformation. Each one of the steps of this construction can be recorded as the positive part of an ordered Bratteli diagram $\mathcal{B}$. Combining Proposition \ref{p:almostThere} and Lemma \ref{l:relabel} immediately proves Theorem \ref{t:weCanGetAnyFlow}.

\subsection{Examples of surfaces which exhibit certain dynamical properties}

\subsubsection{Suspensions of staircase transformations}
\label{subsubsec:staircase}
We mention a specific class of examples of cutting and stacking transformations which are known to be mixing.  Informally, a staircase transformation is a cutting and stacking transformation with a sequence $\{{r_n}\}$ of natural numbers with $r_n \rightarrow \infty$ as $n \rightarrow \infty$ so that at the $n^{th}$ stage the $n^{th}$ stack, which consists of a single column, is cut into $r_n$ subcolumns $\{c_{n,1},\dots,c_{n,r_n} \}$ of equal widths and $s_{n,j}$ spacers are added over the subcolumn $c_{n,j}$ before stacking the subcolumns in order from left to right. In this setup it is always assumed that the sequences $r_{n}$, $s_{n,j}$ are such that the limiting transformation is defined over a set of finite measure. The first staircase transformation explicitly shown to be mixing \cite{adams:mixing} was Smorodinsky's staircase, where $r_{n} = n+1$ and $s_{n,j} = j$, although Ornstein had proved that for a family of staircases defined over some parameter space, the typical staircase transformation is mixing \cite{ornstein:mixing}. In \cite{CreutzSilva} there is a characterization of mixing staircase transformations in terms of uniform convergence of certain averages of partial sums of the spacer sequence $\{s_{n,j}\}$. In particular, the authors show that all polynomial staircase transformations (roughly, $s_{n,j}$ is a polynomial $p_n(j)$, where the degree and coefficients of $p_n$ are bounded uniformly for all $n$ with some additional conditions) are mixing.

Since staircase transformations use spacers, they can be encoded in Bratteli diagrams $(B,w^+, \leq_{r,s})$ where each level has two vertices, one corresponding to the spacers, one for the rest (see \S \ref{subsubsec:chaconMiddleThird} for a concrete example). 
\begin{conjecture}
\label{conj:mixing}
Let $(\mathcal{B},w^\pm, \leq_{r,s})$ be a Bratteli diagram whose positive part encodes a mixing staircase transformation and the negative part is given by the constant $2\times 2$ identity matrix. Then for $\mathcal{V}_0 = \{v_1,v_2\}$, and $w^-(v_1) = p$, $w^-(v_2) = q$ with $p/q$ irrational, the vertical flow on $S(\mathcal{B},w^\pm, \leq_{r,s})$ is mixing.
\end{conjecture}

\subsubsection{The Pascal adic transformation}
\label{subsubsec:pascal}

\begin{figure}[h!] 
\begin{center}
  \includegraphics[width=4in]{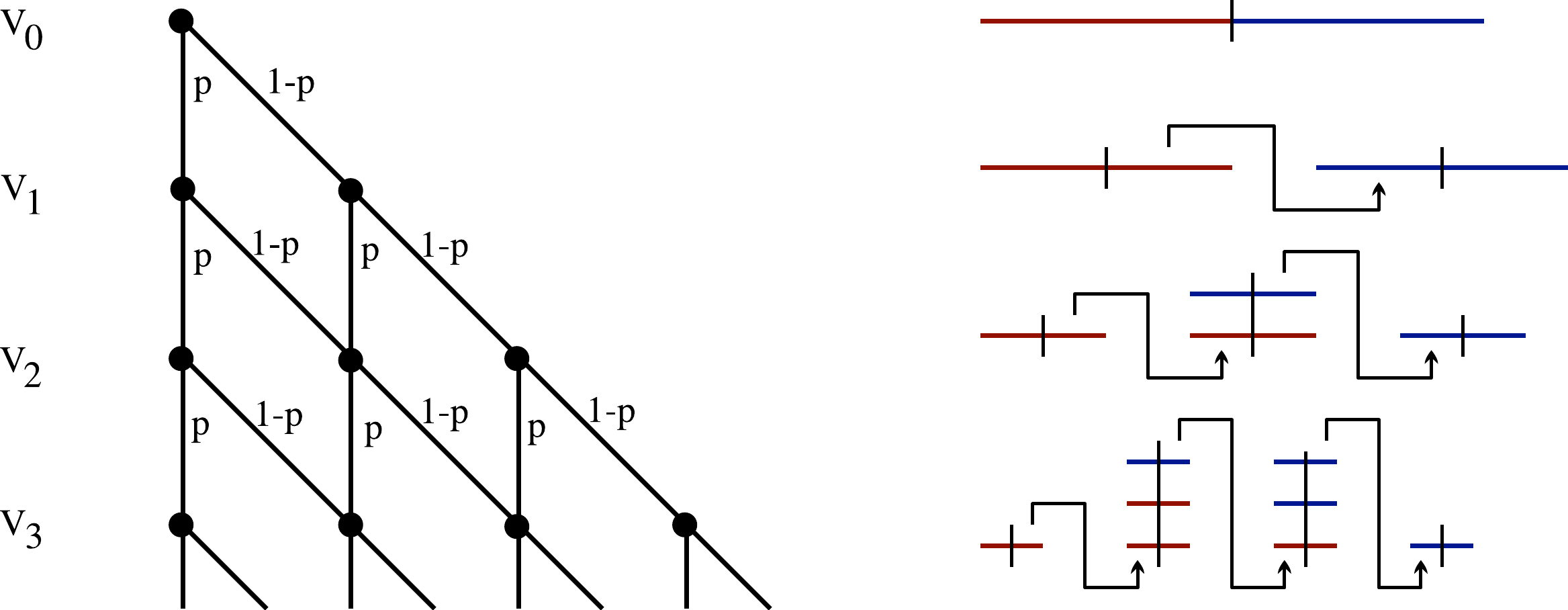}
\caption{The first parts of the Bratteli diagram and cutting-and-stacking steps for the Pascal adic transformation.}
\label{f:pascalpic}
\end{center}
\end{figure}

Let $c_i = i+1$ for each integer $i\geq 0$.  (Recall $c_i := |V_i|$.)  Denote the elements of $V_i$ by $v^i_1,\dots,v^i_{c_i}$ and order them accordingly.  
For each $i \in \mathbb{N}$, define the incidence matrix $F_i = [f^i_{v,w}]$ to be the $c_i \times c_{(i-1)}$ matrix with entries 
 
\[
 f^i_{v^{i-1}_j,v^{i}_k} =
  \begin{cases}
   1 & \text{if } k=j \text{ or } k=j+1 \\
   0      & \text{otherwise}
  \end{cases}
\]
 
Define the partial orders $\leq_s$ and $\leq_r$ as indicated in Figure \ref{f:pascalpic}, using the left-right ordering convention described in  \S \ref{subsec:conventions}.  The Vershik map $T:X \setminus X_{max} \rightarrow X \setminus X_{min}$ is called the Pascal adic transformation.  

For any $p \in (0,1)$, define a weight function $\omega_p$ on the set of edges by 

\[\omega_p(e) = 
\begin{cases}
p & \text{if } k=j \\ 
(1-p) & \text{if } k=j+1\\
\end{cases}
\]
where $j$ and $k$ are defined by $S(e)=v^{i-1}_j$ and $R(e)=v^i_k$.  

It is well known that the invariant ergodic Borel probability measures for the Pascal adic transformation are precisely the measures $\omega_p$ (which are called Bernoulli measures) (see \cite{MelaPetersen}).  In fact, the Pascal adic transformation is totally ergodic (every power $T^n$ is ergodic) for each $\omega_p$.   The Bernoulli measures $\omega_p$ are all mutually singular, and atomless.  Whether or not the Pascal adic transformation is weak mixing is an open question.  

For each $p \in (0,1)$, we may form a flat surface $S(\mathcal{B}, w^\pm, \leq_{r,s})$ where the positive part of $(\mathcal{B}, \leq_{r,s})$ coincides with the Pascal Bratteli diagram in Figure \ref{f:pascalpic}.  The geometry of this surface depends on $p$; the widths of the intervals in the interval exchange transformation used to identify the top and bottom of the rectangles are determined in the cutting and stacking process by cutting the stacks into subcolumns of relative widths $p$ and $1-p$.  Thus, the measure $\omega_p$ on the diagram corresponds with Lebesgue measure on the surface.  Each of the measures $\omega_q$ with $q \not = p$, corresponds to a measure on the surface that is absolutely continuous with respect to Lebesgue measure.  

Thus, the vertical translation flow on any such surface is minimal and has uncountably many ergodic, finite invariant measures. This phenomenon  (having infinitely many ergodic invariant probability measures) does not occur for translation flows on surfaces of finite type.

\subsubsection{Chaotic translation flows}
\label{subsubsec:entropy}
We will review the concept of \emph{independent cutting and stacking}, considered in \cite{shields} to construct cutting and stacking transformations with chaotic properties. Consider a collection of columns of intervals $C_0 = \{ C^1, \dots, C^q \}$, where each $C^i$ is a ordered collection of $h(C^i)$ intervals of the same width $w(C^i)$, each one stacked on top of the previous with the condition that $\sum_i h(C^i)w(C^i) = 1$. Denote by $C^i * C^j$ the stacking of column $C^j$ on top of column $C^i$ (for which it is necessary that $w(C^i) = w(C^j)$). We will denote $w(C_0) := \sum_i w(C^i)$.

Independent stacking is done as follows. Starting with the $q$ intervals (columns of height 1) $C_0 = \{C^1,\dots, C^q\}$ of the same width, cut each column $C^i$ into $2q$ subtowers $C^i_j$, $j=1,\dots,2q$, and stack them into $q^2$ towers $C_1 = \{ \hat{C}^{i,j} \}_{i,j\leq q}$ by
$$\hat{C}^{i,j} = C^i_j * C^j_{q+i} \hspace{.75 in} 1\leq i \leq q \hspace{.75 in}1\leq j \leq q.$$
Note that $\sum_l w(C^l)  = 2 \sum_{i,j} w(\hat{C}^{i,j})$. Iterating the independent cutting and stacking procedure, we obtain a sequence of collections of towers $C_k$ where a map is defined on all but the top levels. This map limits to a piecewise isometry of the unit interval $T(C_0)$, since the iterative construction depends on the starting tower $C_0$.

Let $\mathcal{P}^n$ be the partition of the unit interval making up $C_n$. The following summarizes the properties of transformations obtained through independent cutting and stacking.

\begin{theorem}[\cite{shields}]
The following hold for independent cutting and stacking.
\begin{enumerate}
\item $\mathcal{P}^n$ is a Markov partition for $T(C_0)$.
\item If two columns of some $C_k$ have height which differ by 1 or, more generally, if for some $k$ the greatest common divisor of the heights of the columns of $k$ is 1, then $T(C_0)$ is mixing, and therefore Bernoulli.
\item The topological entropy of $T(C_0)$ is $w(C_0) \log(q_0)$, where $q_0$ is the number of towers in $C_0$.
\end{enumerate}
\end{theorem}

Any independent cutting and stacking procedure can be written as a weighted, ordered Bratteli diagram $(B^+,w^-,\leq^+_{r,s})$. 
\begin{corollary}
Let $(\mathcal{B},w^\pm, \leq_{r,s})$ be a Bratteli diagram whose positive part is realized by an independent cutting and stacking transformation and negative part given by the $|C_0| \times |C_0|$ identity matrix. Then the vertical flow on $S(\mathcal{B},w^\pm, \leq_{r,s})$ has topological entropy $w(C_0) \log(q_0)$, where $q_0 = |C_0|$ is the number of towers in $C_0$.
\end{corollary}

\section{Ergodicity}  
\label{sec:erg}
In this section we give a criterion for ergodicity of the vertical flow on surfaces constructed from diagrams following the construction described in \S \ref{sec:Dictionary}. We then derive some corollaries and show some applications of the criterion through examples before giving the proof at the end of the section.

\subsection{A criterion} 
In this section we adapt Theorem \ref{thm:flat} to fit our setting and the dictionary developed in \S \ref{sec:Dictionary}. We use the same notation as in section \ref{subsec:renorm}. We will need the following notion to state our criterion.

\begin{definition}
Let $\mathcal{B} = (\mathcal{E},\mathcal{V})$ be a Bratteli diagram. Suppose that, for any $k\in\mathbb{N}\cup \{0\}$ there exists numbers $k^+,k^-\in\mathbb{N}$ such that for any $v,w\in \mathcal{V}_k,$ there exist $r_k = r_k(v,w)>0,$ paths $q_0, p_1, q_1,  \dots, q_{r_k}, p_{r_k}$ with $q_i\in \mathcal{E}_{k,k+ k^+}\cup \mathcal{E}_{k,k- k^-}$, $p_i\in \mathcal{E}_{k+ k^+,k} \cup \mathcal{E}_{k- k^-,k}$, such that 
\begin{enumerate}
\item $r(q_i) = s(p_{i+1})$ for $0\leq i < r_k$, 
\item $r(p_i) = s(q_i)$ for $0\leq i \leq r_k$.
\end{enumerate}

The \emph{tunneling distance function} is the function $\Delta^\pm_{\mathcal{B}}:k\mapsto (k^+,k^-)\in (\mathbb{N}\cup \{\infty\})^2$ defined, for $k\in\mathbb{N}$, as the minimal pair $(k^+,k^-)$ which satisfies the conditions above, when they exist. By minimal pair we mean that if $(k^+_*,k^-_*)$ also satisfy the above conditions, then $k^+\leq  k^+_*$ and $k^- \leq k^-_*$. When no such $k^+$ or $k^-$ exist, they take the value $\infty$. The individual components $\Delta^+_{\mathcal{B}}(k) = k^+$ and $\Delta^-_{\mathcal{B}}(k) = k^-$ are, respectively, the positive and negative tunneling distances.
\end{definition}

\begin{remark}
\label{rem:tunnel}
The tunneling distance function measures, in some sense, the connectivity of a Bratteli diagram at each level. This can be seen as follows: for any $k$ and for any two vertices $v,w$ at the level $\mathcal{V}_k$, we want to find the shortest way to connect them by paths. Since there are no paths connecting vertices in $\mathcal{V}_k$, we must find paths which go through other vertex levels in order to find a connecting path between vertex $v$ and $w$. The definition of $\Delta_{\mathcal{B}}^\pm(k)$ guarantees that for any two $v,w\in\mathcal{V}_k$ it suffices to look for paths between the levels $\mathcal{V}_{k-\Delta^-_{\mathcal{B}}(k)}$ and $\mathcal{V}_{k+\Delta^+_{\mathcal{B}}(k)}$ to achieve this.

By this measure of connectivity, it can be seen that if $\Delta^+_{\mathcal{B}}(k) = \infty$ for some $k\geq 0$, the Bratteli diagram consists, essentially, of a disconnected graph beyond level $k$. This is an obstruction to ergodicity for any adic transformation defined from the positive part of $\mathcal{B}$ since this would allow there to be disjoint, invariant sets of positive measure.
\end{remark}

An equivalent definition of the function $\Delta_\mathcal{B}^\pm$ is the following. Let $\{ \mathcal{F}_i \}_{i\in \mathbb{Z}\backslash \{0\}}$ be the matrices describing the transitions between levels of the Bratteli diagram $\mathcal{B}$. The value $k^+ = \Delta^+_\mathcal{B}(k)$ is the smallest natural number so that the product 
$$(\mathcal{F}_{k+k^+}\cdots \mathcal{F}_{k+1})^T(\mathcal{F}_{k+k^+}\cdots \mathcal{F}_{k+1})$$
has all non-zero entries (when it exists). Similarly, the value $k^- = \Delta^-_\mathcal{B}(k)$ is the smallest natural number so that the product 
$$(\mathcal{F}_{k}\cdots \mathcal{F}_{k-k^- + 1}) (\mathcal{F}_{k}\cdots \mathcal{F}_{k-k^- + 1})^T$$
has all non-zero entries.

\begin{remark}
\label{rem:primitivity}
The condition that for any $k > 0$, the value $\Delta^+_\mathcal{B}(k) < \infty$ is a weaker condition than the \emph{non-stationary primitivity condition} in \cite{fisher} which is that for any $k>0$ there exists an $m>0$ such that $\mathcal{F}_k\mathcal{F}_{k+1}\cdots \mathcal{F}_{k+m}$ has all non-zero entries. In particular, the condition that $\Delta^+_\mathcal{B}(k) < \infty$ for all $k>0$ is necessary for ergodicity (see Remark \ref{rem:tunnel}) and is satisfied in some cases by the non-stationary primitivity condition. For example, the Chacon transformation in \S \ref{subsubsec:chaconMiddleThird} satisfies $\Delta^+_\mathcal{B}(k) = 1$ for all $k\geq 0$ while fails to satisfy the primitivity condition in \cite{fisher}.
\end{remark}

\begin{theorem}
\label{thm:erg}
Let $(\mathcal{B},w^\pm, \leq_{r,s})$ be diagram with $\Delta_\mathcal{B}^\pm(k)<\infty$ for all $k \geq 0$ and such that, in the decomposition (\ref{eqn:decomposition}) of its positive part, $X_{\mathcal{B}^+}$ has one minimal component and that $w^+(x) = 0$ for every element $x$ in a periodic component. Suppose that for every $\eta >0$ there exists a sequence $\{\varepsilon_k\}_{k\in\mathbb{N}}$ with
\begin{equation}
\label{eqn:smallSides}
\varepsilon_k \leq \min_{1\leq i\leq |\mathcal{V}_k|}\left\{ \frac{\bar{h}^k_i}{2}, \frac{\bar{\ell}^k_i}{2} \right\}
\end{equation}
such that, for
\begin{equation}
\label{eqn:systole2}
\delta_k := \min  \left\{  \min_{v\in \mathcal{V}_{k+ \Delta_{\mathcal{B}}^+(k)}}\left\{ \frac{e^{ t_k}}{2} w^+ (v)\right\} ,  \min_{v\in \sigma^k(\mathcal{V})_{-\Delta_{\mathcal{B}}^-(0)}}\left\{  \frac{w^-_k(v)}{2}\right\} \right\}
\end{equation}
 and
\begin{equation}
\label{eqn:diam}
 \sigma_k \equiv 1+\sum_{i=1}^{|\mathcal{V}_k|}\bar{h}_i^k 
\end{equation}
we have that
\begin{equation}
\label{eqn:summability}
2\varepsilon_k\sigma_k\leq\eta \hspace{.5in} \mbox{ and }\hspace{.5in} \sum_{k\in\mathbb{N}}\left( \varepsilon^{-2}_k\sigma_k + \frac{|\mathcal{V}_k|-1}{\delta_k} \right)^{-2} = \infty.
\end{equation}
Then any adic transformation $\Phi$ defined from the positive part of $\mathcal{B}$ is ergodic with respect to the probability measure defined by the weight $w^+$. Moreover, the vertical flow on $S(\mathcal{B},w^\pm, \leq_{r,s})$ is ergodic with respect to the Lebesgue measure.
\end{theorem}

\begin{remark}
We note that although we need the orders $\leq_{r,s}$ to construct a flat surface, the above criterion is independent of the orders: the orders may change the topology of the surface when it is constructed, but the geometry is dictated by the weight functions $w^\pm$. Since it is the evolution geometry which we seek to control, the only relevant quantities in the above theorem come from $w^\pm$.
\end{remark}

\begin{remark}
\label{rem:hyp}
The actual proof of the theorem will show that even if $\Delta^-_\mathcal{B}(k) = \infty$ for any $k\geq 0$, the theorem still applies by changing the quantity $\delta_k$ in (\ref{eqn:systole2}) to
$$\delta_k = \frac{e^{t_k}}{2} \min_{v\in \mathcal{V}_{k+\Delta_{\mathcal{B}}^+(k)}}\{w^+(v)\}.$$
In fact, the theorem could be stated under the assumption that $\Delta^+_{\mathcal{B}}(k) < \infty$ for all $k\geq 0$ and change the definition of $\delta_k$ as above. We do not know of any examples for which it is advantageous to state it in one form or another.
\end{remark}

We now derive some corollaries from Theorem \ref{thm:erg}.

\begin{definition}
An ordered Bratteli diagram $(\mathcal{B},\leq_{r,s})$ is \emph{stationary} if there is a number $N>0$ such that the transition matrices satisfy $F_k = F_{k + Ni}$ for all $k,i$ and all the orders are repeat correspondingly.
\end{definition}

\begin{corollary}
\label{cor:veech}
Let $(\mathcal{B},w^\pm, \leq_{r,s})$ be a stationary diagram with $\Delta^+_\mathcal{B}(0) < \infty$ and $w^+(x) = 0$ for every path in a periodic component. Then the Veech group of the surface $S(\mathcal{B},w^\pm, \leq_{r,s})$ is not trivial. Moreover, the vertical flow of $S(\mathcal{B},w^\pm, \leq_{r,s})$ is ergodic with respect to the Lebesgue measure and any adic transformation defined on the positive part of $\mathcal{B}$ is ergodic with respect to the measure defined by $w^+$.
\end{corollary}

\begin{proof}
Without loss of generality (in particular, by telescoping), we can assume the diagram $\mathcal{B}$ is defined by a single transition matrix $F$ along with some orderings and with $\Delta_\mathcal{B}^+(k) = 1$ for all $k\geq 0$. Since $\mathcal{B}$ is stationary, the renormalization maps (\ref{eqn:Rmaps}) induce an affine, hyperbolic automorphism of $S(\mathcal{B},w^\pm, \leq_{r,s})$, so its Veech group is not trivial. Moreover, by the definition of stationary, all the relevant geometric quantities required for Theorem \ref{thm:erg} are constant for every $k$, making the sum (\ref{eqn:summability}) diverge. It follows that the vertical flow of $S(\mathcal{B},w^\pm, \leq_{r,s})$ is ergodic with respect to the Lebesgue measure.
\end{proof}

\begin{definition}
A Bratteli diagram $(B,\leq_{r,s})$ is \emph{eventually stationary} if there is a number $K$ such that for all $k>K$, up to telescoping, there is only one matrix and orders $(F,\leq_{r,s})$ describing the transition between the $k^{th}$ and $(k+1)^{st}$ level of $B$.
\end{definition}
\begin{corollary}
\label{cor:eventually}
Suppose the positive part of $(\mathcal{B},w^\pm, \leq_{r,s})$ is eventually stationary, $\Delta^+_\mathcal{B}(k) < \infty$ for infinitely many $k\in\mathbb{N}$, and $w^+(x) = 0$ for every path in a periodic component. Then the vertical flow of $S(\mathcal{B},w^\pm, \leq_{r,s})$ is ergodic with respect to the Lebesgue measure and any adic transformation defined on the positive part of $\mathcal{B}$ is ergodic with respect to the measure defined by $w^+$.
\end{corollary}
\begin{proof}
Since the positive part of $(\mathcal{B},w^\pm, \leq_{r,s})$ is eventually stationary, the positive part of $\sigma^n(\mathcal{B},w^\pm, \leq_{r,s})$, for all $n$ large enough, is stationary and (up to telescoping) given by a single matrix and its orderings $(F^*,\leq_{r,s})$. Since the renormalization maps (\ref{eqn:Rmaps}) induce a topological conjugacy between the vertical flow of $S(\mathcal{B},w^\pm, \leq_{r,s})$ and that of $\mathcal{R}_k(S(\mathcal{B},w^\pm, \leq_{r,s}))$ for any $k$, and the vertical flow only depends on the positive part of a weighted, ordered Bratteli diagram, we can consider the stationary diagram given by $(F^*,\leq_{r,s})$. By Corollary \ref{cor:veech}, the vertical flow for the surface constructed by this stationary diagram is ergodic, and therefore so is the vertical flow of $S(\mathcal{B},w^\pm, \leq_{r,s})$ through the topological conjugacy.
\end{proof}

\subsection{Applications of Theorem \ref{thm:erg}}

\subsubsection{A very symmetric example}
\label{subsubsec:symmetric}
We consider the system given by the Bratteli diagram in Figure \ref{fig:symmetric}. This generalizes an example given in \cite{FFT}.

In this system, for the positive part of the diagram $\mathcal{B}$, there are $p$ vertices on the $k^{th}$ level of the Bratteli diagram and the $p \times p$ matrix $F_k$ is given by
\begin{equation}
\label{eqn:symTrans}
F_k = \left( \begin{array}{cccc} n_k & 1 & \cdots & 1 \\ 1 & n_k & \cdots & 1 \\ \vdots & \vdots & \ddots & \vdots \\ 1 & 1 & \cdots & n_k \end{array} \right) 
\end{equation}
for some sequence $n_k$. Since $|V_0| = 1$, we take the negative part of the diagram $\mathcal{B}$ to be the stationary diagram given by the matrix 1.
\begin{theorem}
\label{thm:symmetric}
Let $B$ be the Bratteli diagram depicted in Figure \ref{fig:symmetric}, with $|V_k| = p$, for some $p>0$ and for all $k > 0$, and with transition matrices $F_k$ as in (\ref{eqn:symTrans}) for some sequence $\{n_k\}_{k\in\mathbb{N}}$. If
\begin{equation}
\label{eqn:sym}
\sum_{k>0} \frac{1}{n_k^2} = \infty
\end{equation}
then the Bratteli-Vershik system defined by $B$ is ergodic.
\end{theorem}
It is known that when $p = 2$, the exponent 2 in (\ref{eqn:sym}) is not optimal (see \cite{FFT}). We do not know whether it is optimal for other choices of $p$.

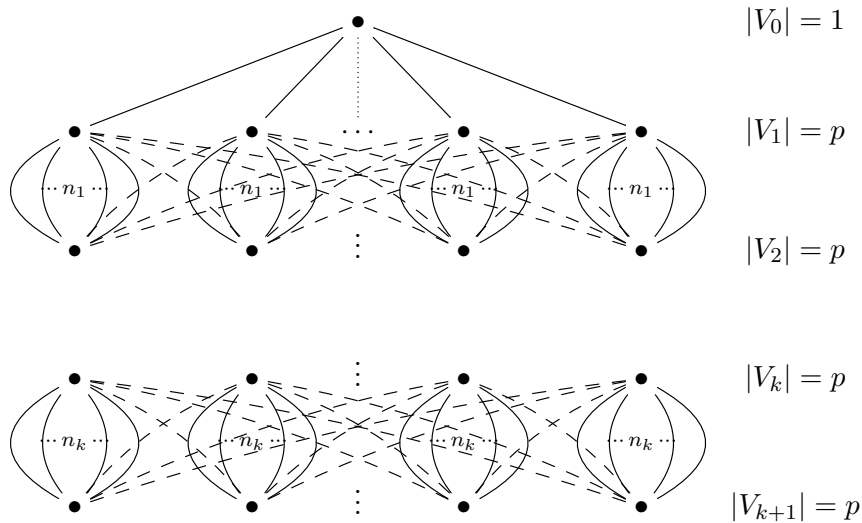
\begin{figure}[h!]
\begin{displaymath}
\xymatrix{
& & &\bullet \ar@{-}[dlll] \ar@{-}[dl] \ar@{.}[d]  \ar@{-}[drrr] \ar@{-}[dr] & &  && |V_0| = 1\\
\bullet \ar@/_2pc/@{-}[d]  \ar@/_1pc/@{-}[d] \ar@{}[d]|{\cdots\, n_1\, \cdots} \ar@/^1pc/@{-}[d] \ar@/^2pc/@{-}[d] \ar@/^/@{--}[drr] \ar@/^/@{--}[drrrr] \ar@/^/@{--}[drrrrrr] & & \bullet \ar@/_/@{--}[dll] \ar@/_2pc/@{-}[d]  \ar@/_1pc/@{-}[d] \ar@{}[d]|{\cdots\, n_1\, \cdots} \ar@/^1pc/@{-}[d] \ar@/^2pc/@{-}[d] \ar@/^/@{--}[drr] \ar@/^/@{--}[drrrr] & \cdots & \bullet \ar@/_/@{--}[dllll] \ar@/_/@{--}[dll] \ar@/_2pc/@{-}[d]  \ar@/_1pc/@{-}[d] \ar@{}[d]|{\cdots\, n_1\, \cdots} \ar@/^1pc/@{-}[d] \ar@/^2pc/@{-}[d] \ar@/^/@{--}[drr]& & \bullet\ar@/_/@{--}[dllllll] \ar@/_/@{--}[dllll] \ar@/_/@{--}[dll]  \ar@/_2pc/@{-}[d]  \ar@/_1pc/@{-}[d] \ar@{}[d]|{\cdots\, n_1\, \cdots} \ar@/^1pc/@{-}[d] \ar@/^2pc/@{-}[d] & |V_1| = p \\
\bullet & & \bullet & \vdots & \bullet & & \bullet & |V_2| = p \\
\bullet \ar@/_2pc/@{-}[d]  \ar@/_1pc/@{-}[d] \ar@{}[d]|{\cdots\, n_k\, \cdots} \ar@/^1pc/@{-}[d] \ar@/^2pc/@{-}[d] \ar@/^/@{--}[drr] \ar@/^/@{--}[drrrr] \ar@/^/@{--}[drrrrrr] & & \bullet \ar@/_/@{--}[dll] \ar@/_2pc/@{-}[d]  \ar@/_1pc/@{-}[d] \ar@{}[d]|{\cdots\, n_k\, \cdots} \ar@/^1pc/@{-}[d] \ar@/^2pc/@{-}[d] \ar@/^/@{--}[drr] \ar@/^/@{--}[drrrr] & \vdots & \bullet \ar@/_/@{--}[dllll] \ar@/_/@{--}[dll] \ar@/_2pc/@{-}[d]  \ar@/_1pc/@{-}[d] \ar@{}[d]|{\cdots\, n_k\, \cdots} \ar@/^1pc/@{-}[d] \ar@/^2pc/@{-}[d] \ar@/^/@{--}[drr]& & \bullet\ar@/_/@{--}[dllllll] \ar@/_/@{--}[dllll] \ar@/_/@{--}[dll]  \ar@/_2pc/@{-}[d]  \ar@/_1pc/@{-}[d] \ar@{}[d]|{\cdots\, n_k\, \cdots} \ar@/^1pc/@{-}[d] \ar@/^2pc/@{-}[d] & |V_k| = p \\
\bullet & & \bullet & \vdots & \bullet & & \bullet & |V_{k+1}| = p \\
}
\end{displaymath}
\vspace{-7mm}
\caption{Bratteli diagram for our symmetric example.}
\label{fig:symmetric}
\vspace{-5mm}
\end{figure}

\begin{proof}
The symmetry of this system yields uniform heights and widths of the towers defined for vertices at any given level. More precisely, we have that
$$\ell_i^k = \frac{1}{p\prod_{i=1}^{k-1}(n_i + p - 1)}$$
for all $i\in \{1,\dots, p\}$. This, combined with (ii) in definition \ref{def:weightfunctiondef}, gives that
$$h_i^k = \prod_{i=1}^{k-1}(n_i+p-1).$$ 
Our renormalization times are
$$t_k = -\log\left(\sum_{i=1}^p \ell_i^k\right) = \log\left( \prod_{i=1}^{k-1}(n_i+p-1) \right)$$
so we have that $\bar{\ell}_i^k = p^{-1}$ and $ \bar{h}_i^k = 1$ for all $k>0$ and all $i\in \{1,\dots, p\}$. We also have that $\sigma_k = p+1$ and finally that
$$\delta_k = \frac{1}{2(n_{k-1}+p-1)}.$$
We can choose $\varepsilon_k$ so that $\varepsilon^{-2}_k\sigma_k = 4(p+1)^3/\eta \equiv C_\eta$ and thus we can rewrite (\ref{eqn:summability}) now as 
$$\left( \varepsilon^{-2}_k\sigma_k + \frac{|V_k|-1}{\delta_k} \right)^{-2} = \frac{1}{( C_\eta + 2(p-1)(n_{k-1}+p-1) )^2}.$$
Theorem \ref{thm:symmetric} follows from this.
\end{proof}

\subsubsection{An explosive example}
\label{subsubsec:explosive}
We now consider the system given by modifying the Bratteli diagram in Figure \ref{fig:symmetric} in the following way. In this system, for the positive part of $\mathcal{B}$, there are $p_k$ vertices on the $k^{th}$ level given by some sequence $\{p_k\}_{k\in\mathbb{N}}$ and the $p_{k+1}\times p_k$ matrix $F_k$ has entries $f^k_{i,j} = n_k$ for all $i\in\{1,\dots,p_{k+1}\}$ and $j\in \{1,\dots, p_k\}$, where the $n_k$ are given by some sequence $\{n_k\}_{k\in\mathbb{N}}$. We again consider the negative part of $\mathcal{B}$ to be given by the $p_0 \times p_0$ identity matrix.

\begin{theorem}
\label{thm:explosive}
Let $B$ be the Bratteli diagram described above with $|V_k| = p_k$ whose transition matrices satisfy $f^k_{i,j} = n_k$, where $\{n_k\}$ and $\{p_k\}$ are some sequences. If 
$$\sum_{k>0}\frac{1}{((p_k+1)^3 + p_{k-1}n_{k-1})^2} = \infty$$
then the Bratteli-Vershik system defined by $B$ is ergodic.
\end{theorem}

\begin{corollary}
Consider the Bratteli diagram described above. If $n_{k-1}p_{k-1}\leq p_k$ for all $k$, and if
\begin{equation}
\label{eqn:explosive}
\sum_{k>0} \frac{1}{p_k^6} = \infty
\end{equation}
then the associated Bratteli-Vershik system is ergodic. In particular, if there is an $N$ such that $n_k< N$ for all $k$ and $p_k = |V_k|\leq |V_{k+1}| = p_{k+1}$ for all $k$ and (\ref{eqn:explosive}) holds, the system is ergodic. Moreover, if there exists a $P$ such that $p_k<P$ for all $k$, the condition (\ref{eqn:sym}) for the sequence $\{n_k\}$ is sufficient for ergodicity.
\end{corollary}

\begin{proof}[Proof of Theorem \ref{thm:explosive}]
We again take advantage of this system's symmetry to explicitly compute all the relevant quantities. We have uniform heights and widths of the towers defined for vertices at any given level. More precisely, we have that
$$\ell_i^k = \frac{1}{p_k\prod_{i=1}^{k-1}(n_i p_i)}$$
for all $i\in \{1,\dots, p_k\}$. Combined with (ii) in definition \ref{def:weightfunctiondef}, this gives
$$h_i^k = p_k \prod_{i=1}^{k-1}n_ip_i.$$ 
Our renormalization times are
$$t_k = -\log\left(\sum_{i=1}^p \ell_i^k\right) = \log\left( \prod_{i=1}^{k-1}n_ip_i \right)$$
so we have that $\bar{\ell}_i^k = p_k^{-1}$ and $ \bar{h}_i^k = 1$ for all $k>0$ and all $i\in \{1,\dots, p_k\}$. We also have that $\sigma_k = p_k+1$ and finally that
$$\delta_k = \frac{p_k}{2(n_{k-1}p_{k-1})}.$$
We can choose $\varepsilon_k$ so that $\varepsilon^{-2}_k\sigma_k = 4(p_k+1)^3/\eta \equiv C_\eta$ and thus we can rewrite (\ref{eqn:summability}) now as 
$$\left( \varepsilon^{-2}_k\sigma_k + \frac{|V_k|-1}{\delta_k} \right)^{-2} = \frac{1}{(2\eta^{-1}(p_k+1)^3 + r_kp_{k-1}n_{k-1})^2}$$
where $r_k = \frac{p_k-1}{p_k} < 1$.
\end{proof}

\subsubsection{The Arnoux-Yoccoz-Bowman surface}
\label{subsec:josh}
In the early 80's, Arnoux and Yoccoz \cite{ArnYoc} constructed a family of flat surfaces, one of every genus $g\geq 3$. These served as examples of surfaces carrying pseudo Anosov maps, which where not well-understood as the theory was still in its infancy. It was eventually shown that the Veech groups of these surfaces are quite peculiar: they do not contain parabolic elements \cite{HubertLanneau}. One usually expects that if the Veech group of a flat surface has an infinite subgroup of hyperbolic automorphisms, then it is generated by parabolic elements. For the Arnoux-Yoccoz family of surfaces, this was shown not to be the case.

\begin{figure}[h!]
\begin{center}
\includegraphics[width=.7\linewidth]{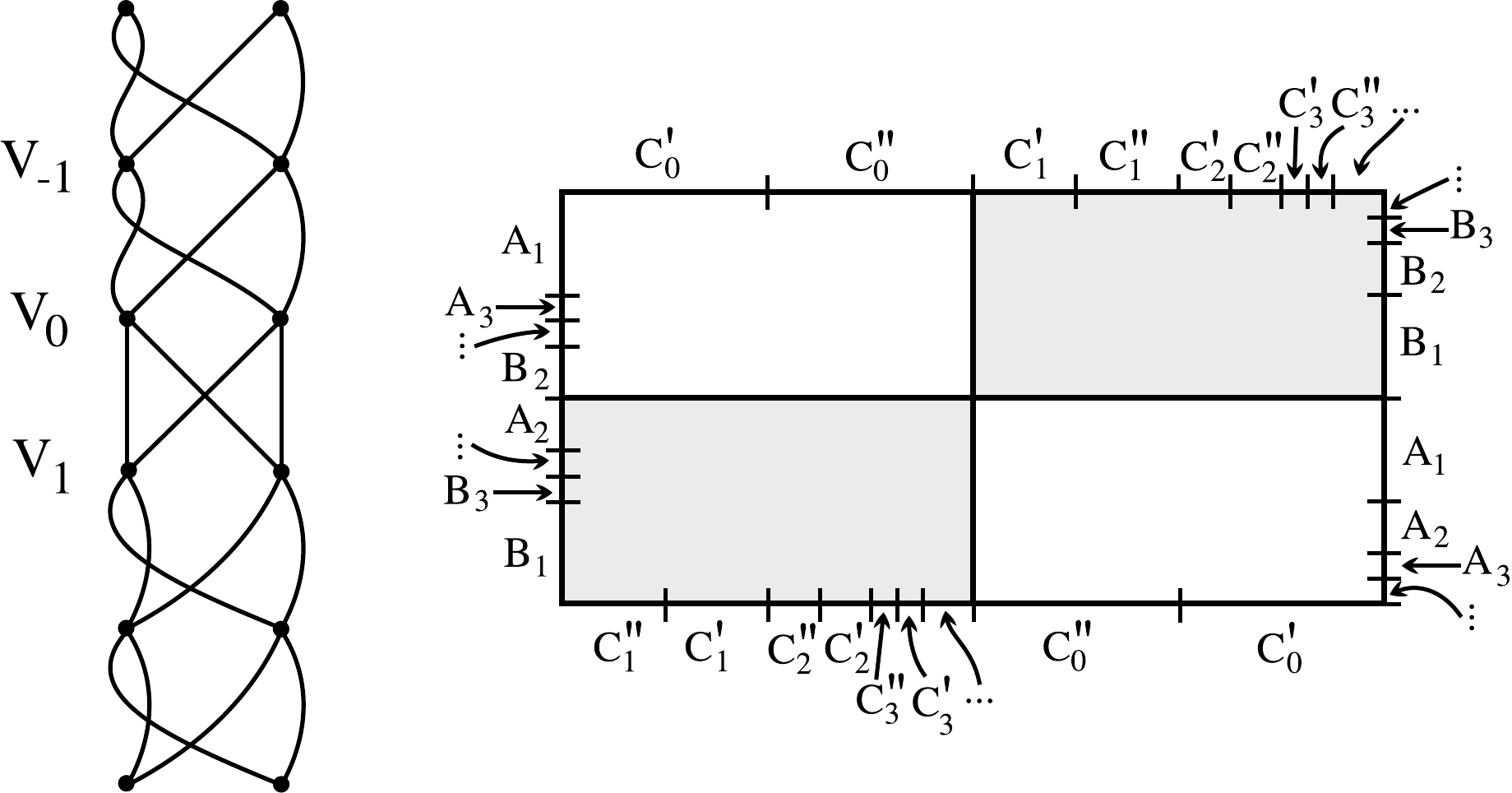}
\caption{The Arnoux-Yoccoz-Bowman surface: A Bratteli diagram that defines it and its rectangle representation as in (\ref{eqn:rectangles}).}.
\label{f:Bowman}
\end{center}
\end{figure}

Bowman \cite{bowman} has taken the geometric limit of this family of surfaces as the genus goes to infinity. The limiting surface will be referred to as the \emph{Arnoux-Bowman-Yoccoz surface}, and it is depicted in Figure \ref{f:Bowman}. Bowman showed that the vertical and horizontal flows of this surface are ergodic. This surface has finite area and, much like its finite-genus ``subsurfaces'', the Veech group of this surface contains no parabolic elements. In fact Bowman showed that the Veech group of this surface is isomorphic to $\mathbb{Z}\times\mathbb{Z}_2$, where the infinite subgroup is generated by the map which expands the horizontal direction by a factor of 2 while contracting the vertical by a factor of $\frac{1}{2}$ (as shown in figure \ref{f:Bowman}). By Corollary \ref{cor:eventually}, we can also recover the fact that the vertical and horizontal flows of this surface are ergodic.

\subsubsection{The Chacon middle third transformation revisited}
\label{subsec:ChaconReprise}
Recall Chacon's middle third transformation, described in \S\ref{subsubsec:chaconMiddleThird}. Since the Bratteli diagram is stationary and $\Delta^+_\mathcal{B}(0) = 1$ we can apply Corollary \ref{cor:veech} to obtain ergodicity of the translation flow of the corresponding surface. Note that the tail equivalence relation is not ergodic: there is an atomic invariant measure supported on the path going along the right-most edge in Figure \ref{f:Chaconpic}.

\subsubsection{The Hajian-Kakutani skyscraper}
\label{subsubsec:HajianKakutani}
Cutting and stacking can also be used to define ergodic transformations on infinite intervals. The Hajian-Kakutani Skyscraper is an example of an infinite measure-preserving, invertible, rank-one, ergodic transformation. We briefly review the cutting and stacking procedure to define this transformation.

Starting with the interval $[0,1)$ (our zeroth tower), we cut it into two intervals of equal length, place 2 spacers over the second interval, and define the first map as the ``moving up one level'' linear map on the first tower obtained by stacking $[\frac{1}{2},1)$ and the two spacers above it above $[0,\frac{1}{2})$. The map is now defined on $[0,\frac{3}{2})$. Proceeding inductively, we may define the $(k+1)^{st}$ map by cutting the $k^{th}$ tower into two subtowers of equal width, adding $2^k$ spacers above the second subtower, and stacking that over the first subcolumn. It can be easily checked that the $k^{th}$ map is defined on $4^k -1$ intervals of length $2^{-k}$. Therefore, the limiting map is defined on $[0,\infty)$.

Let us now describe this transformation through a Bratteli diagram. See Figure \ref{f:HajianKakutani}, where the weight function on each edge emanating from each left-most vertex is $1/2$. By Proposition \ref{thm:existence}, there is a \emph{probability} measure which is invariant for the tail equivalence relation for the corresponding diagram. In fact, there is a unique probability invariant measure, the one supported on the point $x_\infty$, which is the path in $X_B$ consisting of the right-most edge at every level. It is easy to check that any other ergodic invariant probability measure is not compatible with the diagram.

\begin{figure}[h!]
\begin{center}
  \includegraphics[width=1.5in]{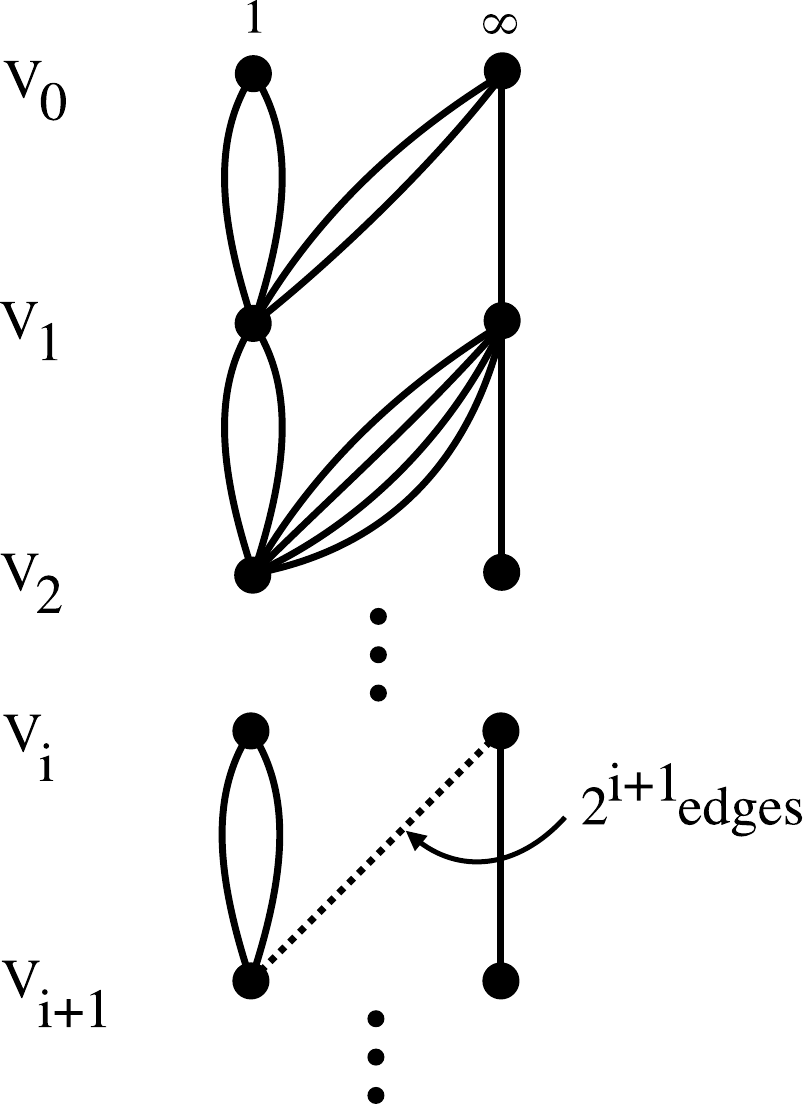}
\caption{The Bratteli diagram corresponding to the Hajian-Kakutani skyscraper.}
\label{f:HajianKakutani}
\end{center}
\end{figure}

This is not surprising: we are trying to encode a minimal, ergodic interval exchange transformation of $[0,\infty)$ into a diagram and hoping that there is a probability invariant measure which reveals meaningful information about the dynamics. As in \S \ref{subsubsec:chaconMiddleThird}, the point $x_\infty$ is also somewhat artificial: it is the ``ghost tower'' from which we are getting spacers for the cutting and stacking transformation. Since the point $x_\infty$ in Figure \ref{f:HajianKakutani} represents an interval of infinite length from which we get our spacers, it makes sense that the ``finite'' invariant measure for the Bratteli diagram is all concentrated in the ghost towers from which we get spacers.

The tools developed here, however, are still useful to determine ergodicity of cutting and stacking transformations with respect to Lebesgue measure even when they are defined on $[0,\infty)$. The crucial observation is that such transformations may be \emph{towers} over transformations on probability spaces. Thus, when coding the cutting and stacking transformation through a Bratteli diagram, one can apply the tools developed here to the base transformation over which the tower is built. If Theorem \ref{thm:erg} can be applied to it, then it follows that the tower transformation is ergodic.  This is true for the example at hand, the Hajian-Kakutani skyscraper; it can be easily seen to be a tower over the dyadic odometer (see Figure \ref{f:HajianKakutani}).

\subsection{Proof of Theorem \ref{thm:erg}}
Our strategy to apply Theorem \ref{thm:flat} is as follows: since we are looking for a one-parameter family of good subsets $S_{\varepsilon(t),t}$ which cover most of the surface and does not degenerate (geometrically) too quickly so that it satisfies (\ref{eqn:integrability}), we will find a sequence of good such subsets along the sequence of times $t_k\rightarrow \infty$ given by (\ref{eqn:rTimes}). Using these times and the renormalization maps (\ref{eqn:Rmaps}) we will see that conditions (\ref{eqn:summability}) satisfy (i)-(iii) in Theorem \ref{thm:flat} and that (\ref{eqn:summability}) is sufficient to satisfy (\ref{eqn:integrability}).

\begin{proof}
Since we assumed that $w^+(x) = 0$ for every element of a periodic component of $X_{\mathcal{B}^+}$, there are no cylinders for the vertical flow on $S(\mathcal{B},w^\pm,\leq_{r,s})$. Let $\eta>0$ be given.

The sequence $\varepsilon_k$ to be chosen will correspond to the value of $\varepsilon(t)$ for the sequence of times $t_k\rightarrow \infty$ in (\ref{eqn:rTimes}). We will first choose sets $S_{\varepsilon(t_k),t_k}$ for this sequence of times satisfying (ii) in Theorem \ref{thm:flat}.

For $t = t_k$, we will let
$$S_k \equiv S_{\varepsilon(t_k),t_k} = \bigsqcup_{i = 1}^{|\mathcal{V}_k|} \mathcal{R}_k^{-1} (S_k^i),$$
where the $S_k^i$ are the points of the rectangles composing the towers of the surface $S(\sigma^k(\mathcal{B},w^\pm,\leq_{r,s}))$ which are $\varepsilon_k$ away from the edge of the rectangles, and the map $\mathcal{R}_k$ is defined in (\ref{eqn:Rmaps}). To ensure that the sets $S^i_k$ are well defined (nonempty), we impose the condition (\ref{eqn:smallSides}). See Figure \ref{fig:towers}. As such, in order to satisfy (i) in Theorem \ref{thm:flat}, the sum of the areas of the $\varepsilon_k$-frames around the $|\mathcal{V}_k|$ towers must be less than $\eta$ (since the maps $T_k$ are area-preserving). 

\begin{figure}[h!]
\begin{center}
  \includegraphics{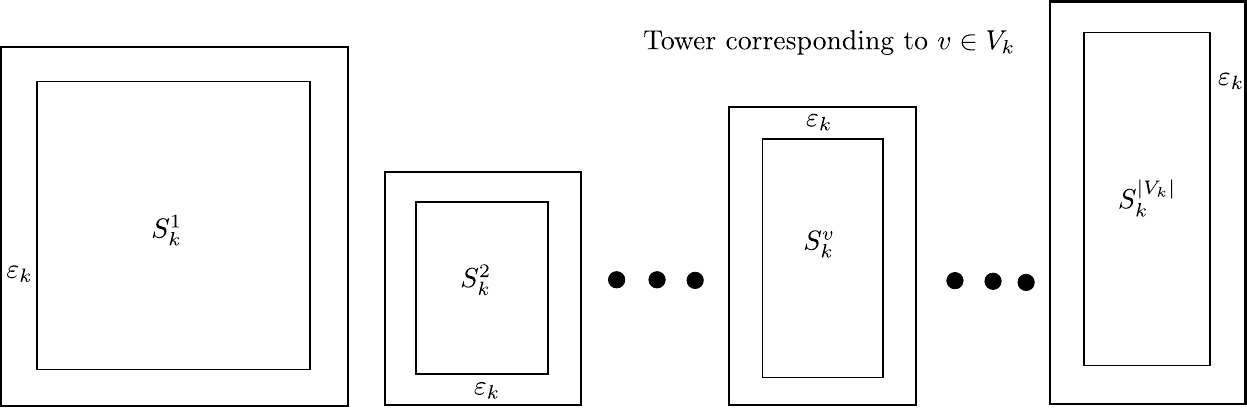}
\caption{Decomposition of the surface $S_k(\mathcal{B},w^\pm,\leq_{r,s})$ as towers. We do not depict the edge identifications along the boundaries of the rectangles making up the towers.}
\label{fig:towers}
\vspace{-5mm}
\end{center}
\end{figure}
The area of the $\varepsilon_k$-frame around the $S_k^i$ is bounded by $2\varepsilon_k(\bar{\ell}^k_i + \bar{h}_i^k)$. Therefore, summing over all rectangles and recalling that $\sum_i \bar{\ell}^k_i = 1$, we see that condition (i) in Theorem \ref{thm:flat} is satisfied if the first equation in (\ref{eqn:summability}) is satisfied.

The diameter of each $S_k^i$ is bounded by the diameter of the tower to which it belongs, and therefore bounded by $\bar{\ell}_i^k + \bar{h}_i^k$. Therefore, since $\sum_i \bar{\ell}^k_i = 1$, we have that
$$\varepsilon(t_k)^{-2}\sum_{i=1}^{C_{t_k}} \mathcal{D}_{t_k}^i \hspace{.1in}\mbox{ in (\ref{eqn:integrability}) becomes } \hspace{.1in} \varepsilon_k^{-2} \left(1+\sum_{i=1}^{|\mathcal{V}_k|}\bar{h}_i^k \right)  =  \varepsilon_k^{-2} \sigma_k \hspace{.1in} \mbox{ in (\ref{eqn:summability})}.$$

We claim now that the quantity $\delta_k$ in Theorem \ref{thm:erg} corresponds to $\delta_{t_k}$ in Theorem \ref{thm:flat}. For every pair $S_k^i$ and $S_k^j$, we want to find a path which connects them which stays as far away from $\Sigma$ as possible. We will choose paths will go through the edges of the towers making up $S(\sigma^k(\mathcal{B},w^\pm, \leq_{r,s}))$. Since the left/right (top/bottom) identifications of these rectangles are encoded in the negative (positive) part of $\sigma^k(\mathcal{B},w^\pm,\leq_{r,s})$, we need to look at a finite chunk of the negative (positive) part of $\sigma^k(\mathcal{B},w^\pm,\leq_{r,s})$ which gives us enough information on how to find a path on $S_k$ connecting $S_k^i$ to $S_k^j$, possibly going through other sets $S_k^l$ in the middle. The tunneling function $\Delta^\pm_{\mathcal{B}}$ gives us precisely enough information to ``tunnel through'' and find paths connecting any two sets $S_k^i$ and $S_k^j$ that stay away from the singularities as much as possible.

Let $R_k^i$ denote the $i^{th}$ rectangle of $S(\sigma^k(\mathcal{B},w^\pm,\leq_{r,s}))$ as in (\ref{eqn:rectangles}) corresponding to the vertex $v_i \in \mathcal{V}_k$. Let $S_k^i\subset R_k^i$, $S_k^j\subset R_k^j$ be two components of $\mathcal{R}_k(S_k)$ and suppose that a left/right edge of $R_k^i$ is identified to some right/left edge of $R_k^j$. 

We claim that we can connect $S_k^i$ with $S_k^j$ by a path $\gamma_{i,j}^k:[0,1]\rightarrow S_k(\mathcal{B},w^\pm,\leq_{r,s})$ which satisfies
\begin{equation}
\label{eqn:tunnel1}
\min_{t\in[0,1]}\mbox{dist}(\gamma^k_{i,j}(t),\Sigma) \geq \frac{1}{2} \min_{v\in \sigma^k(\mathcal{V})_{-\Delta_{\mathcal{B}}^-(0)}}\{w^-_k(v)\}.
\end{equation}
By assumption, there is a vertex $v'\in \mathcal{V}_{k-\Delta_{\mathcal{B}}^-(k)}$ and paths $p\in E_{k,k-\Delta_{\mathcal{B}}^-(k)}$ and $q\in E_{k-\Delta_{\mathcal{B}}^-(k),k}$ such that $s(p) = v_i$, $r(p) = v' = s(q)$, and $r(q) = v_j$. This means that there are horizontal subsets $H_i\subset R_k^i$ and $H_j\subset R_k^j$ with interiors isometric to 
$$(0,w^+(v_l))\times \left(  0,e^{-t_k} h_{v'}^{k-\Delta_{\mathcal{B}}^-(k)} \right)$$
for $l = i,j$ such that 
$$ H'_l \equiv \mathcal{R}_{k-\Delta_{\mathcal{B}}^-(k)} \circ \mathcal{R}_k ^{-1} (H_l) \subset R_k^{v'}$$
for both $l=i,j$ since the paths $p,q$ imply $R^i_k$ and $R^j_k$ came from the tower/rectangle defined by $v'$ by the process of cutting and stacking.

Let $L_{v'}:[0,1]\rightarrow S_{k-\Delta_{\mathcal{B}}^-(k)}^{v'}$ be a path whose image is the unique horizontal line of length $w^+_{k-\Delta_{\mathcal{B}}^-(k)}(v')-2\varepsilon_{k-\Delta_{\mathcal{B}}^-(k)}$ which separates $S_k^{v'}$ into subrectangles of equal area and let $L'$ be a connected subset of this line whose endpoints are contained in $H'_i$ and $H'_j$, respectively. Let $\gamma:[0,1] \rightarrow S_k(\mathcal{B},w^\pm,\leq{r,s})$ be an injective map parametrizing $\mathcal{R}_k\circ\mathcal{R}_{k-\Delta_{\mathcal{B}}^-(k)}^{-1}(L')$. This path connects the sets $S_k^i$ and $S_k^j$.

It now follows that the curve $\gamma$ satisfies (\ref{eqn:tunnel1}). Indeed, the affine scaling factor from the map $\mathcal{R}_k\circ\mathcal{R}_{k-\Delta_{\mathcal{B}}^-(k)}^{-1}$ in the vertical direction is $e^{t_{k-\Delta_{\mathcal{B}}^-(k)}-t_k}$ and since the only singularities the path $\gamma$ gets close to are the ones coming from the cutting and stacking of $R_{k-\Delta_{\mathcal{B}}^-(k)}^{v'}$, the path $\gamma$ stays $e^{-t_k} h^{k-\Delta_{\mathcal{B}}^-(k)}_{v}$ away from $\Sigma$ and therefore satisfies (\ref{eqn:tunnel1}).

Suppose now that $R_k^i$ and $R_k^j$ are connected by some edge identification on their top/bottom edges. An analogous argument shows that there is a curve $\gamma:[0,1]\rightarrow S_k(\mathcal{B},w^\pm, \leq_{r,s})$ connecting $S_k^i$ with $S_k^j$ and satisfying
\begin{equation}
\label{eqn:tunnel2}
\min_{t\in[0,1]}\mbox{dist}(\gamma(t),\Sigma) \geq \frac{e^{t_k}}{2} \min_{v\in \mathcal{V}_{k+\Delta_{\mathcal{B}}^+(k)}}\{w^+(v)\}.
\end{equation}

Therefore, for any two sets $S_k^i$ and $S_k^j$, if we can connect them by a path $\gamma$, using (\ref{eqn:tunnel1}) and (\ref{eqn:tunnel2}), we can do so satisfying 
$$ \min_{t\in[0,1]}\mbox{dist}(\gamma(t),\Sigma) \geq \min_{\pm} \frac{e^{\pm t_k}}{2} \min_{v\in \mathcal{V}_{k\pm\Delta_{\mathcal{B}}^\pm(k)}}\{w^\pm(v)\}$$
and therefore justified our choice for $\delta_{t_k}$ in (\ref{eqn:systole2}) for (\ref{eqn:systole1}).

We have showed why 
$$\varepsilon(t_k)^{-2}\sum_{i=1}^{C_{t_k}}\mathcal{D}_{t_k}^i + \frac{C_{t_k}-1}{\delta_{t_k}} \mbox{ in (\ref{eqn:integrability}) becomes }\varepsilon_k^{-2}\sigma_k + \frac{|\mathcal{V}_k|-1}{\delta_k}\mbox{ in (\ref{eqn:summability})}.$$

We claim that with such choices, (\ref{eqn:summability}) is sufficient for (\ref{eqn:integrability}). Let $\mu>0$ be a real number satisfying $2\mu \ll \inf_k(t_k-t_{k-1})$. Using Teichm\"{u}ller deformations, we can deform the surfaces $S_k(\mathcal{B},w^\pm, \leq_{r,s})$ with their corresponding sets $S_k^i$ as in Figure \ref{fig:towers} in intervals of width $\mu/2$ around $t_k$ and thus, for $t\in(t_k-\frac{\mu}{2},t_k-\frac{\mu}{2})$, obtain the following bounds
\begin{equation}
\label{eqn:interBnds}
\mathcal{D}_t^i \leq e^{\frac{\mu}{2}}\mathcal{D}_{t_k}^i \mbox{ for all $i$},\hspace{.25in} \varepsilon(t) \geq e^{-\frac{\mu}{2}}\varepsilon_k, \hspace{.25in}\delta_t \geq e^{-\frac{\mu}{2}}\delta_k. 
\end{equation}
Using these bounds, for any other choice of sets $S_{\varepsilon(t),t}$ and functions $\varepsilon(t)$, $\mathcal{D}_t^i$ for $t\in (\mathbb{R}\backslash \bigcup_k (t_k-\frac{\mu}{2},t_k+\frac{\mu}{2}))$ we have that
\begin{equation}
\begin{split}
\int_0^\infty \left( \varepsilon(t)^{-2}\sum_{i=1}^{C_t}\mathcal{D}_t^i + \frac{C_t-1}{\delta_t}\right)^{-2}\, dt &\geq \sum_{k>0} \int_{t_k-\frac{\mu}{2}}^{t_k+\frac{\mu}{2}} \left( \varepsilon(t)^{-2}\sum_{i=1}^{C_t}\mathcal{D}_t^i + \frac{C_t-1}{\delta_t}\right)^{-2}\, dt \\
&\geq\mu \sum_{k>0} \left( e^{\frac{3\mu}{2}}\varepsilon_k^{-2}\sigma_k + e^{\frac{\mu}{2}}\frac{|\mathcal{V}_k|-1}{\delta_k}\right)^{-2} \\
&\geq \mu e^{3\mu} \sum_{k\in\mathbb{N}}\left( \varepsilon^{-2}_k\sigma_k + \frac{|\mathcal{V}_k|-1}{\delta_k} \right)^{-2}.
\end{split}
\end{equation}
Therefore, if (\ref{eqn:summability}) holds, then (\ref{eqn:integrability}) holds.

Finally, we point out that the set of trajectories which leave every compact set correspond to trajectories which go along the left/right edges of the rectangles obtained through the procedures of cutting and stacking. Since this set is countable, the set of points which leave every compact set has zero measure. Therefore, by Theorem \ref{thm:flat} the vertical flow on $S(\mathcal{B},w^\pm,\leq_{r,s})$ is ergodic.
\end{proof}

\bibliographystyle{plain}
\bibliography{biblio}

\begin{thebibliography}{10}

\bibitem{AS:Chi2}
Jon Aaronson and Omri Sarig.
\newblock Exponential chi-squared distributions in infinite ergodic theory.
\newblock {\em Ergodic Theory and Dynamical Systems}, FirstView:1--20, 11 2013.

\bibitem{adams:mixing}
Terrence~M. Adams.
\newblock Smorodinsky's conjecture on rank-one mixing.
\newblock {\em Proc. Amer. Math. Soc.}, 126(3):739--744, 1998.

\bibitem{Ambrose}
Warren Ambrose.
\newblock Representation of ergodic flows.
\newblock {\em Ann. of Math. (2)}, 42:723--739, 1941.

\bibitem{AOW}
Pierre Arnoux, Donald~S. Ornstein, and Benjamin Weiss.
\newblock Cutting and stacking, interval exchanges and geometric models.
\newblock {\em Israel J. Math.}, 50(1-2):160--168, 1985.

\bibitem{ArnYoc}
Pierre Arnoux and Jean-Christophe Yoccoz.
\newblock Construction de diff\'eomorphismes pseudo-{A}nosov.
\newblock {\em C. R. Acad. Sci. Paris S\'er. I Math.}, 292(1):75--78, 1981.

\bibitem{BKM09}
S.~Bezuglyi, J.~Kwiatkowski, and K.~Medynets.
\newblock Aperiodic substitution systems and their {B}ratteli diagrams.
\newblock {\em Ergodic Theory Dynam. Systems}, 29(1):37--72, 2009.

\bibitem{BKMS1}
S.~Bezuglyi, J.~Kwiatkowski, K.~Medynets, and B.~Solomyak.
\newblock Invariant measures on stationary {B}ratteli diagrams.
\newblock {\em Ergodic Theory Dynam. Systems}, 30(4):973--1007, 2010.

\bibitem{BKMS2}
S.~Bezuglyi, J.~Kwiatkowski, K.~Medynets, and B.~Solomyak.
\newblock Finite rank {B}ratteli diagrams: structure of invariant measures.
\newblock {\em Trans. Amer. Math. Soc.}, 365(5):2637--2679, 2013.

\bibitem{BowenMarcus}
Rufus Bowen and Brian Marcus.
\newblock Unique ergodicity for horocycle foliations.
\newblock {\em Israel J. Math.}, 26(1):43--67, 1977.

\bibitem{bowman}
Joshua Bowman.
\newblock The complete family of {Arnoux--Yoccoz} surfaces.
\newblock {\em Geometriae Dedicata}, pages 1--18, 2012.
\newblock 10.1007/s10711-012-9762-9.

\bibitem{bratteli}
Ola Bratteli.
\newblock Inductive limits of finite dimensional {$C^{\ast} $}-algebras.
\newblock {\em Trans. Amer. Math. Soc.}, 171:195--234, 1972.

\bibitem{bufetov:limitVershik}
Alexander~I. Bufetov.
\newblock Limit theorems for suspension flows over vershik automorphisms.
\newblock {\em Russian Mathematical Surveys}, 68(5):789, 2013.

\bibitem{bufetov:measures}
Alexander~I. Bufetov.
\newblock Finitely-additive measures on the asymptotic foliations of a {M}arkov
  compactum.
\newblock {\em Mosc. Math. J.}, 14(2):205--224, 426, 2014.

\bibitem{bufetov:limit}
Alexander~I. Bufetov.
\newblock Limit theorems for translation flows.
\newblock {\em Ann. of Math. (2)}, 179(2):431--499, 2014.

\bibitem{chacon}
R.~V. Chacon.
\newblock Weakly mixing transformations which are not strongly mixing.
\newblock {\em Proc. Amer. Math. Soc.}, 22:559--562, 1969.

\bibitem{chamanara}
R.~Chamanara.
\newblock Affine automorphism groups of surfaces of infinite type.
\newblock In {\em In the tradition of {A}hlfors and {B}ers, {III}}, volume 355
  of {\em Contemp. Math.}, pages 123--145. Amer. Math. Soc., Providence, RI,
  2004.

\bibitem{CreutzSilva}
Darren Creutz and Cesar~E. Silva.
\newblock Mixing on rank-one transformations.
\newblock {\em Studia Math.}, 199(1):43--72, 2010.

\bibitem{infinite-step}
Mirko Degli~Esposti, Gianluigi Del~Magno, and Marco Lenci.
\newblock Escape orbits and ergodicity in infinite step billiards.
\newblock {\em Nonlinearity}, 13(4):1275--1292, 2000.

\bibitem{DHL:wind-tree}
V.~{Delecroix}, P.~{Hubert}, and S.~{Leli{\`e}vre}.
\newblock {Diffusion for the periodic wind-tree model}.
\newblock {\em ArXiv e-prints}, July 2011.

\bibitem{DHS}
F.~Durand, B.~Host, and C.~Skau.
\newblock Substitutional dynamical systems, {B}ratteli diagrams and dimension
  groups.
\newblock {\em Ergodic Theory Dynam. Systems}, 19(4):953--993, 1999.

\bibitem{FFT}
Sebastien Ferenczi, Albert~M. Fisher, and Marina Talet.
\newblock Minimality and unique ergodicity for adic transformations.
\newblock {\em J. Anal. Math.}, 109:1--31, 2009.

\bibitem{fisher}
Albert~M. Fisher.
\newblock Nonstationary mixing and the unique ergodicity of adic
  transformations.
\newblock {\em Stoch. Dyn.}, 9(3):335--391, 2009.

\bibitem{FM:intro}
G.~{Forni} and C.~{Matheus}.
\newblock {Introduction to Teichm\"{u}ller theory and its applications to
  dynamics of interval exchange transformations, flows on surfaces and
  billiards}.
\newblock {\em ArXiv e-prints}, November 2013.

\bibitem{corinna}
K.~{Fr{\c a}czek} and C.~{Ulcigrai}.
\newblock {Non-ergodic Z-periodic billiards and infinite translation surfaces}.
\newblock {\em Invent. Math.}, 2013.

\bibitem{GjerdeJohansen}
Richard Gjerde and {\O}rjan Johansen.
\newblock Bratteli-{V}ershik models for {C}antor minimal systems associated to
  interval exchange transformations.
\newblock {\em Math. Scand.}, 90(1):87--100, 2002.

\bibitem{HermanPutnamSkau}
Richard~H. Herman, Ian~F. Putnam, and Christian~F. Skau.
\newblock Ordered {B}ratteli diagrams, dimension groups and topological
  dynamics.
\newblock {\em Internat. J. Math.}, 3(6):827--864, 1992.

\bibitem{hooper:measures}
W.~Patrick Hooper.
\newblock {The Invariant Measures of some Infinite Interval Exchange Maps}.
\newblock {\em Arxiv preprint arXiv:1005.1902}, 2010.

\bibitem{infinite-staircase}
W.~Patrick Hooper, Pascal Hubert, and Barak Weiss.
\newblock Dynamics on the infinite staircase.
\newblock {\em Discrete Contin. Dyn. Syst.}, 33(9):4341--4347, 2013.

\bibitem{HubertLanneau}
Pascal Hubert and Erwan Lanneau.
\newblock Veech groups without parabolic elements.
\newblock {\em Duke Math. J.}, 133(2):335--346, 2006.

\bibitem{HLT:Ehrenfest}
Pascal Hubert, Samuel Leli{\`e}vre, and Serge Troubetzkoy.
\newblock The {E}hrenfest wind-tree model: periodic directions, recurrence,
  diffusion.
\newblock {\em J. Reine Angew. Math.}, 656:223--244, 2011.

\bibitem{katok:mixing}
Anatole Katok.
\newblock Interval exchange transformations and some special flows are not
  mixing.
\newblock {\em Israel J. Math.}, 35(4):301--310, 1980.

\bibitem{MelaPetersen}
Xavier M{\'e}la and Karl Petersen.
\newblock Dynamical properties of the {P}ascal adic transformation.
\newblock {\em Ergodic Theory Dynam. Systems}, 25(1):227--256, 2005.

\bibitem{ornstein:mixing}
Donald~S. Ornstein.
\newblock On the root problem in ergodic theory.
\newblock In {\em Proceedings of the {S}ixth {B}erkeley {S}ymposium on
  {M}athematical {S}tatistics and {P}robability ({U}niv. {C}alifornia,
  {B}erkeley, {C}alif., 1970/1971), {V}ol. {II}: {P}robability theory}, pages
  347--356, Berkeley, Calif., 1972. Univ. California Press.

\bibitem{PetersenSchmidt}
Karl Petersen and Klaus Schmidt.
\newblock Symmetric {G}ibbs measures.
\newblock {\em Trans. Amer. Math. Soc.}, 349(7):2775--2811, 1997.

\bibitem{RalstTroub}
D.~{Ralston} and S.~{Troubetzkoy}.
\newblock {Ergodic infinite group extensions of geodesic flows on translation
  surfaces}.
\newblock {\em J. Mod. Dyn.}, (6):477 -- 497, 2012.

\bibitem{rauzy:CFE}
G{\'e}rard Rauzy.
\newblock Une g\'en\'eralisation du d\'eveloppement en fraction continue.
\newblock In {\em S\'eminaire {D}elange-{P}isot-{P}oitou, 18e ann\'ee: 1976/77,
  {T}h\'eorie des nombres, {F}asc. 1}, pages Exp. No. 15, 16. Secr\'etariat
  Math., Paris, 1977.

\bibitem{rauzy:IET}
G{\'e}rard Rauzy.
\newblock \'{E}changes d'intervalles et transformations induites.
\newblock {\em Acta Arith.}, 34(4):315--328, 1979.

\bibitem{Rudolph}
Daniel Rudolph.
\newblock A two-valued step coding for ergodic flows.
\newblock {\em Math. Z.}, 150(3):201--220, 1976.

\bibitem{shields}
Paul Shields.
\newblock Cutting and independent stacking of intervals.
\newblock {\em Math. Systems Theory}, 7:1--4, 1973.

\bibitem{SilvaBook}
C.~E. Silva.
\newblock {\em Invitation to ergodic theory}, volume~42 of {\em Student
  Mathematical Library}.
\newblock American Mathematical Society, Providence, RI, 2008.

\bibitem{strebel:book}
Kurt Strebel.
\newblock {\em Quadratic differentials}, volume~5 of {\em Ergebnisse der
  Mathematik und ihrer Grenzgebiete (3) [Results in Mathematics and Related
  Areas (3)]}.
\newblock Springer-Verlag, Berlin, 1984.

\bibitem{rodrigo:erg}
Rodrigo Trevi{\~n}o.
\newblock On the ergodicity of flat surfaces of finite area.
\newblock {\em Geom. Funct. Anal.}, 24(1):360--386, 2014.

\bibitem{veech:IETs}
William~A. Veech.
\newblock Interval exchange transformations.
\newblock {\em J. Analyse Math.}, 33:222--272, 1978.

\bibitem{veech:gauss}
William~A. Veech.
\newblock Gauss measures for transformations on the space of interval exchange
  maps.
\newblock {\em Ann. of Math. (2)}, 115(1):201--242, 1982.

\bibitem{vershik}
A.~M. Vershik.
\newblock A new model of the ergodic transformations.
\newblock In {\em Dynamical systems and ergodic theory ({W}arsaw, 1986)},
  volume~23 of {\em Banach Center Publ.}, pages 381--384. PWN, Warsaw, 1989.

\bibitem{viana:iet}
Marcelo Viana.
\newblock {\em {Dynamics of Interval Exchange Transformations and
  Teichm\"{u}ller Flows}}.
\newblock Lecture Notes. 2008.

\bibitem{zorich:intro}
Anton Zorich.
\newblock Flat surfaces.
\newblock In {\em Frontiers in number theory, physics, and geometry. {I}},
  pages 437--583. Springer, Berlin, 2006.

\end{thebibliography}

\end{document}